\documentclass[11pt,reqno]{amsart} 

\usepackage{graphicx}
\usepackage{xcolor}
\usepackage[labelfont=bf,font=sf]{caption}
\usepackage{amsmath, amsthm, amssymb}
\usepackage{tikz}
\usepackage{mathtools}
\usepackage{polynom}
\usepackage{enumerate}
\usepackage{microtype}

\usepackage{dsfont}
	\newcommand{\one}{\mathds{1}}
	
\usepackage{multicol}
\usepackage{boldline}
\usepackage{placeins}
\usepackage{afterpage}
\usepackage{xifthen}
\usepackage{framed}
\usepackage[font={sf,small},labelfont=md]{subfig}
\usepackage{hyperref, upref}
	\hypersetup{hypertexnames=false,colorlinks=true,linkcolor=blue,citecolor=blue,urlcolor=black}
	
\allowdisplaybreaks 

\numberwithin{equation}{section} 


\newcommand{\eq}[1]{\begin{align*} #1 \end{align*}}
\newcommand{\eeq}[1]{\begin{align} \begin{split} #1 \end{split} \end{align}}

\newcommand{\N}{\mathbb{N}}

\newcommand{\R}{\mathbb{R}}

\newcommand{\Z}{\mathbb{Z}}

\newcommand{\EE}{\mathbf{E}}

\newcommand{\PP}{\mathbf{P}}

\renewcommand{\a}{\mathcal{A}}
\renewcommand{\b}{\mathcal{B}}
\renewcommand{\c}{\mathcal{C}}

\newcommand{\f}{\mathcal{F}}
\newcommand{\g}{\mathcal{G}}

\renewcommand{\i}{\mathcal{I}}

\renewcommand{\k}{\mathcal{K}}
\renewcommand{\l}{\mathcal{L}}
\newcommand{\m}{\mathcal{M}}
\newcommand{\n}{\mathcal{N}}

\newcommand{\p}{\mathcal{P}}

\renewcommand{\r}{\mathcal{R}}
\newcommand{\s}{\mathcal{S}}
\renewcommand{\t}{\mathcal{T}}
\renewcommand{\u}{\mathcal{U}}
\renewcommand{\v}{\mathcal{V}}
\newcommand{\w}{\mathcal{W}}
\newcommand{\x}{\mathcal{X}}
\newcommand{\y}{\mathcal{Y}}

\newtheorem{thm}{Theorem}[section]
\newtheorem{prop}[thm]{Proposition}
\newtheorem{cor}[thm]{Corollary}
\newtheorem{lemma}[thm]{Lemma}

\newtheorem{theirthm}{Theorem}[section] 

\theoremstyle{definition}

\newtheorem*{remark}{Remark}

\def\eps{\varepsilon}
\def\vphi{\varphi}

\newcommand{\vc}[1]{{\boldsymbol #1}}

\newcommand{\wt}[1]{\widetilde{#1}}

\newcommand{\givenk}[3][]{#1[ #2 \: #1| \: #3 #1]} 

\newcommand{\dd}{\mathrm{d}}

\DeclareMathOperator{\WD}{WD}
\DeclareMathOperator{\SD}{SD}
\DeclareMathOperator{\SL}{SL}
\DeclareMathOperator{\WL}{WL}
\DeclareMathOperator{\VSD}{VSD}

\DeclareMathOperator{\diam}{diam}

\begin{document}

\bibliographystyle{acm}

\title{The endpoint distribution of directed polymers}

\subjclass[2010]{ 
60K37, 
82B26, 
82B44, 
82D60} 

\keywords{Directed polymer, free energy, disordered system, phase transition}

\author{Erik Bates}
\thanks{E.B.~was partially supported by NSF grant DGE-114747} 
\address{\newline Department of Mathematics \newline Stanford University \newline 450 Serra Mall, Bldg 380 \newline Stanford, CA 94305 \newline \textup{\tt ewbates@stanford.edu}}

\author{Sourav Chatterjee}
\thanks{S.C.~was partially supported by NSF grants DMS-1441513 and DMS-1608249}

\address{\newline Department of Statistics \newline Stanford University\newline Sequoia Hall, 390 Serra Mall \newline Stanford, CA 94305\newline \textup{\tt souravc@stanford.edu}}

\begin{abstract}
Probabilistic models of directed polymers in random environment have received considerable attention  in recent years. Much of this attention has focused on integrable models. In this paper, we introduce some new computational tools that do not require integrability. We begin by defining a new kind of abstract limit object, called ``partitioned subprobability measure'', to describe the limits of endpoint distributions of directed polymers. Inspired by a recent work of Mukherjee and Varadhan on large deviations of the occupation measure of Brownian motion,  we define a suitable topology on the space of partitioned subprobability measures and prove that this topology is compact. Then using a variant of the cavity method from the theory of spin glasses, we  show that any limit law of a sequence of endpoint distributions must satisfy a fixed point equation on this abstract space, and that the limiting free energy of the model can be expressed as the solution of a variational problem over the set of fixed points.  As a first application of the theory, we prove that in an environment with finite exponential moment, the endpoint distribution is asymptotically purely atomic if and only if the system is in the low temperature phase. The analogous result for a heavy-tailed environment was proved by Vargas in~2007.
As a second application, we prove a subsequential version of the longstanding conjecture that in the low temperature phase, the endpoint distribution is asymptotically localized in a region of stochastically bounded diameter. All  our results hold in arbitrary dimensions, and make no use of integrability. 
\end{abstract}
\maketitle

\tableofcontents

\section{Introduction}
The model of directed polymers in random environment was introduced in the physics literature by Huse and Henley \cite{huse-henley85} to represent the phase boundary of the Ising model in the presence of random impurities. It was later mathematically reformulated as a model of random walk in random potential by Imbrie and Spencer~\cite{imbrie-spencer88}. Over the last thirty years, the directed polymer model has played an important role as a source of many fascinating problems in the probability literature, culminating in the amazing recent developments in integrable polymer models. However, in spite of the wealth of information now available for integrable models, our knowledge about the general case is fairly limited, especially in spatial dimension greater than one. The goal of this paper is to introduce an abstract theory that allows computations for polymer models that are not integrable. 
Before we discuss our approach in detail, let us very briefly communicate the main consequences:

\begin{itemize}
\item The probabilistic model of $(d+1)$-dimensional directed polymers of length $n$ assigns a random probability measure to the set of random walk paths of length $n$ in $\Z^d$ that start at the origin. The precise mathematical model is defined in Section \ref{model} below. The ``endpoint distribution'' is the probability distribution of the final vertex reached by the random walk. Note that this is a random probability measure on $\Z^d$, which we will denote $f_n$. Understanding the  behavior of $f_n$ is the main goal of this article. Although a number of results about $f_n$ were known prior to this work (reviewed in  Section \ref{disorder_background}), this is the first paper that puts forward a comprehensive theoretical framework for analyzing the asymptotic properties of~$f_n$.
\item Like many statistical mechanical models, directed polymers have a low temperature phase and a high temperature phase. One of the most striking features of directed polymers is that in the low temperature phase, $f_n$ has ``atoms'' whose weights do not decay to zero as $n\to \infty$. This is in stark contrast with the endpoint distribution of simple random walk, where the most likely site has mass of order $n^{-d/2}$. The precise statement of this well-known localization phenomenon will be discussed in Section~\ref{disorder_background}. One of the main applications of the abstract machinery developed in this paper is to show that in the low temperature regime, the atoms account for {\it all} of the mass --- that is, there is no part of the mass that diffuses out. Such a result was proved earlier for directed polymers in heavy-tailed environment, and also for a particular $(1+1)$-dimensional integrable model. 
We prove it under finite exponential moments, to which previously known techniques do not apply.

\item Our second main application is to show that in the low temperature regime, there is almost surely a subsequence of positive density along which the endpoint distribution concentrates mass $> 1-\delta$ on a set of diameter $\leq K(\delta)$, where $\delta$ is arbitrary and $K(\delta)$ is a deterministic constant that depends only on $\delta$ and some features of the model. In other words, not only does the mass localize on atoms, but the atoms themselves localize in a set of bounded diameter. This proves a subsequential version of a longstanding conjecture about the endpoint distribution. Prior to this work, the only case where a similar statement could be proved was for an integrable model.

\item Crucial to each of these results is a new variational formula for the limiting free energy, given as the infimum of a continuous functional over a certain closed subset $\k$ of a compact space.
The high and low temperature phases are then characterized respectively by whether $\k$ is just a single trivial object or instead contains non-trivial elements.

\end{itemize}
We will now begin a more detailed presentation by defining the model below. This is followed by a discussion of the known results about polymer models, and then a general overview of the results proved in this paper and the ideas involved in the proofs. 

\subsection{The model of directed polymers in random environment} \label{model}
Take any integer $d\ge 1$. The probabilistic model of $(d+1)$-dimensional \textit{directed polymers in random environment} is defined as follows. We begin with a simple random walk $\omega = (\omega_i)_{i \geq 0}$ on $\Z^d$, letting $P$ denote the law of the walk when started at the origin.
Let $E$ denote expectation according to $P$. 

Next let $\N \coloneqq  \{1,2,\dots\}$ and introduce a collection of i.i.d.~random variables $(X_u)_{u \in \N \times \Z^d}$ called the \textit{random environment}, defined on some probability space $(\Omega_{\mathrm{e}},\f,\PP)$.
We will write $\EE$ for expectation according to $\PP$. In what follows, it will not be problematic to define all random variables on the abstract probability space $(\Omega_{\mathrm{e}},\f,\PP)$.
Unless stated otherwise, ``almost sure" statements are made with respect to $\PP$. 

Let $\beta > 0$ be a parameter, called the \textit{inverse temperature}.
Let $\mathfrak{L}$ denote the common law of the $X_u$, often called the \textit{disorder distribution}. 
We will assume the logarithmic moment generating function for $\mathfrak{L}$ satisfies
\eeq{
\lambda(\alpha) \coloneqq  \log\EE(e^{\alpha X_u}) < \infty \quad \text{for all $\alpha \in [-2\beta,2\beta]$}. \label{mgf}
}
We now have the notation to define the model of $(d+1)$-dimensional directed polymers in random environment.
The \textit{quenched polymer measure} of length $n \geq 0$, denoted $\rho_n$, is the Gibbs measure for $P$ with Hamiltonian
\eq{
H_n(\omega) \coloneqq  -\sum_{i = 1}^n X_{i,\, \omega_i}.
}
That is,
\eeq{
\rho_n(\dd\omega) = \frac{1}{Z_n}e^{-\beta H_n(\omega)}\ P(\dd\omega), \label{rho_def}
}
where the normalization constant $Z_n \coloneqq  E(e^{-\beta H_n(\omega)})$
is called the \textit{quenched partition function}.
Explicitly,
\eq{
Z_n = \frac{1}{(2d)^n} \sum_{\gamma} \exp\bigg(\beta\sum_{i = 1}^n X_{i,\, \gamma(i)}\bigg), 
}
where the sum is over the $(2d)^n$ nearest-neighbor paths $\gamma : \{0,1,\dots,n\} \to \Z^d$ of length $|\gamma| = n$, starting at the origin ($\gamma(0) = 0$).
We can equivalently write
\eq{
Z_n = \frac{1}{(2d)^n} \sum_{x \in \Z^d} Z_n(x),
}
where
\eq{
Z_n(x) \coloneqq  \sum_{\gamma\, :\, \gamma(n) = x} \exp\bigg(\beta\sum_{i = 1}^n X_{i,\, \gamma(i)}\bigg) 
}
is often referred to as a \textit{point-to-point partition function}.
When $d = 1$, convention often replaces $Z_n$ by $(2d)^n Z_n$, which is correspondingly called the \textit{point-to-line partition function}.
So that the model does not reduce to a simple random walk (which is the case when $\beta = 0$), we assume the $X_u$ are non-degenerate random variables.
That is, $\mathfrak{L}$ is not supported on a single point.


\subsection{An overview of known results about general polymer models} \label{disorder_background}
In this section we will review the results that are known in arbitrary dimensions, with usually mild assumptions on the distribution of the environment.
Indeed, our paper works in this general setting, and so these results will be most relevant to the present study.
One can find a more detailed review in the recent notes by Comets \cite{comets17}.

\subsubsection{High and low temperature phases}
The qualitative behavior of directed polymers depends on the disorder distribution $\mathfrak{L}$, the inverse temperature $\beta$, and the traversal dimension $d$.
Much of this dependence can be observed through the system's \textit{quenched free energy}, 
given by
\eq{
F_n \coloneqq  \frac{\log Z_n}{n},
}
where $F_0 \coloneqq  0$.
Like $(\rho_n)_{n \geq 0}$ and $(Z_n)_{n \geq 0}$, the sequence $(F_n)_{n \geq 0}$ is a random process with respect to the filtration $(\f_n)_{n \geq 0}$, where
\eeq{
\f_n \coloneqq  \sigma(X_{i,\,x} : 1 \leq i \leq n,\, x \in \Z^d).\label{gndef}
}
When the randomness of the environment is averaged out, one obtains the \textit{averaged quenched free energy},
\eq{
\EE(F_n) = \frac{1}{n}\, \EE \log Z_n.
}
In general, this quantity is distinct from the \textit{annealed free energy}, which is
\eq{
\frac{1}{n} \log \EE(Z_n) = \frac{1}{n}\log e^{n\lambda(\beta)} = \lambda(\beta).
}
Indeed, Jensen's inequality 
gives the comparison
\eeq{ \label{jensen_applied_finite}
\EE(F_n) < \lambda(\beta),
}
where the inequality is strict because $F_n$ is not an almost sure constant, and $x\mapsto\log x$ is not linear.
A superadditivity argument  (e.g.~see \cite{carmona-hu02}, proof of Proposition 1.4) 
shows that $\EE(F_n)$ converges to $\sup_{n \geq 0} \EE(F_n)$ as $n\to\infty$.
It has also been shown in \cite{vargas07} that under hypotheses much weaker than \eqref{mgf}, $F_n-\EE(F_n)$ tends to $0$ almost surely.
Therefore, there is a deterministic limit
\eeq{
p(\beta) \coloneqq  \lim_{n \to \infty} \EE(F_n) = \lim_{n\to\infty} F_n \quad \mathrm{a.s.} \label{Fn_lim}
}
Using the FKG inequality, Comets and Yoshida~\cite{comets-yoshida06} identified a phase transition:

\begin{theirthm}[\cite{comets-yoshida06}, Theorem 3.2]\label{critical_temperature}
There exists a critical inverse temperature $\beta_{\mathrm{c}} = \beta_{\mathrm{c}}(\mathfrak{L},d) \in [0,\infty]$ such that
\eq{
0 \le \beta \leq \beta_{\mathrm{c}} \quad &\implies \quad p(\beta) = \lambda(\beta) \\
\beta > \beta_{\mathrm{c}} \quad &\implies \quad p(\beta) < \lambda(\beta).
}
\end{theirthm}

We refer to the region $0 \le \beta < \beta_{\mathrm{c}}$ as the ``high temperature phase'', while $\beta > \beta_{\mathrm{c}}$ defines the ``low temperature phase".
Roughly speaking, high temperatures reduce the influence of the random environment, and so polymer growth resembles a simple random walk, while at low temperatures the random impurities force a much different behavior.
This distinction has been most frequently made in terms of the \textit{endpoint distribution} $\rho_n(\omega_n \in \cdot) $, which is a random probability measure on $\Z^d$.
For instance, one striking result first proved by Carmona and Hu \cite{carmona-hu02} (for a Gaussian environment) 
and then by Comets, Shiga, and Yoshida \cite{comets-shiga-yoshida03} (in the general case) 
is the following: 
If $\beta > \beta_{\mathrm{c}}$, then the polymer endpoint observes so-called \textit{strong localization}:
\eq{
(\SL): \quad \exists\ c > 0, \quad \limsup_{n \to \infty} \max_{x \in \Z^d} \rho_{n}(\omega_n = x) \geq c \quad \text{a.s.}
}
That is, infinitely often the polymer has ``favorite sites" at which its endpoint distribution concentrates.

\subsubsection{Weak and strong disorder regimes}
While examining the high and low temperature regimes is very natural, the mathematical development of directed polymers in random environment has often followed an ostensibly different route.
Since the work of Bolthausen \cite{bolthausen89}, analysis of the directed polymer model has frequently focused on the \textit{normalized partition function},
\eq{
\wt{Z}_n \coloneqq  Z_n\, e^{-n\lambda(\beta)}, 
}
with $\wt{Z}_0 = Z_0 = 1$ and $\lambda(\beta)$ defined as in \eqref{mgf}.
It is not difficult to check that $\wt{Z}_n$ is a positive martingale adapted to the filtration $(\f_n)_{n \geq 0}$ defined in \eqref{gndef}.
In particular, $\EE(\wt{Z}_n) = 1$ for all $n$, and the martingale convergence theorem 
implies that there is an $\f$-measurable random variable $\wt{Z}_\infty$ such that
\eeq{
\lim_{n \to \infty} \wt{Z}_n = \wt{Z}_\infty \quad \text{a.s.} \label{Ztilde_def}
}
Furthermore, we necessarily have $\wt{Z}_\infty \geq 0$ almost surely, and the event of positivity
$\{\wt{Z}_\infty > 0\}$ is measurable with respect to the tail $\sigma$-algebra,
\eq{
\bigcap_{n = 1}^\infty \sigma(X_{i,\, x} : i \geq n,\, x\in \Z^d).
}
By Kolmogorov's zero-one law 
we have either \textit{weak} or \textit{strong disorder},
\eq{
(\WD)&: \quad \wt{Z}_\infty > 0 \quad \text{a.s.},
\qquad
(\SD): \quad \wt{Z}_\infty = 0 \quad \text{a.s.}
}
Carmona and Hu \cite{carmona-hu02} 
and Comets, Shiga, and Yoshida \cite{comets-shiga-yoshida03} 
gave the following characterization of the two phases.
It says that there is strong disorder exactly when the overlap of two independent polymers has infinite expectation.

\begin{theirthm}[\cite{carmona-hu02}, Proposition 5.1 and \cite{comets-shiga-yoshida03}, Theorem 2.1]\label{disorder_equiv} 
\label{sd_equals_wl}
Strong disorder $(\SD)$ is equivalent to \textit{weak localization},
\eq{
(\WL):\quad \sum_{n = 0}^\infty \rho_n^{\otimes 2}(\omega_n = \omega_n') = \infty \quad \mathrm{a.s.},
}
where $\omega$ and $\omega'$ are independent samples from $\rho_n$.
\end{theirthm}

Despite the unfortunate clash of terminology in the above theorem, one should not mistake weak localization to be associated with weak disorder.
The use of ``weak" in the first case is only to distinguish this notion of localization from the other (SL) notion defined before.
Indeed, it is apparent from the inequality
\eq{
\rho_n^{\otimes 2}(\omega_n = \omega_n') \geq \Big(\max_{x \in \Z^d} \rho_n(\omega_n = x)\Big)^2
}
that $(\SL) \implies (\WL)$.
In fact, Carmona and Hu \cite{carmona-hu06} proved in a continuous-time model with an environment of i.i.d.~Brownian motions --- the parabolic Anderson model --- that the converse is also true: $(\SL) \impliedby (\WL)$. 
It is believed that the two notions are equivalent in general.

As with the high and low temperature regimes, there is a phase transition between weak and strong disorder. 

\begin{theirthm}[\cite{comets-yoshida06}, Theorem 12.1]
There exists a critical inverse temperature $\wt{\beta}_{\mathrm{c}} = \wt{\beta}_{\mathrm{c}}(\mathfrak{L},d) \in [0,\infty]$ such that
\eq{
0 \le \beta < \wt{\beta}_{\mathrm{c}} \quad &\implies \quad (\WD) \\
\beta > \wt{\beta}_{\mathrm{c}} \quad &\implies \quad (\SD).
}
\end{theirthm}

Determining the behavior of $\wt{Z}_\infty$ at $\wt{\beta}_{\mathrm{c}}$ is an open problem. 
For analogous models on $b$-ary trees, it is known from work of Kahane and Peyri\`ere \cite{kahane-peyriere76} that strong disorder occurs at $\wt{\beta}_{\mathrm{c}}$.
Interestingly, in dimensions $d = 1$ and $d = 2$, there is strong disorder at all finite temperatures (i.e.~$\wt{\beta}_{\mathrm{c}} = 0$), while in higher dimensions weak disorder occurs at sufficiently high temperatures ($\wt{\beta}_{\mathrm{c}} > 0$).
Precise conditions on $\lambda(\beta)$ guaranteeing either behavior can be found in Theorem 2.3.2 of the review \cite{comets-shiga-yoshida04}.
These conditions are in fact a culmination of results from \cite{imbrie-spencer88,bolthausen89,sinai95,albeverio-zhou96,song-zhou96,kifer97,carmona-hu02,comets-shiga-yoshida03}.
Whereas Theorem \ref{disorder_equiv} characterizes the disorder regimes in terms of endpoint localization, one can also attempt to give a characterization based on endpoint diffusion. 
This is usually described by some exponent $\xi = \xi(\mathfrak{L},\beta,d)$ such that ``typical" polymer endpoints are distance $O(n^\xi)$ from the origin.
One way to make this precise is to insist that $\xi$ satisfy
\eq{
\lim_{C\to\infty}\liminf_{n\to\infty} \PP\bigg(C^{-1}n^{\xi} \leq \int \|\omega_n\|_2\ \rho_n(\dd\omega) \leq Cn^{\xi}\bigg) = 1.
}
For instance, Comets and Yoshida  \cite{comets-yoshida06} showed, as part of a more general Brownian central limit theorem, that $\xi = 1/2$ in weak disorder.
That is, the location of the polymer endpoint asymptotically matches that of simple random walk on $\Z^d$.
On the other hand, it is believed (at least in low dimensions, see \cite{piza97}) that the polymer endpoint is superdiffusive in strong disorder.
Specifically, when $d = 1$ (for which any $\beta>0$ yields strong disorder), it is conjectured that $\xi = 2/3$.
Apart from the integrable models discussed in Section \ref{solvability_background}, the only result in this direction is an upper bound due to Piza \cite{piza97} and is conditional on a curvature assumption.
In some related models on $\R$, both upper and lower bounds are known: $3/5\leq \xi \leq 3/4$ \cite{wuthrich98,petermann00,mejane04,comets-yoshida05,comets-yoshida04,bezerra-tindel-viens08}.


A major challenge is to reconcile the notion of high and low temperature with the notion of weak and strong disorder.
The first is more natural from the perspective of statistical physics, 
because phase transitions, in the standard sense, are defined in terms of non-analyticities of limiting free energies. Since the annealed free energy is typically analytic everywhere, the transition of the 
difference $\lambda(\beta)-p(\beta)$ from zero to nonzero denotes the first point of non-analyticity of the limiting free energy $p(\beta)$.
On the other hand, the more abstract conditions of $(\WD)$ and $(\SD)$ have lead to a wealth of mathematical results.
It is conjectured (see \cite{comets-yoshida06,carmona-hu06}) that $\beta_{\mathrm{c}} = \wt{\beta}_{\mathrm{c}}$, which would mean
\eq{
0 \leq \beta < \beta_{\mathrm{c}} \quad \implies \quad (\WD) \qquad \text{and} \qquad
\beta > \beta_{\mathrm{c}} \quad \implies \quad (\SD).
}
Evidence for this belief includes the result by Comets and Vargas  \cite{comets-vargas06} that $\beta_{\mathrm{c}} = 0$ universally in $d = 1$, and the subsequent proof by Lacoin  \cite{lacoin10} that $\beta_{\mathrm{c}} = 0$ in $d = 2$.
For $d \geq 3$, only the trivial second implication above is known (i.e.~$\beta_{\mathrm{c}} \geq \wt{\beta}_{\mathrm{c}}$).
Indeed, it is clear from \eqref{Fn_lim} and \eqref{Ztilde_def} that $p(\beta)<\lambda(\beta) \implies (\SD)$.
Therefore, it has become common to say that in the low temperature phase, there is \textit{very strong disorder},
\eq{
(\VSD): \quad p(\beta) < \lambda(\beta).
}
%

In recent work, Rassoul-Agha, Sepp\"al\"ainen and Yilmaz \cite{rassoul-seppalainen-yilmaz13,rassoul-seppalainen-yilmaz17} 
expressed $p(\beta)$ in terms of several variational formulas. 
A method for finding minimizers to these formulas was proposed in \cite{georgiou-rassoul-seppalainen16}, and it was noted in \cite{rassoul-seppalainen-yilmaz17} 
that they may not admit minimizers when $\beta > \beta_{\mathrm{c}}$.
These results suggest the possibility of a more general correspondence between the disorder regimes and the existence of variational minimizers (see \cite{rassoul-seppalainen-yilmaz17}, Conjecture 2.13).

When $d = 1$, a so-called ``intermediate" disorder regime was discovered in \cite{alberts-khanin-quastel14I} by scaling the inverse temperature to $0$ as $n\to\infty$, in such a way that features of both weak and strong disorder are observed in the limit. 
More specifically, when $\beta_n = \beta n^{-1/4}$, the polymer measure exhibits diffusivity ($\xi=1/2$), and the fluctuations of $\log Z_n$ are order $1$ as in weak disorder.
The actual fluctuations, however, are not Gaussian, but rather depend on $\beta$ and the random environment, interpolating between Gaussian at $\beta = 0$ and Tracy--Widom at $\beta = \infty$.

\subsubsection{Further versions of localization}

Another central task in the theory of directed polymers is determining if and when weak localization results imply stronger ones.
The following analog to Theorem \ref{disorder_equiv} shows that $(\VSD) \implies (\SL)$, which means the strongest type of localization occurs throughout the low temperature regime.
Therefore, if $\beta_{\mathrm{c}} = \wt{\beta}_{\mathrm{c}}$, then all notions of disorder and localization are equivalent, except possibly at the critical temperature.

\begin{theirthm}[\cite{comets-shiga-yoshida03}, Corollary 2.2 and Theorem 2.3(a)]\label{characterization0} 
There is very strong disorder $(\VSD)$ if and only if there exists $c > 0$ such that
\eq{
\liminf_{n \to \infty} \frac{1}{n} \sum_{i = 0}^{n-1} \rho_{i}^{\otimes 2}(\omega_i = \omega_i') \geq c \quad \mathrm{a.s.,}
}
or, equivalently, there exists $c > 0$ such that
\eeq{
\liminf_{n \to \infty} \frac{1}{n} \sum_{i = 0}^{n-1}  \max_{x \in \Z^d} \rho_{i}(\omega_i = x) \geq c \quad \mathrm{a.s.} \label{similar_apa}
}
\end{theirthm}

Strong localization $(\SL)$ or the stronger property \eqref{similar_apa} captures the tendency of the endpoint distribution to localize mass when $\beta > \beta_{\mathrm{c}}$. 
In this scenario, it is natural to ask whether the entire mass localizes, or if some positive proportion of the mass remains delocalized.  In \cite{vargas07}, Vargas proposed the following definition for complete localization, or as Vargas called it, \textit{asymptotic pure atomicity}.
We consider the set of ``$\eps$-atoms",
\eq{
\a_i^\eps \coloneqq  \{x \in \Z^d : \rho_{i}(\omega_i = x) > \eps\}, \quad i \geq 0,\ \eps > 0,
} 
and say that the sequence $(\rho_{i}(\omega_i \in \cdot))_{i \geq 0}$ is \textit{asymptotically purely atomic} if for every sequence $(\eps_i)_{i \geq 0}$ tending to 0 as $i \to \infty$, we have
\eq{
\frac{1}{n} \sum_{i = 0}^{n-1} \rho_{i}(\omega_i \in \a_i^{\eps_i}) \to 1 \quad \text{in probability, as $n \to \infty$.}
}
The following result was obtained in \cite{vargas07}. 

\begin{theirthm}[\cite{vargas07}, Theorem 3.2 and Corollary 3.3] \label{vargas_apa}
If $\lambda(\beta) = \infty$, then $(\rho_{i}(\omega_i \in \cdot))_{i \geq 0}$ is asymptotically purely atomic.
\end{theirthm}

Actually, Vargas showed that $(\rho_{i-1}(\omega_i \in \cdot))_{i \geq 1}$ is asymptotically purely atomic, with the set $\a_i^\eps$ replaced by $\{x \in \Z^d : \rho_{i-1}(\omega_i = x) > \eps\}$.  
From observation \eqref{next_step}, it is not difficult to check that the two notions are equivalent (a proof can be found in Appendix \ref{equivalent_notions_apa}).
One of the main results of this paper, Theorem \ref{total_mass}, asserts that the conclusion of Theorem \ref{vargas_apa} continues to hold even if $\lambda(\beta)$ is finite, as long as $\beta > \beta_{\mathrm{c}}$ and~\eqref{mgf} holds. This is discussed in greater detail in Section \ref{results} below.

\subsubsection{A remark about the definition of endpoint distribution}
All the results of Section \ref{disorder_background} are normally stated in the literature using the measure $\rho_{n-1}(\omega_n \in \cdot)$, as opposed to $\rho_n(\omega_n \in \cdot)$. 
The reason is that the former arises naturally out of the martingale analysis for $\wt{Z}_n$ (for instance, see the proof of Theorem~3.3.1 in \cite{comets-shiga-yoshida04}), but in all cases, the statements are equivalent when the latter is used instead.
This equivalence can be seen by writing
\eeq{
\rho_{n-1}(\omega_n = x) = \frac{1}{2d} \sum_{\|x - y\|_1 = 1} \rho_{n-1}(\omega_{n-1} = y), \label{next_step}
}
(where $\|\cdot\|_1$ denotes the $\ell^1$ norm on $\Z^d$), which implies
\eq{
\frac{1}{2d}\, \max_{y \in \Z^d} \rho_{n-1}(\omega_{n-1} = y) \leq \max_{x \in \Z^d} \rho_{n-1}(\omega_n = x) \leq \max_{y \in \Z^d} \rho_{n-1}(\omega_{n-1} = y).
}
Similarly, we have
\eq{
\frac{1}{2d}\, \rho^{\otimes 2}_{n-1}(\omega_{n-1} = \omega_{n-1}') \leq \rho^{\otimes 2}_{n-1}(\omega_{n} = \omega_{n}') \leq \rho^{\otimes 2}_{n-1}(\omega_{n-1} = \omega_{n-1}'),
}
where the first inequality follows from the observation that if $\omega_{n-1} = \omega_{n-1}'$, then $\omega_n = \omega_n'$ with probability $(2d)^{-1}$ under $\rho_{n-1}$.
The second inequality is due to Cauchy--Schwarz applied to \eqref{next_step}:
\eq{
\rho^{\otimes 2}_{n-1}(\omega_n = \omega_n') 
&= \sum_{x \in \Z^d} \rho_{n-1}(\omega_n = x)^2\\
&\leq \sum_{x \in \Z^d} \frac{1}{2d} \sum_{\|x - y\|_1 = 1} \rho_{n-1}(\omega_{n-1} = y)^2 \\
&= \sum_{y \in \Z^d} \rho_{n-1}(\omega_{n-1} = y)^2
= \rho^{\otimes 2}_{n-1}(\omega_{n-1}=\omega_{n-1}').
}
As $\rho_n(\omega_n \in \cdot)$ will be the more natural object for our purposes, we reserve the term ``endpoint distribution" for this measure.

\subsection{Related results in integrable models} \label{solvability_background}
Until recently, no directed polymer model offered the possibility of exact asymptotic calculations.
This was in contrast to the integrable last passage percolation (LPP) models in $1+1$ dimensions, for which specific choices of the passage time distribution (namely geometric or exponential) allow one to use representation theory to derive exact formulas for the distribution of passage times, which are the LPP analogs of partition functions.
Such was the approach advanced in the seminal work of Johansson \cite{johansson00}, which showed that asymptotic fluctuations of passage times follow the Tracy--Widom distributions at scale $n^{1/3}$ --- GUE for point-to-point passage times, and GOE for the point-to-line passage time.
The connection to Tracy--Widom laws can be seen more explicitly in a model known as Brownian LPP \cite{oconnell03}, which also exhibits the $n^{1/3}$ scaling.
These results place the LPP models within the KPZ universality class (see \cite{corwin12, borodin-petrov14}). 
For a review of related works, including results verifying spatial fluctuations at scale $n^{2/3}$, we refer the reader to \cite{quastel-remenik14} and references therein.

In the zero-temperature limit $\beta \to \infty$, the polymer measure $\rho_n$ concentrates on the path that is most likely given the random environment.
When $\beta = \infty$, the directed polymer model is equivalent to LPP.
The natural conjecture, therefore, is that directed polymers in strong disorder obey the same KPZ scaling relations.
In particular, models in $1+1$ dimensions should have energy fluctuations of order $n^{1/3}$, and endpoint fluctuations of order $n^{2/3}$. 
The first case permitting exact calculations was the integrable log-gamma model introduced by Sepp{\"a}l{\"a}inen \cite{seppalainen12}, whose breakthrough work proved the conjectured exponents.
Subsequent studies showed that free energy fluctuations were again of Tracy--Widom type \cite{corwin-oconnell-seppalainen-zygouras14,borodin-corwin-remenik13}, gave a large deviation rate function \cite{georgiou-seppalainen13}, and computed the limiting value $p(\beta)$ \cite{georgiou-rassoul-seppalainen-yilmaz15}.

Regarding spatial fluctuations, Comets and Nguyen \cite{comets-nguyen16} found an explicit limiting endpoint distribution in the equilibrium case of the log-gamma model.
More precisely, they showed that the endpoint distribution, when re-centered around the most likely site, converges in law to a certain random distribution on $\Z$. 
While the boundary conditions they enforce prevent the correct exponent of $2/3$ from appearing, which would mean fluctuations of the averaged quenched distribution occur on the order of $n^{2/3}$, their result implies that the \textit{quenched} endpoint distributions concentrate all their mass in a microscopic region of $O(1)$ diameter around a single favorite site.
In the notation of this paper, if $x_n = \arg \max \rho_n(\omega_n = \cdot)$, then the result of \cite{comets-nguyen16} says
\eeq{
\lim_{K \to \infty} \liminf_{n \to \infty} \rho_n(\|\omega_n - x_n\|_1 \leq K) = 1 \quad \text{in probability.}\label{mode}
}
This provides an affirmative case of the so-called ``favorite region" conjecture, which speculates that in strong disorder, directed polymers on the lattice localize their endpoint to a region of fixed diameter. 
We emphasize once more that this is a statement about quenched endpoint distributions.
While little is known on annealed distributions in the general case, 
there has been considerable progress for the integrable LPP models. 
Johansson \cite{johansson03} proved that the rescaled location of the endpoint converges in distribution to the maximal point of an Airy$_2$ process minus a parabola.
More recent work has expressed the density of this limit in terms of Fredholm determinants \cite{flores-quastel-remenik13}, or from a different approach, in terms of the Hastings-McLeod solution to the Painlev\'{e} II equation \cite{schehr12}; the two formulas were shown to be equivalent in \cite{baik-liechty-schehr12}.
Tail estimates for this density can be found in \cite{bothner-liechty13,quastel-remenik15}.

For further literature on integrable directed polymers, see \cite{corwin-seppalainen-shen15,barraquand-corwin17,thiery-doussal15} and references therein.
We also mention the semi-discrete model introduced by O'Connell and Yor \cite{oconnell-yor01}, which is a positive-temperature version of Brownian LPP and has offered another provable case of $n^{1/3}$ energy fluctuations \cite{seppalainen-valko10,borodin-corwin14,borodin-corwin-ferrari14}. 


\subsection{An overview of results proved in this paper} \label{results}
The endpoint distribution $\rho_n(\omega_n \in \cdot)$ of length-$n$ polymers is a random probability measure on $\mathbb{Z}^d$. The main goal of this paper is to understand the behavior of this object as $n\to\infty$.
The majority of our work is in 
developing methods
to compute limits of endpoint distributions, or more to the point, of \textit{functionals} of the endpoint distribution.
While proving the existence of these limits is presently out of reach, by instead considering empirical averages we are able to establish results in Ces\`aro form.
Broadly speaking, the construction consists of three components: (i) a compactification of measures on $\Z^d$; (ii) a Markov kernel whose invariant measures are the possible limits of the endpoint distribution; and (iii) a functional that distinguishes those invariant measures with minimal energy.

\subsubsection{Compactification of measures on $\Z^d$}
One key obstacle to studying the endpoint distribution is that the standard topology of weak convergence of probability measures is inadequate for understanding its limiting behavior. Even the weaker topology of vague convergence does not provide an adequate description of the localization phenomenon that is of central interest.
In the recent work of Mukherjee and Varadhan \cite{mukherjee-varadhan16}, this issue was addressed by exploiting translation invariance inherent in their problem to pass to a certain compactification of probability measures on $\R^d$.
We have drawn inspiration from their methods, using a similar construction we will soon describe.
Ultimately, this approach is enabled by the ``concentration compactness" phenomenon.
Before we discuss this topic further, let us motivate the discussion with two elementary examples.

Roughly speaking, the difficulty of using standard weak or vague convergence is that the endpoint distribution may manifest itself as multiple localized ``blobs'' that escape to infinity in different directions as $n\to \infty$. Additionally, some part of the mass may just ``diffuse to zero''. As a first example, consider the following sequence of probability mass functions on $\mathbb{Z}$:
\[
q_n(x) \coloneqq 
\begin{cases}
1/2 &\text{if $x=n$,}\\
1/(2n) &\text{if $0\le x< n$,}\\
0 &\text{else.}
\end{cases}
\]
This sequence fails to converge weakly, and its vague limit is the zero measure. 
However, to understand the true limiting behavior of $q_n$, it seems more appropriate to take the vague limit after translating the measure by $n$; that is, taking $\wt{q}_n(x) \coloneqq  q_n(x-n)$. 
The limit of the sequence $\wt{q}_n$ is the subprobability measure that puts mass $1/2$ at the origin and zero elsewhere, which gives a better picture of the true limiting behavior of $q_n$. 

The situation is more complicated when multiple blobs escape to infinity in different directions. For example, consider the following sequence of probability mass functions on $\mathbb{Z}$:
\[
r_n(x) \coloneqq 
\begin{cases}
1/5 &\text{if $x \in \{-2n, -n, n, n+1\}$,}\\
1/(5n) &\text{if $0\le x<n$,}\\
0 &\text{else.}
\end{cases}
\]
Like $q_n$, the sequence $r_n$ also converges vaguely 
to the subprobability measure of mass zero. But in this case, no sequence of translates of $r_n$ can fully capture its limiting behavior, because a translate can only capture the mass at one of the three blobs but not all of them simultaneously.  The only recourse, it seems, is to express the limit not as a single subprobability measure, but a union of three subprobability measures on three distinct copies of $\mathbb{Z}$. Two of these will put mass $1/5$ at the origin in their respective copies of $\Z$, and the third will put mass $1/5$ at $0$ and $1/5$ at $1$ in its copy of $\Z$. 
The first two are the vague 
limits of $r_n(\cdot +2n)$ and $r_n(\cdot + n)$, and the third is the vague 
limit of $r_n(\cdot - n)$. One can view this object jointly as a subprobability measure of total mass $4/5$ on the set $\{1,2,3\}\times \Z$. Of course, it is not important which copy of $\Z$ gets which part of the measure, nor does it matter if a translation is applied to the subprobability measure on one of the copies. Thus, the limit object is an equivalence class of subprobability measures on $\{1,2,3\}\times \Z$ rather than a single subprobability measure.

Generalizing the above idea, we can consider equivalence classes of subprobability measures on $\N \times \Z^d$. First, define the $\N$-support of a subprobability measure $f$ on $\N\times \Z^d$ to be the set of all $n\in \N$ such that $f$ puts nonzero mass on $\{n\}\times \Z^d$. Next, declare two subprobability measures $f$ and $g$ on $\N \times \Z^d$ to be equivalent if there is a bijection $\sigma$ between their $\N$-supports and a number $x_n\in \Z^d$ for each $n$ in the $\N$-support of $f$ such that 
\[
g(\sigma(n), \cdot) = f(n, x_n +\cdot).
\]
It is easy to verify that this is indeed an equivalence relation. The equivalence classes are called \textit{partitioned subprobability measures} on $\Z^d$ in this paper, and the set of all equivalence classes is denoted by $\s$. The set $\s$ is formally defined and studied in Section \ref{topology-definitions}.

We will view probability measures on $\Z^d$ (or more precisely, their mass functions) as elements of $\s$ supported on a single copy of $\Z^d$.
In particular, we have a natural embedding of endpoint distributions into $\s$.
In Section \ref{topology-definitions}, we define a metric $d$ on $\s$ (not to be confused with the dimension $d$).
Due to the somewhat complicated nature of the metric, we will not reproduce the definition in this overview section, but it is constructed to ensure that certain functionals on $\s$ are continuous or at least semi-continuous.
In this study, these functionals are related to either localization or the free energy of the system.
One of the first nontrivial results of this paper is that $(\s, d)$ is a compact metric space, in analogy with the construction in \cite{mukherjee-varadhan16}. 
This is Theorem~\ref{compactness_result} in Section~\ref{compactness}.

\begin{remark}
A formal comparison between our construction and the one of Mukherjee and Varadhan \cite{mukherjee-varadhan16} is provided in Appendix \ref{compare_topologies}.
\end{remark} 

In summary, the metric $d$ imposes a translation invariance that prevents so-called blobs from escaping to infinity.
Moreover, such escape turns out to be the only obstruction to compactness.
This is the fundamental observation of a more general theory known as ``concentration compactness", which builds on the notion of the concentration functions of probability measures, defined by L\'evy~\cite{levy54}.  The idea of using L\'evy's concentration functions for constructing compactifications of function spaces was introduced by Parthasarathy, Ranga Rao, and Varadhan~\cite{parthasarathy-rangarao-varadhan62}, and  developed into a powerful tool by Lions \cite{lions84I,lions84II,lions85I,lions85II}. It has been applied broadly in the study of calculus of variations and PDEs. 
In the language of those settings, the equivalence relation described above plays the role of a ``profile decomposition" (for a friendly discussion, see \cite{tao08,tao10}).

\subsubsection{The update map}
In Section \ref{transformation} we 
define
an ``update'' map $T:\s \to \p(\s)$,  where $\p(\s)$ is the set of all probability measures on $\s$ equipped with the Kantorovich--Rubinstein--Wasserstein  distance  $\w$. 
The map $T$ has the property that if $f_n(\cdot) \coloneqq  \rho_n(\omega_n = \cdot)$ is the endpoint mass function of the length-$n$ polymers (considered as an element of $\s$ supported on a single copy of $\Z^d$), then $Tf_n$ is the law of $f_{n+1}$ given $\f_n$ (recall the definition \eqref{gndef} of $\f_n$).
In fact, it is not difficult to see that $(f_n)_{n\geq0}$ is Markovian, and thus $T$ is the Markov kernel generating this chain.
Considerable work is done in Section \ref{proof_continuity} to show that $T$ is a continuous map (i.e.~has the Feller property).  
This is the conclusion of Proposition \ref{continuous1}. The explicit nature of the metric $d$ is particularly important in the proof of this result.
Finally, the map $T$ lifts to a map $\t: \p(\s) \to \p(\s)$, defined as
\eq{
\t \nu(\dd g) \coloneqq  \int Tf(\dd g)\ \nu(\dd f).
}
The continuity of $T$ implies that $\t$ is also continuous.

In Section \ref{free_energy} we study the following random element of $\p(\s)$:
\eq{
\mu_n \coloneqq  \frac{1}{n}\sum_{i = 0}^{n-1} \delta_{f_i}.
}
Here $\delta_{f_i}$ is the unit point mass at the $i^{\text{th}}$ endpoint mass function $f_i$, considered as an element of $\s$ as before. In words, $\mu_n$ is the empirical measure of the endpoint distributions up to time $n$. Let 
\[
\k \coloneqq  \{\nu \in \p(\s) : \t\nu = \nu\}
\]
be the set of fixed points of $\t : \p(\s) \to \p(\s)$.
The first main result of Section \ref{free_energy} is Corollary \ref{close_probability}, which says that
\eeq{
\lim_{n \to \infty} \inf_{\nu \in \k} \w(\mu_n,\nu) = 0 \quad \mathrm{a.s.} \label{going_to_K}
}
This result provides a heuristic connection to $(1+1)$-dimensional integrable models (for instance, see \cite{seppalainen12,corwin-seppalainen-shen15,barraquand-corwin17,thiery-doussal15}), which work in part by identifying a disorder distribution and boundary conditions such that the system is stationary under spatial translations.
Loosely speaking, our approach similarly recovers a stationarity property \textit{in the limit}, even without explicit calculations.
In this way, our methods replace this key feature of integrable models with a much weaker, but more general, abstract framework.
This is enabled by a topology on endpoint distributions that is rich enough to capture the desired localization, yet sufficiently ``compressed" to be compact.

\subsubsection{The energy functional}
Given the convergence \eqref{going_to_K}, the next key observation is that the free energy $F_n$ can be expressed in terms of the empirical measure $\mu_n$, see \eqref{Fn_ito_empirical}.
We thus define a functional $\r : \p(\s) \to \R$ so that we may concisely write
$\EE(F_n) = \EE\, \r(\mu_{n})$.
This and Corollary \ref{close_probability} lead to the following variational formula for the limiting free energy, given in Theorem \ref{upper_bound}. For any $\beta$ such that \eqref{mgf} holds,
\eq{
p(\beta) = \lim_{n \to \infty} \EE(F_n) = \inf_{\nu \in \k} \r(\nu).
}
The connections between this formula and those of \cite{rassoul-seppalainen-yilmaz13,rassoul-seppalainen-yilmaz17} are unclear.

Nevertheless, this computation allows us to improve Corollary \ref{close_probability} in an important way to yield Theorem \ref{close_M}, which says that if 
\[
\m \coloneqq  \Big\{\nu_0 \in \k : \r(\nu_0) = \inf_{\nu \in \k} \r(\nu)\Big\},
\]
then
\eq{
\lim_{n \to \infty} \inf_{\nu \in \m} \w(\mu_n,\nu) = 0 \quad \mathrm{a.s.}
}
In Section \ref{empirical_limits} we study the minimizing set $\m$. In particular, Theorem \ref{characterization} says that either $\m$ consists of the single element of total mass zero (which happens when $0\le \beta\le \beta_{\mathrm{c}}$, where $\beta_{\mathrm{c}}$ is the critical inverse temperature of Theorem \ref{critical_temperature}), or every element of $\m$ has total mass one (which happens when $\beta>\beta_{\mathrm{c}}$).
This result is similar to the technique of identifying a phase transition as the critical point at which a recursive distributional equation begins to have a nontrivial solution; for an account of this method, we refer the reader to the survey of Aldous and Bandyopadhyay~\cite{aldous-bandyopadhyay05} and references therein.
This idea is also present in work of Yoshida \cite{yoshida08} on more general linear stochastic evolutions, although there the nontrivial solutions exist in the high temperature regime rather than the low temperature regime.

From a different perspective, the limit law of the empirical measure can be viewed as an ``order parameter'' for the model, whose behavior distinguishes between the high and low temperature regimes. Such order parameters arise frequently in the study of disordered systems. A prominent example is the Sherrington--Kirkpatrick (SK) model of spin glasses, where the limiting distribution of the overlap serves as the order parameter (see Panchenko~\cite{panchenko13}). Interestingly, the limiting free energy of the SK model can also be expressed as the solution of a variational problem involving the order parameter. This is the famous Parisi formula proved by Talagrand~\cite{talagrand06}. 
In this way, the partitioned subprobability measures might be seen as counterparts to the random overlap structures introduced by Aizenman, Sims and Starr~\cite{ass07}, 
and the update map as the analog to similar stabilizing maps arising out of the cavity method for spin glasses and related competing particle systems, that were studied by Aizenman and Contucci~\cite{aizenman-contucci98}, Ruzmaikina and Aizenman~\cite{ruzmaikina-aizenman05} and Arguin and Chatterjee~\cite{arguin-chatterjee13}.
In other ways, however, the analogy is quite distant.
For instance, the spin glass systems we speak of are mean-field models lacking any geometry from the lattice.
Also, our variational formula relies on the directed nature of the problem; that is, the random environment refreshes at each time step, allowing us to exploit Markovian structure.

\subsection{Main applications}
Theorem \ref{characterization} yields the following concrete application of our abstract theory of partitioned subprobability measures; it is later stated as Theorem \ref{total_mass}.
Recall the notations and terminologies related to Theorem~\ref{critical_temperature} and Theorem~\ref{vargas_apa}. 

\begin{thm} \label{intro_result1}
Assume \eqref{mgf}.
\begin{itemize}
\item[(a)] If $\beta > \beta_{\mathrm{c}}$, then for every sequence $(\eps_i)_{i \geq 0}$ tending to 0 as $i \to \infty$,
\eq{
\lim_{n \to \infty} \frac{1}{n} \sum_{i = 0}^{n-1} \rho_{i}(\omega_i \in \a_i^{\eps_i}) = 1 \quad \mathrm{a.s.}
}
\item[(b)] If $0\le \beta \le \beta_{\mathrm{c}}$, then there exists a sequence $(\eps_i)_{i \geq 0}$ tending to 0 as $i \to \infty$ such that
\eq{
\lim_{n \to \infty} \frac{1}{n} \sum_{i = 0}^{n-1} \rho_{i}(\omega_i \in \a_i^{\eps_i}) = 0 \quad \mathrm{a.s.}
}
\end{itemize}
\end{thm}

This generalizes Theorem \ref{vargas_apa}, where an ``in probability" version of (a) was proved under the condition $\lambda(\beta)=\infty$. 

In Section \ref{main_thm2} we apply our techniques to go further than atomic localization by considering ``geometric localization".
In the low temperature phase, the endpoint distribution can not only concentrate mass on a few likely sites, but moreover have those sites close together.
We make this phenomenon precise in the following manner.
For $\delta \in (0,1)$ and a nonnegative number $K$, let $\g_{\delta,K}$ denote the set of probability mass functions on $\Z^d$ that assign measure greater than $1-\delta$ to some subset of $\Z^d$ having diameter at most $K$; see \eqref{g_set_def} for a symbolic definition.
We will say that the sequence $(\rho_i(\omega_i \in \cdot))_{i \geq 0}$ is \textit{geometrically localized with positive density} if for every $\delta$, there is $K< \infty$ and $\theta > 0$ such that
\eq{
\liminf_{n \to \infty} \frac{1}{n} \sum_{i = 0}^{n-1} \one_{\{\rho_i(\omega_i = \cdot)\in \g_{\delta,K}\}} \geq \theta \quad \mathrm{a.s.,}
}
where $\one_{A}$ denotes the indicator of the event $A$.
That is, there are endpoint distributions with limiting density at least $\theta$ that place mass greater than $1 - \delta$ on a set of bounded diameter.
We will say $(\rho_i(\omega_i \in \cdot))_{i \geq 0}$ is \textit{geometrically localized with full density} if for every $\delta$, there is $K < \infty$ such that
\eeq{
\liminf_{n \to \infty} \frac{1}{n} \sum_{i = 0}^{n-1} \one_{\{\rho_i(\omega_i = \cdot)\in \g_{\delta,K}\}} \geq 1-\delta \quad \mathrm{a.s.} \label{geometric_localization}
}
The main result of Section \ref{main_thm2} is contained in Theorem \ref{localized_subsequence} and says the following.

\begin{thm} \label{intro_result2}
Assume \eqref{mgf}.  
\begin{itemize} 
\item[(a)] If $\beta > \beta_{\mathrm{c}}$, then there is geometric localization with positive density. 
Moreover, the numbers $K$ and $\theta$ are deterministic quantities that depend only on the choice of $\delta$, the disorder distribution $\mathfrak{L}$, the parameter $\beta$, and the dimension $d$.
\item[(b)] If $0 \leq \beta \leq \beta_{\mathrm{c}}$, then for any $K$ and any $\delta \in (0,1)$,
\eq{
\lim_{n\to\infty} \frac{1}{n} \sum_{i=0}^{n-1} \one_{\{\rho_i(\omega_i = \cdot) \in \g_{\delta,K}\}} = 0 \quad \mathrm{a.s.}
}
\end{itemize}
\end{thm}

As mentioned in Section \ref{solvability_background}, the only case where a version of geometric localization has been proved is an integrable $(1+1)$-dimensional model, for which Comets and Nguyen \cite{comets-nguyen16} proved localization and moreover computed the limit distribution of the endpoint. Similar results for one-dimensional random walk in random environment were proved by Sinai~\cite{sinai82}, Golosov~\cite{golosov84}, and Gantert, Peres and Shi~\cite{gantert-peres-shi10}.

\subsubsection*{The single-copy condition}
In addition to the above unconditional results, we also prove a few conditional statements, which hold under the condition that every $\nu \in \m$ puts all its mass on those $f \in \s$ that are supported on a single copy of $\Z^d$. We call this the ``single-copy condition''. 

One consequence of the single-copy condition is geometric localization with full density, as defined in equation~\eqref{geometric_localization}. Part (b) of Theorem \ref{localized_subsequence} proves this conditional claim. 
A second consequence of the single-copy condition is Proposition \ref{localization_thm}, which gives the following Ces\`aro form of \eqref{mode}. For each $i\ge 0$ and $K\ge 0$, let $\c_{i}^K$ be the set of all $x\in \Z^d$ that are at distance $\le K$ from {\it every} mode of the endpoint mass function $\rho_i(\omega_i = \cdot)$. 
Then, assuming \eqref{mgf} and the single-copy condition, Proposition~\ref{localization_thm} asserts that 
\eq{
\lim_{K \to \infty} \liminf_{n \to \infty} \frac{1}{n} \sum_{i = 0}^{n-1} \rho_i(\omega_i\in \c_{i}^K) = 1 \quad \mathrm{a.s.}
}
In view of the result \eqref{mode} of Comets and Nguyen~\cite{comets-nguyen16}, it seems plausible that the single-copy condition holds for the log-gamma polymer in $1+1$ dimensions. 
Unfortunately, we have been unable to determine whether or not the single-copy condition holds in general. 
Furthermore, we are not aware of any conjectures on what is true in higher dimensions.
The results of \cite{barral-rhodes-vargas12} suggest that at least for directed polymers on $b$-ary trees, the single-copy condition does not hold, and full geometric localization is not valid. 
This may be related to the fluctuations of $\log Z_n$, which are known to be order $1$ on the tree (see \cite{derrida-spohn88}) but conjectured to be order $n^{1/3}$ when $d=1$.


\section{Partitioned subprobability measures} \label{topology}
In this section and in the remainder of this manuscript, we will  always assume \eqref{mgf}. Also, throughout, $f_i(\cdot)$ will denote the endpoint probability mass function $\rho_i(\omega_i=\cdot)$. 
 The goals of this section are to (i) define a space of functions which contains endpoint distributions (i.e.~probability measures on $\Z^d$) and their localization limits (subprobability measures on $\N \times \Z^d$); (ii) equip said space with a metric; and (iii) prove that the induced metric topology is compact.

\subsection{Definition and properties} \label{topology-definitions}
Let us restrict our attention to the set of functions
\eq{
S \coloneqq  \{f : \N \times \Z^d \to \R : f \geq 0,\, \|f\| \leq 1\},
}
where
\eeq{
\|f\| \coloneqq  \sum_{u \in \N \times \Z^d} f(u). \label{norm_def}
}
Since we regard distant point masses as nearly existing on separate copies of $\Z^d$, we consider
differences between elements of $\N \times \Z^d$ in the following non-standard way:
\eq{
(n,x) - (m,y) \coloneqq  \begin{cases}
x - y &\text{if } n = m, \\
\infty &\text{else}.
\end{cases}
}
It then makes sense to write, for $u = (n,x)$ and $v = (m,y)$ in $\N \times \Z^d$,
\eq{
\|u - v\|_1 \coloneqq  \begin{cases}
\|x - y\|_1 &\text{if } n = m, \\
\infty &\text{else},
\end{cases}
}
although $\|\cdot\|_1$ is not to be thought of as a norm.
If $u = (n,x)$ and $y \in \Z^d$, then $u \pm y$ will be understood to mean $(n,x\pm y)$.
Our main tool for enabling compactness will be certain injective functions on $\N \times \Z^d$.
Given a set $A \subset \N \times \Z^d$, we call a map $\phi: A \to \N \times \Z^d$ an \textit{isometry} of degree $m \geq 1$ if for all $u,v\in A$,
\eeq{
\|u - v\|_1 < m
\quad \text{or} \quad
\|\phi(u) - \phi(v)\|_1 < m
\quad \implies \quad 
\phi(u) - \phi(v) = u - v.
 \label{isometry}
}
The maximum $m$ for which \eqref{isometry} holds is called the \textit{maximum degree} of $\phi$, denoted by $\deg(\phi)$.
To say that $\phi$ is an isometry of degree $1$ simply means $\phi$ is injective.
If \eqref{isometry} holds for every $m \in \N$, then $\deg(\phi) = \infty$, meaning $\phi$ acts by translations. That is, each copy of $\Z^d$ intersecting the domain $A$ is translated and moved to some copy of $\Z^d$ in the range, with no two copies in the domain going to the same copy in the range.
Note that an isometry is necessarily injective and thus has an inverse defined on its image.
Because the hypothesis in \eqref{isometry} is symmetric between domain and range, it is clear that
$\deg(\phi) = \deg(\phi^{-1})$.

%

Another useful property of isometries is closure under composition, which is defined in the next lemma.
It is important, especially for the proof below, to note that an isometry is permitted to have empty domain.
That is, there exists the empty isometry $\phi: \varnothing \to \N \times \Z^d$, which is its own inverse and has $\deg(\phi) = \infty$.

\begin{lemma} \label{composition}
Let $\phi : A \to \N \times \Z^d$ and $\psi : B \to \N \times \Z^d$ be isometries.
Define $A' \coloneqq  \{a \in A : \phi(a) \in B\}$.  Then $\theta: A' \to \N \times \Z^d$ defined by $\theta(u) = \psi(\phi(u))$ is an isometry with $\deg(\theta) \geq \min(\deg(\phi),\deg(\psi))$.
\end{lemma}

\begin{proof}
Let $m \coloneqq  \min(\deg(\phi),\deg(\psi))$ so that $\phi$ and $\psi$ are each isometries of degree $m$.
For any $a_1,a_2 \in A'$,
\eq{
\|a_1 - a_2\|_1 < m \quad &\implies \quad \phi(a_1) - \phi(a_2) = a_1 - a_2 \\
&\implies \quad \|\phi(a_1) - \phi(a_2)\|_1 < m \\
&\implies \quad \psi(\phi(a_1)) - \psi(\phi(a_2)) = \phi(a_1) - \phi(a_2) = a_1 - a_2,
}
as well as
\eq{
\|\theta(a_1) - \theta(a_2)\|_1 < m \quad &\implies \quad \phi(a_1) - \phi(a_2) = \psi(\phi(a_1)) - \psi(\phi(a_2)) \\
&\implies \quad \|\phi(a_1) - \phi(a_2)\|_1 < m \\
&\implies \quad a_1 - a_2 = \phi(a_1) - \phi(a_2) \\
&\phantom{\implies \quad a_1-a_2}\hspace{0.7ex} = \psi(\phi(a_1)) - \psi(\phi(a_2)).
}
Indeed, $\theta$ is an isometry of degree $m$.
\end{proof}

A final observation about isometries is that they obey a certain extension property, which is proved below:
By allowing the maximum degree of an isometry to be lowered by at most 2, we can expand its domain  by one unit in every direction.
If the maximum degree is infinite (i.e.~$\phi$ acts by translations), then the extension can be repeated \textit{ad infinitum} to recover the translation on all of $\Z^d$, for any copy of $\Z^d$ intersecting the domain.

\begin{lemma} \label{extension}
Suppose that $\phi : A \to \N \times \Z^d$ is an isometry of degree $m \geq 3$.
Then $\phi$ can be extended to an isometry $\Phi: B \to \N \times \Z^d$ of degree $m-2$, where
\eq{
B \coloneqq  \{v \in \N \times \Z^d : \|u - v\|_1 \leq 1 \textup{ for some $u \in A$}\} \supset A.
}
\end{lemma}

\begin{proof}
For $v \in B$ such that $\|u - v\|_1 \leq 1$ with $u \in A$, define
\eq{
\Phi(v) \coloneqq  \phi(u) + (v - u).
}
If $u' \in A$ also satisfies $\|u' - v\|_1 \leq 1$, then
\eq{
\|u - u'\|_1 \leq 2 \quad &\implies \quad \phi(u) - \phi(u') = u - u'
\quad \\
&\implies \quad \phi(u) + (v - u) = \phi(u') + (v - u').
}
So $\Phi$ is well-defined; in particular, $\Phi(u) = \phi(u)$ for all $u \in A$.
To see that $\Phi$ is an isometry of degree $m-2$, consider any $v,v' \in B$ and take $u,u' \in A$ such that $\|u - v\|_1 \leq 1$ and $\|u' - v'\|_1 \leq 1$.
If $\|v - v'\|_1 < m - 2$, then
\eq{
\|u - u'\|_1 < m \quad &\implies \quad \phi(u) - \phi(u') \hspace{0.3ex} = u - u' \\
&\implies \quad \Phi(v) - \Phi(v') = \phi(u) + (v - u)  - \phi(u') - (v' - u') \\
&\phantom{\implies \quad \Phi(v) - \Phi(v') }\hspace{0.7ex}   = v - v'.
}
Alternatively, if $\|\Phi(v) - \Phi(v')\|_1 < m -2$, then
\eq{
\|\phi(u) - \phi(u')\| < m
\quad &\implies \quad u - u' = \phi(u) - \phi(u') \\
&\phantom{\implies \quad u - u' }\hspace{0.7ex}  = \Phi(v) - (v-u) - \Phi(v') + (v'-u') \\
&\implies \quad v - v' \hspace{0.3ex} = \Phi(v) - \Phi(v').
}
Indeed, $\deg(\Phi) \geq m-2$.
\end{proof}

We can now define the desired metric on functions.
Roughly speaking, an isometry allows for the comparison of the large values of two functions. 
The size of the isometry's domain reflects how many of their large values are similar, while the degree of the isometry captures how similar their relative positioning is.
The metric is constructed in stages.

First, given an isometry $\phi: A \to \N \times \Z^d$ with finite (for measurability reasons, see Lemma \ref{S_meas}) domain $A \subset \N \times \Z^d$ and two functions $f,g \in S$, define
\eeq{
d_\phi(f,g) &\coloneqq  2\sum_{u \in A} |f(u) - g(\phi(u))| + \sum_{u \notin A} f(u)^{2}  + \sum_{u \notin \phi(A)} g(u)^{2} + 2^{-\deg(\phi)}. \label{d_def}
}
Next define
\eq{
d(f,g) \coloneqq  \inf_{\phi} d_\phi(f,g),
}
where the infimum is taken over all isometries $\phi$ with finite domain.
Since $\deg(\phi^{-1}) = \deg(\phi)$, it is easy to see that $d_{\phi^{-1}}(g,f) = d_\phi(f,g)$, and so the function $d$ is symmetric:
\eeq{
d(f,g) = d(g,f) \quad \text{for all } f,g \in S. \label{symmetry}
}
In fact, $d$ also satisfies the triangle inequality on $S$.

\begin{lemma}
For any $f,g,h \in S$,
\eeq{
d(f,h) \leq d(f,g) + d(g,h). \label{triangle}
}
\end{lemma}

\begin{proof}
Fix $\eps > 0$, and choose isometries $\phi : A \to \N \times \Z^d$ and $\psi : B \to \N \times \Z^d$ such that
\eq{
d_\phi(f,g) < d(f,g) + \eps \quad \text{and} \quad
d_\psi(g,h) < d(g,h) + \eps.
}
Define $\theta : A' \to \N \times \Z^d$ as in Lemma \ref{composition}.
We have
\eeq{
d_\theta(f,h) &= 2\sum_{u \in A'} |f(u) - h(\theta(u))| + \sum_{u \notin A'} f(u)^2  + \sum_{u \notin \theta(A')} h(u)^2 + 2^{-\deg(\theta)}. \label{theta_distance}
}
The first sum above can be bounded as
\eeq{ 
\sum_{u \in A'} |f(u) - h(\theta(u))|
&\leq \sum_{u \in A'} \bigl(|f(u) - g(\phi(u))| + |g(\phi(u)) - h(\psi(\phi(u)))|\bigr)  \\
&= \sum_{u \in A'} |f(u) - g(\phi(u))| + \sum_{u \in B \cap \phi(A)} |g(u) - h(\psi(u))|. \label{tri_bound_1}
}
Recall that $f$, $g$, and $h$ take values in $[0,1]$.
Now, if $f(u) < g(v)$, then one trivially has $f(u)^2 \leq g(v)^2 \leq g(v)^2 + 2|f(u) - g(v)|$. 
On the other hand, if $f(u) \geq g(v)$, then one again has
\eq{
f(u)^2 - g(v)^2 &= (f(u)+g(v))(f(u) - g(v))  \\
&\leq 2(f(u) - g(v))  \\
\implies \quad f(u)^2 &\leq 2|f(u) - g(v)| + g(v)^2.
}
As a result, the second sum in \eqref{theta_distance} satisfies
\eeq{
\sum_{u \notin A'} f(u)^{2} &= \sum_{u \in A \setminus A'} f(u)^2 + \sum_{u \notin A} f(u)^2  \\
&\leq \sum_{u \in A \setminus A'}\bigl( 2|f(u) - g(\phi(u))| + |g(\phi(u))|^2\bigr) + \sum_{u \notin A} f(u)^2  \\
&\leq 2\sum_{u \in A \setminus A'} |f(u) - g(\phi(u))| + \sum_{u \notin B} g(u)^2 + \sum_{u \notin A} f(u)^2. \label{tri_bound_2}
}
Similarly, the third sum satisfies
\eeq{ 
\sum_{u \notin \theta(A')} h(u)^2 &= \sum_{u \in \psi(B) \setminus \theta(A')} h(u)^2 + \sum_{u \notin \psi(B)} h(u)^{2}  \\
&\leq \sum_{u \in B \setminus \phi(A)} \bigl(2|h(\psi(u)) - g(u)| + g(u)^2\bigr) + \sum_{u \notin \psi(B)} h(u)^{2}  \\
&\leq 2\sum_{u \in B \setminus \phi(A)} |g(u) - h(\psi(u))| + \sum_{u \notin \phi(A)} g(u)^{2}+ \sum_{u \notin \psi(B)} h(u)^{2}. \label{tri_bound_3}
}
Finally, Lemma \ref{composition} guarantees 
\eeq{
\deg(\theta) &\geq \min(\deg(\phi),\deg(\psi))  \\
\quad \implies 2^{-\deg(\theta)} &\leq 2^{-\deg(\phi)} + 2^{-\deg(\psi)}. \label{tri_bound_4}
}
Using \eqref{tri_bound_1}--\eqref{tri_bound_4} in \eqref{theta_distance}, we find
\eq{
d(f,h) \leq d_\theta(f,h) \leq d_\phi(f,g) + d_\psi(g,h) < d(f,g) + d(g,h) + 2\eps.
}
As $\eps$ is arbitrary, \eqref{triangle} follows.
\end{proof}

From \eqref{symmetry} and \eqref{triangle}, we see that $d$ is a pseudometric on $S$.
It does not, however, separate points.
Nevertheless, one can construct a metric space $(\mathcal{S},d)$ by taking the quotient of $S$ with respect to the equivalence relation
\eq{
f \equiv g \quad \iff \quad d(f,g) = 0.
}
We shall write without confusion the symbol $f$ for both the equivalence class in $\s$ and the representative chosen from $S$.
When $f$ is evaluated at a particular element $u \in \N \times \Z^d$, an explicit representative has been chosen.
In the sequel, it will be important that certain properties of elements of $S$ are invariant within the equivalence classes and thus well-defined in $\s$.
To identify the equivalence classes, we have the following result.

\begin{prop} \label{superisometry}
Two functions $f,g \in S$ satisfy $d(f,g) = 0$ if and only if
there is a set $B\subset \N \times \Z^d$ and a map $\psi : B \to \N \times \Z^d$ such that 
\begin{itemize}
\item[(i)] $f(u) = g(\psi(u))$ for all $u \in B$,
\item[(ii)] $f(u) = 0$ for all $u \notin B$,
\item[(iii)] $g(u) = 0$ for all $u \notin \psi(B)$, and
\item[(iv)] $\psi(u) - \psi(v) = u - v$ for all $u,v \in B$.
\end{itemize}
\end{prop}

The proof is purely technical and thus placed in Appendix \ref{proof_superisometry}.
A more transparent description of equivalence under $d$ was given in the introductory Section~\ref{results}.
We restate it here, and prove equivalence to the conditions of Proposition \ref{superisometry}.
Recall that the $\N$-support of $f \in S$ is the set
\eq{
H_f \coloneqq  \{n \in \N : f(n,x) > 0 \text{ for some $x \in \Z^d$}\}.
}

\begin{cor} \label{better_def_cor}
Let $f,g \in S$ have $\N$-supports denoted by $H_f$ and $H_g$, respectively.
Then $d(f,g) = 0$ if and only if there is a bijection $\sigma : H_f \to H_g$ and vectors $(x_n)_{n \in H_f} \subset \Z^d$ such that 
\eeq{
g(\sigma(n),x) = f(n,x+x_n) \quad \text{for all $n \in H_f$, $x \in \Z^d$.} \label{better_def}
}
\end{cor}

\begin{proof}
First assume $d(f,g) = 0$, and take $\psi : B \to \N \times \Z^d$ as in Proposition \ref{superisometry} so that properties (i)--(iv) hold.
By repeatedly applying Lemma \ref{extension},
one may assume that the domain $B$ is a union of copies of $\Z^d$; that is, $B = H \times \Z^d$.
By property (ii), we can take $H = H_f$, while property (iv) shows first that for every $n \in H_f$, there is $\sigma(n) \in \N$ and $x_n \in \Z^d$ satisfying
\eq{
\psi(n,x) = (\sigma(n),x-x_n) \quad \text{for all $x \in \Z^d$,}
}
and second that $n \mapsto \sigma(n)$ is injective.
Then (i) leads to \eqref{better_def}, while (iii) guarantees that $\sigma(H_f) = H_g$. Conversely, assume $\sigma : H_f \to H_g$ and $(x_n)_{n \in H_f}$ satisfy \eqref{better_def}.
Then define $\psi : H_f \times \Z^d \to H_g \times \Z^d$ by $\psi(n,x) \coloneqq  (\sigma(n),x-x_n)$.
Since the domain of $\psi$ is $H_f \times \Z^d$, and its range is $H_g \times \Z^d$, this map satisfies (ii) and (iii).
At the same time, (iv) holds by construction, and (i) follows from~\eqref{better_def}.
\end{proof}

We can now state a clear geometric condition for a function on $S$ to be well-defined on $\s$.
For the functions that will be of interest to us, it will be obvious that the hypotheses of the following corollary are satisfied.

\begin{cor} \label{defined_pspm}
Suppose $L$ is a function on $S$ satisfying the following:
\begin{itemize}
\item[(i)] L is invariant under shifts of $\Z^d$: If there are vectors $(x_n)_{n \in \N}$ in $\Z^d$ such that $f(n,x) = g(n,x-x_n)$, then $L(f) = L(g)$.
\item[(ii)] L is invariant under permutations of $\N$: If there is a bijection $\sigma : \N \to \N$ such that $f(n,x) = g(\sigma(n),x)$, then $L(f) = L(g)$.
\item[(iii)] L is invariant under zero-padding: If there is an increasing sequence $(n_k)_{k \geq 1}$ in $\N$ such that
\eq{
f(n,x) = \begin{cases}
g(k,x) &\text{if } n = n_k, \\
0 &\text{else},
\end{cases}
}
then $L(f) = L(g)$.
\end{itemize}
Then $L$ is well-defined on $\s$ by evaluating at any representative in $S$.
\end{cor}

\begin{proof}
Suppose $f,g \in S$ are such that $d(f,g) = 0$.
We wish to show that $L(f) = L(g)$.
Let $H_f$ and $H_g$ be the $\N$-supports of $f$ and $g$ respectively, and let $\sigma : H_f \to H_g$ be a bijection satisfying \eqref{better_def} with vectors $(x_n)_{n \in H_f}$.
If $H_f$ is finite, then $\sigma$ can be trivially extended to a bijection on $\N$, and then invariance properties (i) and (ii) are sufficient to show $L(f) = L(g)$.
If $H$ is infinite, then we enumerate the sets
\eq{
H_f = \{n_1 < n_2 < \cdots\} \qquad H_g = \{m_1 < m_2 < \cdots\},
}
and define the functions $h_f \in S$ by $h_f(k,x) = f(n_k,x)$ and $h_g(\ell,x) = g(m_\ell,x)$.
By (iii), we have $L(f) = L(h_f)$ and $L(g) = L(h_g)$.
Furthermore, $\sigma$ induces a bijection $\sigma' : \N \to \N$ by
\eq{
\sigma : n_k \mapsto m_\ell \quad \iff \quad \sigma' : k \mapsto \ell.
}
We now have
\eq{
h_f(k,x) = f(n_k,x) = g(\sigma(n_k),x - x_n) = h_g(\sigma'(k),x - x_n),
}
so that (i) and (ii) give $L(h_f) = L(h_g)$.
So $L(f) = L(g)$ in this case as well.
\end{proof}

In the next lemma we discuss our first example of a function $L$ satisfying the hypotheses of Corollary \ref{defined_pspm}.
This function will be important in defining the ``update procedure" of Section~\ref{transformation}. 

\begin{lemma} \label{norm_equivalence}
The map $\|\cdot\| : \s \to [0,1]$ defined by \eqref{norm_def} is lower semi-continuous and thus measurable.
\end{lemma}

The proof below of Lemma \ref{norm_equivalence} is similar to the arguments for several later results. 
Since one encounters in these cases only variations of this proof, we will frequently spare the reader details and place them in Appendix \ref{remaining_technical_details}.

\begin{proof}[Proof of Lemma \ref{norm_equivalence}]
It is clear that $\|\cdot\| : S \to [0,1]$ satisfies (i)--(iii) in Corollary \ref{defined_pspm}, and so the map $f \mapsto \|f\|$ is well-defined on $\s$.
To prove lower semi-continuity, it suffices to fix $f \in \s$, let $\eps > 0$ be arbitrary, and find $\delta > 0$ sufficiently small that
\eq{
d(f,g) < \delta \quad \implies \quad \|g\| > \|f\| - \eps.
}
Upon selecting a representative $f \in S$, we can find $A \subset \N \times \Z^d$ finite but large enough that
$\sum_{u \notin A} f(u) < \frac{\eps}{2}$.
By possibly omitting some elements of $A$, we may assume that $f$ is strictly positive on $A$ (if $f$ is the constant zero function, this results in $A = \varnothing$).

Now take $0 < \delta < \inf_{u \in A} f(u)^2$, where an infimum over the empty set is $\infty$.
We will further assume $\delta < \eps$.
If $d(f,g) < \delta$, then there is a representative $g \in S$ and an isometry $\phi : C \to \N \times \Z^d$ such that $d_\phi(f,g) < \delta$.
It follows that $A \subset C$, since otherwise we would have $d_\phi(f,g) \geq f(u)^2 > \delta$ for some $u \in A \setminus C$.
Hence
\eq{
\|g\| \geq \sum_{u \in \phi(A)} g(u)
&\geq \sum_{u \in A} f(u) - \sum_{u \in A} |f(u) - g(\phi(u))| + \sum_{u \notin A} f(u) - \sum_{u \notin A} f(u) \\
&> \|f\| - \frac{d_\phi(f,g)}{2} - \frac{\eps}{2} > \|f\| - \frac{\delta}{2} - \frac{\eps}{2} > \|f\| - \eps.
}
\end{proof}

\subsection{Compactness} \label{compactness}
We now state 
the key compactness result for the metric space $(\s,d)$.

\begin{thm} \label{compactness_result}
Every sequence $(f_n)_{n \geq 1}$ in $(\s,d)$ contains a converging subsequence.  
That is, there is $f \in \s$ and a subsequence $(f_{n_k})_{k \geq 1}$ such that
\eq{
\lim_{k \to \infty} d(f_{n_k},f) = 0.
}
\end{thm}

A continuum version of this result is Theorem 3.2 of \cite{mukherjee-varadhan16}. 
The proof of Theorem \ref{compactness_result}, however, strongly uses discreteness of $\Z^d$ to deal with the explicit metric $d$.
It can be found in Appendix \ref{proof_compactness}.

The sequential compactness guaranteed by Theorem \ref{compactness_result} 
will be used to establish results in the following metric space of probability measures.
Denote by $\p(\s)$ the set of Borel probability measures on $\s$, and equip this space with the \textit{Wasserstein metric} (\cite{villani09}, Definition 6.1):
\eq{
\w(\mu,\nu) \coloneqq  \inf_{\pi \in \Pi(\mu,\nu) }\int_{\s \times \s} d(f,g)\, \pi(\dd f,\dd g),
}
where $\Pi(\mu,\nu)$ denotes the set of probability measures on $\s \times \s$ having $\mu$ and $\nu$ as marginals.
Note that we may use dual representation of $\w$ due to Kantorovich~\cite{kantorovich42} when it is convenient:
\eeq{
\w(\mu,\nu) = \sup_{\vphi} \biggl(\int_\s \vphi(f)\ \mu(\dd f) - \int_\s \vphi(f)\ \nu(\dd f)\biggr),\label{kantorovich}
}
where the supremum is over $1$-Lipschitz functions $\vphi : \s\to\R$.

Now we recall some general facts on convergence in $\p(\s)$.
It is a standard result (e.g.~see \cite{villani09}, Theorem 6.9) that for measures on a compact metric space, the Wasserstein distance metrizes the topology of weak convergence.
Furthermore, this topology is again compact (\cite{villani09}, Remark 6.19).
In the coming sections, we will employ weak convergence to prove convergence of not only continuous test functions, but also semi-continuous test functions.
Therefore, we will repeatedly apply the Portmanteau lemma, which we record here so that it can be properly quoted later in the paper.

\begin{lemma}[{Portmanteau, cf.~\cite{billingsley99}, Theorem 2.1 and \cite{vandervaart-wellner96}, Theorem 1.3.4}] \label{portmanteau}
Given a function $L: \s \to \R$, define the map $\l : \p(\s) \to \R$ by
\eq{
\l(\mu) \coloneqq  \int_\s L(f)\ \mu(\dd f).
}
If $L$ is lower (resp.~upper) semi-continuous, then $\l$ is lower (resp.~upper) semi-continuous.
\end{lemma}

We conclude this section with some observations that will be needed in later arguments.

\begin{lemma} \label{trivial_bound}
For any $f,g \in \s$, $d(f,g) \leq 2$.
\end{lemma}

\begin{proof}
Pick representatives $f,g \in S$ and let $\phi : \varnothing \to \N \times \Z^d$ be the empty isometry.
Then
\eq{
d(f,g) \leq d_\phi(f,g) = \sum_{u \in \N \times \Z^d} f(u)^2 + \sum_{u \in \N \times \Z^d} g(u)^2 \leq 2.
}
\end{proof}

The next lemma concerns measure-theoretic properties of the spaces $S$ and $\s$.
In particular, $S$ is considered as a subset of
\eq{
\ell^1(\N \times \Z^d) = \{f : \N \times \Z^d \to \R : \|f\| < \infty\},
}
on which there is the standard $\ell^1$ norm $\|\cdot\|$ that extends \eqref{norm_def}:
\eq{
\|f\| \coloneqq  \sum_{u \in \N \times \Z^d} |f(u)|, \quad f \in \ell^1(\N \times \Z^d).
}

\begin{lemma} \label{S_meas}
Consider the metric spaces $\ell^1(\N \times \Z^d)$ and $\s$ with their Borel $\sigma$-algebras.
Then the following statements hold: 
\begin{itemize}
\item[(a)] $S$ is a closed (in particular, measurable) subset of $\ell^1(\N \times \Z^d)$ and is thus itself a measurable space with the subspace $\sigma$-algebra.
\item[(b)] The quotient map $\iota: S \to \s$ that sends $f \in S$ to its equivalence class in $\s$ is measurable.
\end{itemize}
\end{lemma}

\begin{proof}
To show $S \subset \ell^1(\N \times \Z^d)$ is closed, we express $S$ as the 
intersection of closed sets:
\eq{
S = \{f \in \ell^1(\N \times \Z^d) : \|f\| \leq 1\} \cap \bigg(\bigcap_{u \in \N \times \Z^d} \{f \in \ell^1(\N \times \Z^d) : f(u) \geq 0\}\bigg).
}
To next show $\iota$ is measurable, it suffices to verify that the inverse image of any open ball is measurable.
For $f \in \s$ and $r > 0$, we write
\eq{
B_r(f) \coloneqq  \{g \in \s : d(f,g) < r\}.
}
Notice that
\eq{
\iota^{-1}(B_r(f)) = \bigcup_{\phi}\ \{g \in S : d_\phi(f,g) < r\},
}
where the union is over isometries with \textit{finite} domains.
The union occurs, therefore, over a countable set.
For each $\phi$, it is clear from \eqref{d_def} that $d_\phi(f,\, \cdot\, )$ is a measurable function on $S$, and so each set in the union is measurable.
Being a countable union of measurable sets, $\iota^{-1}(B_r(f))$ is measurable.
\end{proof}

\section{The update map} \label{transformation}
Consider the distribution of $\omega_n$ under $\rho_n$. 
In this section we identify said distribution with an element $f_n \in \s$ and define the Markov kernel $T$ that maps $f_n$ to the law of $f_{n+1}$ when conditioned on the random environment only up to time $n$.
In the following, the notation $u \sim v$ is used when $u,v \in \N \times \Z^d$ satisfy $\|u-v\|_1 = 1$.
The same notation will be used for adjacent elements $x,y \in \Z^d$.

\subsection{Definition of the update map} \label{endpoint_distributions}
To make the setup precise, we recall the polymer measure $\rho_n$ defined by \eqref{rho_def} and the endpoint probability mass function $f_n : \Z^d \to \R$ given by
\eeq{
f_n(x) \coloneqq  \rho_n(\omega_n = x) = \frac{(2d)^{-n}}{Z_n} \sum_{\substack{|\gamma| = n \\ \gamma(0) = 0,\: \gamma(n) = x}} \exp\bigg(\beta\sum_{i = 1}^n X_{i,\, \gamma(i)}\bigg), \label{fn_def}
}
where the sum is over nearest-neighbor paths $\gamma : \{0,1,\dots,n\} \to \Z^d$ of length $|\gamma| = n$, starting at the origin and ending at $x$.
Then $f_n$ is a $[0,1]$-valued function on $\Z^d$ and random with respect to $(\Omega_{\mathrm{e}},\f_n)$, where the $\sigma$-algebra $\f_n$ is defined in \eqref{gndef}.
Its value at $x$ gives the probability that a polymer of length $n$ in the given random environment has $x$ as its endpoint.

When the polymer is extended from length $n$ to length $n+1$, the mass function updates to
\eq{
f_{n+1}(x) 
&= \frac{(2d)^{-(n+1)}}{Z_{n+1}} \sum_{y\sim x} \sum_{\substack{|\gamma| = n \\ \gamma(0) = 0,\: \gamma(n) = y}} \exp\bigg(\beta\sum_{i = 1}^n X_{i,\, \gamma(i)} + \beta X_{n+1,\, x}\bigg) \\
&= \frac{Z_{n}}{(2d)Z_{n+1}} \sum_{y\sim x} f_{n}(y) e^{\beta X_{n+1,\, x}}
=\frac{\sum_{y\sim x} f_{n}(y) e^{\beta X_{n+1,\, x}}}{2d\cdot Z_{n+1}/Z_n},
}
where the denominator is a normalizing factor and thus equal to
\eeq{
2d\, \frac{Z_{n+1}}{Z_{n}} = 2d\, \frac{Z_{n+1}}{Z_n}\sum_{x \in \Z^d} f_{n+1}(x) = \sum_{x \in \Z^d}\sum_{y\sim x} f_n(y) e^{\beta X_{n+1,\, x}}. \label{Zfrac}
}
Recall that $(X_{n+1,\, x})_{x \in \Z^d}$ is independent of $\f_n$, while $f_n$ is measurable with respect to $\f_n$. 
Therefore, the distribution of $f_{n+1}(x)$ given $\f_n$ is equal to the distribution of
\eeq{
F(x) \coloneqq \frac{\sum_{y\sim x} f_{n}(y) e^{\beta Y_x}}{\sum_{z \in \Z^d} \sum_{y\sim z} f_{n}(y) e^{\beta Y_z}}, \label{F_predef}
}
where $(Y_z)_{z \in \Z^d}$ are i.i.d.~random variables distributed according to $\mathfrak{L}$ and independent of $\f_n$.
Correspondingly, there is an updated free energy,
\eq{
\log Z_{n+1} &= \log Z_{n} + \EE\givenk[\Big]{\log \frac{Z_{n+1}}{Z_{n}}}{\f_{n}} + (\log Z_{n+1} - \EE\givenk{\log Z_{n+1}}{\f_{n}}) \\
&= \log Z_n + R(f_n) + d_n,
}
where 
\eeq{ \label{R_predef}
R(f_n)\coloneqq\EE\givenk[\Big]{\log \frac{Z_{n+1}}{Z_{n}}}{\f_{n}}
=\EE\givenk[\bigg]{\log\bigg(\frac{1}{2d}\sum_{x \in \Z^d} \sum_{y \sim x} f_n(y)e^{\beta Y_x}\bigg)}{\f_n},
}
and
\eeq{ \label{d_n_def}
d_n \coloneqq \log \frac{Z_{n+1}}{Z_n} - R(f_n)
} 
is a martingale increment.
That is, the Doob decomposition of $(\log Z_n)_{n\geq0}$ is
\eeq{ \label{doob_decomposition}
\log Z_n = \sum_{i=0}^{n-1} d_i + \sum_{i=0}^{n-1} R(f_i),
}
a fact which has been frequently used in the literature (e.g.~\cite{carmona-hu02,comets-shiga-yoshida03,comets-shiga-yoshida04,carmona-hu06}).

Next we wish to reinterpret these observations in the space $\s$.
For a given $f\in\s$, choose a representative (also called $f$) from $S$.
Consider the law of the random variable $F \in \s$ whose representative is defined by
\eeq{ \label{F_def}
F(u) &\coloneqq \frac{\sum_{v \sim u} f(v) e^{\beta Y_u}}{\sum_{w \in \N \times \Z^d} \sum_{v \sim w} f(v)e^{\beta Y_w} + 2d(1 - \|f\|){e^{\lambda(\beta)}}},\ \ u\in\N\times\Z^d,
}
where $(Y_w)_{w \in \N \times \Z^d}$ are i.i.d.~random variables with law $\mathfrak{L}$.
Notice that \eqref{F_def} is a generalization of \eqref{F_predef} because in the latter, $\|f_n\| = 1$.
The additional summand in the denominator of \eqref{F_def} is needed so that $F$ is defined even when $f \equiv 0$; its precise value is chosen so that \eqref{Zfrac} is generalized in an averaged form:
\eq{
\EE\bigg[\sum_{w \in \N \times \Z^d} \sum_{v \sim w} f(v) e^{\beta Y_w} + 2d(1 - \|f\|){e^{\lambda(\beta)}}\bigg] 
&= 2d\cdot {e^{\lambda(\beta)}}
= 2d\cdot\EE\Big(\frac{Z_{n+1}}{Z_n}\Big).
}
It is thus apparent that the correct extension of \eqref{R_predef} is to define $R : \s\to\R$ by
\eeq{
R(f) &\coloneqq  \EE\log\Big(\frac{\wt F}{2d}\Big), \quad \wt F \coloneqq \sum_{u \in \N \times \Z^d} \sum_{v \sim u} f(v)e^{\beta Y_u} + 2d(1-\|f\|){e^{\lambda(\beta)}}, \label{R_def}
}
where the expectation is over the $Y_u$.

We must now check that the above definitions do not depend on the representative $f\in S$.

\begin{prop} \label{same_law}
Suppose $f,g \in S$ satisfy $d(f,g) = 0$.
Define $F$ and $\wt F$ as in~\eqref{F_def} and \eqref{R_def}, 
and similarly define
\eq{
G(u) \coloneqq  \frac{\sum_{v \sim u} g(v) e^{\beta Z_u}}{\wt G}, \quad \wt G \coloneqq{\sum_{w \in \N \times \Z^d} \sum_{v \sim w} g(v)e^{\beta Z_w} + 2d(1 - \|g\|){e^{\lambda(\beta)}}}, 
}
where the $Z_w$ are i.i.d.~with law $\mathfrak{L}$.
Then the following statements hold:
\begin{itemize}
\item[(a)] The law of $F \in \s$ is equal to the law of $G \in \s$.
\item[(b)] The law of $\wt F \in \R$ is equal to the law of $\wt G \in \R$.
In particular, $R(f) = R(g)$.
\end{itemize}
\end{prop}

\begin{proof}
To show the two claims, it suffices to exhibit a coupling of the environments $(Y_u)_{u \in \N \times \Z^d}$ and $(Z_u)_{u \in \N \times \Z^d}$ such that $F = G$ in $\s$ and $\wt F = \wt G$ in $\R$.
So we let $H_f$ and $H_g$ denote the $\N$-supports of $f$ and $g$, respectively, and take
$\sigma : H_f \to H_g$ and $(x_n)_{n \in H_f}$ as in Corollary \ref{better_def_cor}, so that \eqref{better_def} holds.
Define $\psi : H_f \times \Z^d \to H_g \times \Z^d$ by $\psi(n,x) \coloneqq  (\sigma(n),x-x_n)$, and
let $Z_u$ be equal to $Y_{\psi^{-1}(u)}$ whenever $u \in H_g \times \Z^d$.
Otherwise, we may take $Z_u$ to be an independent copy of $Y_u$.
For $u \in H_f \times \Z^d$, we have by \eqref{better_def}:
\eeq{
\sum_{v \sim u} f(v)e^{\beta Y_u} = \sum_{v \sim u} g(\psi(v))e^{\beta Y_u}
= \sum_{v \sim \psi(u)} g(v)e^{\beta Z_{\psi(u)}}. \label{numerator_same}
}
Since $f(v) = 0$ for all $v \notin H_f \times \Z^d$, and similarly $g(v) = 0$ for all $v \notin H_g \times \Z^d$, we can sum over all of $\N \times \Z^d$, and \eqref{numerator_same} gives
\eq{
\sum_{w \in \N \times \Z^d} \sum_{v \sim w} f(v) e^{\beta Y_w}
= \sum_{w \in \N \times \Z^d} \sum_{v \sim w} g(v) e^{\beta Z_w}. 
}
This identity, together with the fact that $\|f\| = \|g\|$ (cf.~Lemma \ref{norm_equivalence}), shows $\wt F = \wt G$ and thus proves claim (b).
When we further consider \eqref{numerator_same}, we see that $F(u) = G(\psi(u))$ for all $u \in H_f \times \Z^d$.
Hence $d_\psi(F,G) = 0$.
Moreover, for any $\eps > 0$, we can find a finite subset $A \subset H_f \times \Z^d$ such that
\eq{
\sum_{u \notin A} F(u)^2 + \sum_{u \notin \psi(A)} G(u)^2 < \eps.
}
With $\phi \coloneqq  \psi\rvert_{A}$, we thus have $d(F,G) \leq d_\phi(F,G) < \eps$.
Letting $\eps$ tend to 0 gives claim (a).
\end{proof}

Given Proposition \ref{same_law}, we may define $Tf \in \p(\s)$ to be the law of $F$ given $f \in \s$.
The next step in our construction is to establish continuity. 
The seemingly unexciting result below requires a careful and lengthy proof that deserves a separate Section \ref{proof_continuity}.

\begin{prop} \label{continuous1} \hspace*{1ex}
\begin{itemize}
\item[(a)] $f\mapsto Tf$ is a continuous map from $(\s,d)$ to $(\p(\s),\w)$.
\item[(b)] For any positive integer $q$, $f\mapsto \EE(\log^q\wt F)$ is a continuous map from $(\s,d)$ to $\R$.
In particular, the case $q=1$ implies $R(\cdot)$ is continuous.
\end{itemize}
\end{prop}

The following consequence of part (b) is not new but will be convenient to have recorded for later use.

\begin{cor} \label{increment_cor}
Recall the martingale increment $d_n$ from \eqref{d_n_def}.
For any positive integer $q$, there exists a constant $C$ depending only on $\mathfrak{L}$, $\beta$, $d$, and $q$, such that $\EE(d_n^q) \leq C$ for every $n$.
\end{cor}

\begin{proof}
If we fix $f = f_n$ and take $\wt F$ as in \eqref{R_def}, then \eqref{Zfrac} tells us
\eq{
\EE\givenk[\Big]{\log^q\frac{Z_{n+1}}{Z_n}}{\f_n} = \EE\Big(\log^q \frac{\wt F}{2d}\Big).
 }
By the compactness of $\s$ and the continuity property of Proposition \ref{continuous1}(b), the above right-hand side is bounded by a constant depending only on $\mathfrak{L}$, $\beta$, $d$, and $q$.
The same is then true for $\EE(\log^q \frac{Z_{n+1}}{Z_n})$, and thus also for $\EE(d_n^q)$.
\end{proof}

To apply $f\mapsto Tf$ to the functions of interest, namely the endpoint distributions from \eqref{fn_def}, we must first identify $f_{n}$ with the partitioned subprobability measure having representative
\eeq{
f_n(u) = \begin{cases} 
f_n(x) &\text{if } u = (n,x), \\
0 &\text{else,}
\end{cases} \quad
u \in \N \times \Z^d. \label{fn_def2}
}
In review, $f_n$ is a function on $\Z^d$ that is random with respect to $\f_n$.
It is natural (and measurable) to identify $f_n$ with a function on $\N \times \Z^d$ that is supported on the $n^{\text{th}}$ copy of $\Z^d$; this is \eqref{fn_def2}.
Finally, that function --- thus far an element of $S$ --- is identified with its equivalence class $\iota(f_n) \in \s$, where $\iota$ is the quotient map from Lemma \ref{S_meas}.
In accordance with our previous notation, the symbol $f_n$ will henceforth denote the equivalence class in $\s$ unless stated otherwise, while $f_n(u)$ will denote the representative defined by \eqref{fn_def2}, evaluated at $u \in \N \times \Z^d$.
By the discussion preceding \eqref{F_predef}, the law of $f_{n+1} \in \s$ given $\f_n$ is equal to $Tf_n$.
For this reason we refer to $T$ as the ``update map".

%
%

\subsection{Lifting the update map} \label{extension_to_measures}
Thus far we have only considered the random variable $F$ in \eqref{F_def} for \textit{fixed} $f \in \s$.
If $f$ is itself a random element of $\s$ drawn from the probability measure $\mu \in \p(\s)$, the resulting total law of $F$ will be written $\t\mu$.
That is,
\eq{
\t\mu(\a) = \int Tf(\a)\ \mu(\dd f), \quad \text{Borel $\a \subset \s$}.
}
More generally, for $\vphi : \s\to\R$,
\eeq{
\int \vphi(g)\ \t\mu(\dd g) = \int_\s \int_\s \vphi(g)\ Tf(\dd g)\, \mu(\dd f), \label{T_fubini}
}
where the integral is well-defined if $\vphi$ is nonnegative, or if
\eq{
\int_\s \int_\s |\vphi(g)|\ Tf(\dd g)\,\mu(\dd f) < \infty. 
}

We can recover the map $T$ as 
$Tf = \t\delta_f$, where $\delta_f \in \p(\s)$ is the unit point mass at $f$.
From Proposition \ref{continuous1}, one can conclude by standard arguments that $\mu \mapsto \t\mu$ is (uniformly) continuous.
A proof is included in Appendix~\ref{continuous2_subsection}.
\begin{prop} \label{continuous2}
$\mu\mapsto\t\mu$ is a continuous map from $(\p(\s),\w)$ to $(\p(\s),\w)$.
\end{prop}

\section{The empirical measure of the endpoint distribution} \label{free_energy}

\subsection{Definition and properties} \label{empirical_measures}
As discussed in Section \ref{endpoint_distributions}, the law of the endpoint distribution $f_{n+1}$ given $\f_n$ is equal to $Tf_n$.
Now define the empirical probability measure on $\s$ generated by the $f_i$,
\eeq{
\mu_n \coloneqq  \frac{1}{n} \sum_{i = 0}^{n-1} \delta_{f_i}. \label{mu_n_def}
}
Then $\mu_n$ is a random element of $\p(\s)$, measurable with respect to $\f_n$. 
While we will be interested in the quantity $\w(\mu_n,\t\mu_n)$, it is easier to replace $\mu_n$ by 
the shifted empirical measure,
\eq{
\mu_n' \coloneqq  \frac{1}{n}\sum_{i=1}^{n} \delta_{f_i},
}
since $Tf_{i}$ is the distribution of $f_{i+1}$ given $\f_{i}$.

Making use of the dual formulation \eqref{kantorovich} of Wasserstein distance, one has
\eq{
\w(\mu_n',\t\mu_n) &= \sup_{\vphi} \biggl(\int \vphi(f)\ \mu_n'(\dd f) - \int \vphi(f)\ \t\mu_n(\mathrm{d}f) \biggr)\\
&= \sup_\vphi \frac{1}{n}\sum_{i = 0}^{n-1} \bigl(\vphi(f_{i+1}) - \EE\givenk{\vphi(f_{i+1})}{\f_{i}}\bigr),
}
where the supremum is taken over all functions $\vphi : \s \to \R$ satisfying 
\eeq{
|\vphi(f) - \vphi(g)| \leq d(f,g) \quad \text{for all $f,g \in \s$.} \label{lip1}
}
Notice, however, that adding a constant to $\vphi$ does not change the value of 
$\int \vphi(f)\, \mu'_n(\dd f) - \int \vphi(f)\, \t\mu_n(\dd f)$, and so the supremum can equivalently be taken over $\vphi$ satisfying $\vphi(\vc{0}) = 0$, where $\mathbf{0}$ denotes (the equivalence class of) the constant zero function.
Let us denote the set of such functions by
\eq{
\l = \{\vphi : \s \to \R : |\vphi(f) - \vphi(g)| \leq d(f,g) \text{ for all $f,g \in \s$},\ \vphi(\vc{0}) = 0\}.
}
Recall that the space of real-valued continuous functions on a compact metric space is equipped with the uniform norm,
\eq{
\|\vphi\|_\infty \coloneqq  \sup_{f \in \s} |\vphi(f)| < \infty.
}
For $\vphi \in \l$, the Lipschitz condition \eqref{lip1} and Lemma \ref{trivial_bound} imply
$\|\vphi\|_\infty \leq 2$.
In particular, $\l$ is a uniformly bounded family of continuous functions.
Furthermore, since $\l$ consists of Lipschitz functions whose minimal Lipschitz constants are uniformly bounded, it is both equicontinuous and closed under the topology induced by the uniform norm.
By the Arzel\`a--Ascoli Theorem (see Munkres \cite{munkres00}, Theorem 47.1), $\l$ is compact in this topology.
Having made this observation, we are now ready to prove the following convergence result.

\begin{prop} \label{primeT_as}
As $n \to \infty$, $\w(\mu_n',\t\mu_n) \to 0$ almost surely.
\end{prop}

We will use the following well-known fact.

\begin{lemma}[Burkholder--Davis--Gundy Inequality, see \cite{burkholder-davis-gundy72}, Theorem 1.1]
\label{bdg}
Let $(M_n)_{n \geq 0}$ be a martingale, and write
\eq{
M_n = \sum_{i = 0}^n d_i, \qquad \text{$d_0 = M_0$, $d_i = M_{i}-M_{i-1}$ for $i \geq 1$}.
}
Let $M_n^* \coloneqq  \sup_{0 \leq i \leq n} M_n$.
Then for any $q \geq 1$, there are positive constants $c_q$ and $C_q$ such that
\eq{
c_q\, \EE\bigg[\bigg(\sum_{i = 0}^n d_i^2\bigg)^{q/2}\bigg]
\leq \EE\big[(M_n^*)^q\big]
\leq C_q\, \EE\bigg[\bigg(\sum_{i = 0}^n d_i^2\bigg)^{q/2}\bigg] \quad \text{for all $n \geq 0$.}
}
\end{lemma}

\begin{proof}[Proof of Proposition \ref{primeT_as}]
We have
\eq{
\w(\mu_n',\t\mu_n) = \sup_{\vphi \in \l} \frac{1}{n} \sum_{i = 0}^{n-1} \bigl(\vphi(f_{i+1}) - \EE\givenk{\vphi(f_{i+1})}{\f_i}\bigr).
}
Notice that for any fixed $\vphi \in \l$,
\eq{
M_{n}(\vphi) \coloneqq  \sum_{i = 0}^{n-1} \bigl(\vphi(f_{i+1}) - \EE\givenk{\vphi(f_{i+1})}{\f_i}\bigr)
}
defines a martingale $(M_n(\vphi))_{n \geq 0}$ adapted to the filtration $(\f_{n})_{n \geq 0}$.
By Lemma \ref{bdg}, there is a constant $C = C(\vphi)$ such that
\eq{
\EE\big[M_n(\vphi)^4\big] &\leq \EE\big[M_n^*(\vphi)^4\big]\\
& \leq C\, \EE\bigg[\bigg(\sum_{i = 0}^{n-1} \big(\vphi(f_{i+1}) - \EE\givenk{\vphi(f_{i+1})}{\f_i}\big)^2\bigg)^2\bigg]
\leq 16Cn^2,  
}
where the final inequality follows from \eqref{lip1} and Lemma \ref{trivial_bound}.
A Markov bound now gives
\eq{
\sum_{n = 0}^\infty \PP\Big(\frac{|M_n(\vphi)|}{n} > n^{-1/5}\Big)
&= \sum_{n = 0}^\infty \PP\big(M_n(\vphi)^4 > n^{16/5}) \\
&\leq \sum_{n = 0}^\infty n^{-16/5}\EE\big[M_n(\vphi)^4\big] \\
&\leq \sum_{n = 0}^\infty 16Cn^{-6/5}
< \infty.
}
By Borel--Cantelli, we may conclude
\eeq{
\lim_{n \to \infty} \frac{|M_n(\vphi)|}{n} = 0 \quad \mathrm{a.s.} \label{phi_MG}
}
As discussed above, $\l$ is compact in the uniform norm topology.
In particular, it is separable.
Let $\vphi_1,\vphi_2,\dots$ be a countable, dense subset of $\l$.
Because of \eqref{phi_MG}, we can say that with probability one,
\eeq{
\lim_{n \to \infty} \frac{M_n(\vphi_j)}{n} = 0 \quad \text{for all $j \geq 1$.} \label{dense_to_0}
}
Assume that this almost sure event occurs.
Again from \eqref{lip1} and Lemma \ref{trivial_bound}, we know
\eq{
\|\vphi - \vphi'\|_\infty < \eps \quad \implies \quad \bigg|\frac{M_n(\vphi)}{n} - \frac{M_n(\vphi')}{n}\bigg| < 2\eps,
}
meaning $(M_n(\cdot)/n)_{n \geq 0}$ is an equicontinuous sequence of functions on the compact metric space $\l$.
The assumption \eqref{dense_to_0} says that this family converges pointwise to $0$ on a dense subset.
The Arzel\`a--Ascoli Theorem forces this convergence to be uniform.
That is, for any $\eps > 0$, there is $N$ large enough that
\eq{
n \geq N \quad \implies \quad \w(\mu_n',\t\mu_n) = \sup_{\vphi \in \l} \frac{M_n(\vphi)}{n} < \eps.
}
We conclude that $\w(\mu_n',\t\mu_n)$ tends to 0 as $n \to \infty$.
As this holds given the almost sure event \eqref{dense_to_0}, we are done.
\end{proof}

\subsection{Convergence to fixed points of the update map} \label{convergence_fixed}
Proposition \ref{primeT_as} suggests that for large $n$, the empirical measure will be close to the set of fixed points of $\t$:
\eeq{
\k \coloneqq  \{\nu \in \p(\s) : \t\nu = \nu\}. \label{K_def}
}
For $\u \subset \p(\s)$, we will denote distance to $\u$ by
\eq{
\w(\mu,\u) \coloneqq  \inf_{\nu \in \u} \w(\mu,\nu), \quad \mu \in \p(\s).
}

\begin{cor} \label{close_probability}
As $n \to \infty$, $\w(\mu_n,\k) \to 0$ almost surely.
\end{cor}

\begin{proof}
Recall that $\vc{0}$ is the element of $\s$ whose unique representative in $S$ is the constant zero function.
Notice that $\t\delta_\vc{0} = \delta_\vc{0}$ so that $\k$ is nonempty.
Next observe that from Lemma \ref{trivial_bound}, we have the trivial bound $\w(\mu_n,\mu_{n}') \leq 2/n$, and so Proposition \ref{primeT_as} immediately implies $\w(\mu_n,\t\mu_n) \to 0$ almost surely.
On this almost sure event, it follows that $\w(\mu_n,\k) \to 0$, since otherwise there would exist a subsequence $\mu_{n_k}$ remaining a fixed positive distance away from $\k$.
By compactness, we could assume $\mu_{n_k}$ converges to some $\mu$, but then continuity of $\t$ would force $\w(\mu,\t\mu) = 0$; that is, $\mu_{n_k}$ converges to an element of $\k$, which is a contradiction.
%
\end{proof}

Now that the set $\k$ is seen to contain all possible limits of the empirical measure, we should like to have some description of its elements.
One suggestive fact proved below is that any measure in $\k$ places all its mass on those elements of $\s$ with norm 0 or 1.
This observation will be crucial in proving our characterization of the low temperature phase in Section \ref{empirical_limits}.

\begin{prop} \label{no_middle}
If $\nu \in \k$, then
\eq{
\nu(\{f \in \s : 0 < \|f\| < 1\}) = 0.
}
\end{prop}

\begin{proof}
First take $f \in \s$ to be non-random.
Then $Tf$ is the law of the random function
\eq{
F(u) = \frac{\sum_{v \sim u} f(v) e^{\beta Y_u}}{\sum_{w \in \N \times \Z^d} \sum_{v \sim w} f(v) e^{\beta Y_w} + 2d(1 - \|f\|){e^{\lambda(\beta)}}}.
}
If $\|f\| = 0$ or $\|f\| = 1$, then $\|F\| = \|f\|$.
If instead $0 < \|f\| < 1$, then $\|F\|$ is random and still satisfies $0 < \|F\| < 1$.
By summing over $u \in \N \times \Z^d$, we have
\eq{
\EE(\|F\|) = \EE\bigg[\frac{\sum_{u \in \N \times \Z^d} \sum_{v \sim u} f(v) e^{\beta Y_u}}{\sum_{u \in \N \times \Z^d} \sum_{v \sim u} f(v) e^{\beta Y_u} + 2d(1-\|f\|){e^{\lambda(\beta)}}}\bigg].
}
Upon observing that for any constant $C > 0$, the mapping $t \mapsto \frac{t}{t+C}$
is strictly concave, we deduce from Jensen's inequality that
\eq{
\EE(\|F\|) &\leq \frac{\EE \big[\sum_{u \in \N \times \Z^d} \sum_{v \sim u} f(v) e^{\beta Y_u}\big]}{\EE\big[\sum_{u \in \N \times \Z^d} \sum_{v \sim u} f(v) e^{\beta Y_u}\big] + 2d(1-\|f\|){e^{\lambda(\beta)}}} \\
&= \frac{2d\, \|f\| \cdot e^{\lambda(\beta)}}{2d\cdot e^{\lambda(\beta)}} = \|f\|,
}
where equality holds if and only if $\sum_{u \in \N \times \Z^d} \sum_{v \sim u} f(v) e^{\beta Y_u}$ is an almost sure constant.
Since the disorder distribution $\mathfrak{L}$ is not a point mass, this is not the case.

Now let $\nu \in \p(\s)$, and take $f \in \s$ to be random with law $\nu$.
If $\nu$ assigns positive measure to the set $\{f \in \s : 0 < \|f\| < 1\}$, then the above argument gives
\eq{
\int \|g\|\ \t\nu(\dd g) = \iint \|F\|\ Tf(\dd F)\, \nu(\dd f) < \int \|f\|\ \nu(\dd f).
}
But when $\nu \in \k$ 
so that $\t\nu = \nu$, the first and last expressions above are equal. 
We conclude that any $\nu \in \k$ must assign mass 0 to the set $\{f \in \s : 0 < \|f\| < 1\}$.
\end{proof}

\subsection{Variational formula for the free energy} \label{calculations}
In order to connect the results of Section \ref{convergence_fixed} to the temperature conditions of Theorems \ref{intro_result1} and \ref{intro_result2}, we will need to relate the free energy $F_n = \frac{1}{n}\log Z_n$ to the abstract objects we have introduced.
Recall the Doob decomposition from \eqref{R_predef}--\eqref{doob_decomposition}:
\eq{
\log Z_n = M_n + A_n, \qquad \text{where} \qquad M_n = \sum_{i=0}^{n-1} d_i, \quad  A_n = \sum_{i=0}^{n-1} R(f_i).
}
If we lift the functional $R : \s\to\R$ of \eqref{R_def} to $\r:\p(\s)\to\R$ by defining
\eeq{
\r(\mu) \coloneqq  \int R(f)\ \mu(\dd f), \quad \mu\in\p(\s), \label{R_def2}
}
then we can conveniently rewrite the above decomposition as
\eeq{ \label{Fn_ito_empirical}
F_n = \frac{\log Z_n}{n} = \frac{M_n}{n} + \r(\mu_n),
}
where $\mu_n$ is the empirical measure from \eqref{mu_n_def}.
Our next result, which says the mean-zero martingale $(M_n)_{n\geq0}$ has a vanishing contribution in the above expression, is merely a reinterpretation of arguments appearing in earlier works (e.g.~\cite{carmona-hu02,comets-shiga-yoshida03,comets-shiga-yoshida04,carmona-hu06}).

\begin{prop} \label{FR_prop}
As $n \to \infty$, $|F_n - \r(\mu_{n})| \to 0$ almost surely.
\end{prop}

\begin{proof}
By Corollary \ref{increment_cor}, we have $\EE(d_i^4) \leq C_1$ for some constant $C_1$ independent of $i$.
From Lemma \ref{bdg}, we deduce that
\eq{
\EE(M_n^4) \leq C_2\, \EE\bigg[\bigg(\sum_{i = 0}^{n-1} d_i^2 \bigg)^2\bigg]
\leq C_2\sum_{i = 0}^{n-1} \sum_{j = 0}^{n-1} \sqrt{\EE(d_i^4)\EE(d_j^4)}
\leq Cn^2,
}
where the constant $C = C_1C_2$ is independent of $n$.
It follows that \linebreak $\EE\big[(n^{-1}M_n)^4] \leq Cn^{-2}$.
As in the proof of Proposition \ref{primeT_as}, an argument using Markov's inequality and Borel--Cantelli shows
\eq{
\lim_{n \to \infty} |F_n - \r(\mu_{n})| = \lim_{n \to \infty} \frac{|M_n|}{n} = 0 \quad \mathrm{a.s.}
}
\end{proof}

Given the convergence of $\mu_{n}$ to the set $\k$, one should expect $\EE(F_n) = \EE\,\r(\mu_{n})$ to become close to $\r(\nu)$ for some $\nu \in \k$.
One difficulty is that $\mu_{n}$ does not converge to a particular $\nu \in \k$, but rather becomes arbitrarily close to the set $\k$. 
Nevertheless, we can instead consider the subset
\eeq{
\m = \Big\{\nu_0 \in \k : \r(\nu_0) = \inf_{\nu \in \k} \r(\nu)\Big\}, \label{M_def}
}
and show convergence to $\m$.
By Proposition \ref{continuous1}(b) and Lemma \ref{portmanteau}, $\r$ is continuous. 
Therefore, since $\k$ is compact (being a closed subset of a compact metric space), the set $\m$ is nonempty.
Moreover, $\m$ is a closed subset of the compact space $\k$, and so $\m$ is compact.
The first step in proving the desired convergence is the following consequence of Corollary \ref{close_probability}.

\begin{prop} \label{lower_bound}
Let $\k$ be defined by \eqref{K_def}.
Let $\r : \p(\s) \to \R$ be defined by \eqref{R_def2}.
Then
\eeq{
\liminf_{n \to \infty} F_n \geq \inf_{\nu \in \k} \r(\nu) \quad \mathrm{a.s.} \label{lower_bound_as}
}
In particular,
\eeq{
\liminf_{n \to \infty} \EE(F_n) \geq \inf_{\nu \in \k} \r(\nu). \label{lower_bound_eq}
}
\end{prop}

\begin{proof}
$\r$ is continuous on the compact metric space $(\p(\s),\w)$, and thus uniformly continuous.
Therefore, a simple argument using Corollary \ref{close_probability} gives
\eeq{
\liminf_{n \to \infty} \r(\mu_{n}) \geq \inf_{\nu \in \k} \r(\nu) \quad \mathrm{a.s.} \label{lower_bound_asR}
}
Since $\r$ is bounded, we may apply Fatou's Lemma to obtain
\eeq{
\liminf_{n \to \infty} \EE\, \r(\mu_n) \geq \EE\Big[\liminf_{n \to \infty} \r(\mu_{n})\Big] \geq \inf_{\nu \in \k} \r(\nu). \label{lower_bound_eqR}
}
Now \eqref{lower_bound_as} follows from \eqref{lower_bound_asR} and Proposition \ref{FR_prop}, and \eqref{lower_bound_eq} follows from \eqref{lower_bound_eqR} and the fact that $\EE(F_n) = \EE\, \r(\mu_n)$.
\end{proof}

Following Proposition \ref{lower_bound}, we naturally ask if there is a matching upper bound.
The next result answers this question in the affirmative.
To state the full theorem, we need to denote one element of $\s$ in particular.
Notice that for $f \in S$ satisfying $f(u) = 1$ for some $u \in \N \times \Z^d$, Proposition \ref{superisometry} implies $d(f,g) = 0$ for $g \in S$ if and only if $g(v) = 1$ for some $v \in \N \times \Z^d$.
We can thus define the element $\vc{1} \in \s$ whose representatives in $S$ are the norm-1 functions supported on a single point.

\begin{thm} \label{upper_bound}
Let $\k$ be defined by \eqref{K_def}.  
Let $\r : \p(\s) \to \R$ be defined by \eqref{R_def2}.
Then
\eq{
\limsup_{n \to \infty} \EE(F_n) \leq \inf_{\nu \in \k} \r(\nu),
}
and so
\eeq{
\lim_{n \to \infty} F_n = \inf_{\nu \in \k} \r(\nu) \quad \mathrm{a.s.} \label{limit}
}
The minimum value is equal to
\eeq{
\inf_{\nu \in \k} \r(\nu) = \lim_{n \to \infty} \frac{1}{n} \sum_{i = 0}^{n-1} \r(\t^i\delta_{\vc{1}}). \label{minimal_value}
}
\end{thm}
The key ingredient in the proof of Theorem \ref{upper_bound} is the following lemma.
When $\|f_0\|=1$, the result can be regarded as Jensen's inequality applied to a weighted sum of partition functions, each one corresponding to a different starting vertex chosen at random according to $f_0$.

\begin{lemma} \label{upper_bound_lemma}
For any $f_0 \in \s$ and $n \geq 1$,
\eq{
\sum_{i = 0}^{n-1} \r(\t^i \delta_{f_0}) \geq \EE\log Z_n,
}
where $\delta_{f_0} \in \p(\s)$ is the unit mass at $f_0$.
Equality holds if and only if $f_0 = \vc{1}$.
\end{lemma}

\begin{proof}
If $f_0 = \vc 0$, then $\r(\t^i\delta_{\vc 0}) = \r(\delta_{\vc 0}) = \lambda(\beta)$, and so the desired inequality is immediate from \eqref{jensen_applied_finite}.
So we may assume $f_0 \neq \vc 0$.
Fix a representative $f_0 \in S$.
Let $(Y^{(i)}_u)_{u \in \N \times \Z^d}$, $1 \leq i \leq n$, be independent collections of i.i.d.~random variables with law $\mathfrak{L}$.
For $1 \leq i \leq n-1$, inductively define $f_i \in \s$ to have representative
\eq{
f_i(u) = \frac{\sum_{v \sim u} f_{i-1}(v) e^{\beta Y^{(i)}_u}}{ \sum_{w \in \N \times \Z^d}\sum_{v \sim w} f_{i-1}(v) e^{\beta Y^{(i)}_w} + 2d(1-\|f_{i-1}\|){e^{\lambda(\beta)}}},
}
so that the law of $f_i$ is equal to $\t$ applied to the law of $f_{i-1}$.
By induction, then,
the law of $f_i$ is $\t^i \delta_{f_0}$.
By definitions \eqref{R_def2} and \eqref{R_def},
\eq{
\r(\t^i \delta_{f_0}) 
&= \EE \log \wt F_{i} - \log 2d,
}
where 
\eq{
\wt F_i &\coloneqq  \sum_{u_{i+1} \in \N \times \Z^d}\sum_{u_{i} \sim u_{i+1}} f_{i}(u_{i})e^{\beta Y^{(i+1)}_{u_{i+1}}} + 2d(1 - \|f_{i}\|){e^{\lambda(\beta)}}, 
}
and the expectation is over both $f_i$ and the $Y^{(i+1)}$ variables.
Repeatedly rewriting $f_i$ in terms of $f_{i-1}$ leads to the identity
\eq{
&\wt F_0\wt F_1\cdots \wt F_{n-1} \\
&= \sum_{u_n\in\N\times\Z^d}\sum_{u_0\sim u_1\sim\cdots\sim u_n} f_0(u_0)e^{\beta\sum_{i=1}^nY_{u_i}^{(i)}} + (2d)^n(1-\|f_0\|){e^{n\lambda(\beta)}} \\
&= \sum_{u_0 \in \N \times \Z^d} \frac{f_0(u_0)}{\|f_0\|}\bigg[\|f_0\|\sum_{u_{n} \sim \cdots \sim u_0}e^{\beta \sum_{i = 1}^n Y^{(i)}_{u_i}} + (2d)^n(1 - \|f_0\|)\EE(Z_{n})\bigg].
}
Using the concavity of the $\log$ function in three consecutive steps, we deduce
\eq{
&\log \wt F_0\wt F_1\cdots \wt F_{n-1} \\
&\geq 
\sum_{u_0} \frac{f_0(u_0)}{\|f_0\|} \log \bigg[\|f_0\|\sum_{u_{n} \sim \cdots \sim u_0}e^{\beta \sum_{i = 1}^n Y^{(i)}_{u_i}}+ (2d)^n(1 - \|f_0\|)\EE(Z_{n})\bigg] \\
&\geq 
\sum_{u_0} \frac{f_0(u_0)}{\|f_0\|} \bigg[\|f_0\|\log\sum_{u_{n} \sim \cdots \sim u_0}e^{\beta \sum_{i = 1}^n Y^{(i)}_{u_i}} +(1 - \|f_0\|)\log\big((2d)^n\EE(Z_{n})\big)\bigg] \\
&\geq 
\sum_{u_0} \frac{f_0(u_0)}{\|f_0\|} \bigg[\|f_0\|\log\sum_{u_{n} \sim \cdots \sim u_0}e^{\beta \sum_{i = 1}^n Y^{(i)}_{u_i}} +(1 - \|f_0\|)\EE\log\big((2d)^nZ_{n}\big)\bigg],
}
where equality holds throughout if and only if $f_0(u_0) = 1$ for some $u_0 \in \N\times\Z^d$.
Since the random variable $\sum_{u_{n} \sim \cdots \sim u_0}\exp\big(\beta \sum_{i = 1}^n Y^{(i)}_{u_i}\big)$ is equal in law to $(2d)^n Z_n$ for any fixed $u_0 \in \N \times \Z^d$, taking expectation yields
\eq{
&\EE \log \wt F_0\wt F_1\cdots \wt F_{n-1} \\
&\geq \sum_{u_0} \frac{f_0(u_0)}{\|f_0\|} \Big(\|f_0\|\, \EE \log\big((2d)^n Z_n) + (1 - \|f_0\|)\, \EE\log\big((2d)^nZ_n\big)\Big) \\
&= \EE\log\big((2d)^nZ_n\big).
}
It follows that
\eq{
\sum_{i = 0}^{n-1}  \r(\t^i \delta_{f_0})
&= \sum_{i = 0}^{n-1}\big(\EE\log \wt F_{i} - \log 2d\big)\\
&= \EE \log(\wt F_0\wt F_1\cdots \wt F_{n-1}) - \log (2d)^n
\geq \EE\log Z_n,
}
with equality if and only if $f_0 = \vc{1}$.
\end{proof}

\begin{proof}[Proof of Theorem \ref{upper_bound}]
Consider any $\nu \in \k$.
Using the fact that $\t\nu = \nu$ and Lemma \ref{upper_bound_lemma}, we have
\eq{
\r(\nu) = \frac{1}{n} \sum_{i = 0}^{n-1} \r(\t^i \nu) 
&= \int \frac{1}{n} \sum_{i = 0}^{n-1} \r(\t^i\delta_f)\ \nu(\dd f) \\
&\geq \int \frac{1}{n}\, \EE \log Z_n\ \nu(\dd f)
= \EE(F_n).
}
We note that linearity of $\t$ and Fubini's theorem \eqref{T_fubini} have been applied, which is permissible since Proposition \ref{continuous1}(b) shows that $R$ is continuous on $\s$ and hence bounded.

As the above estimate holds for every $\nu \in \k$ and every $n$, we have
\eeq{
\limsup_{n \to \infty} \EE(F_n) \leq \inf_{\nu \in \k} \r(\nu). \label{upper_bound_eq}
}
It now follows from \eqref{lower_bound_eq} and \eqref{upper_bound_eq} that
\eeq{
\lim_{n \to \infty} \EE(F_n) = \inf_{\nu \in \k} \r(\nu). \label{expectations_converge}
}
Then \eqref{Fn_lim} improves \eqref{expectations_converge} to \eqref{limit}.
Finally, equation \eqref{minimal_value} follows from the final statement in Lemma~\ref{upper_bound_lemma}.
\end{proof}

We now turn to strengthening Corollary \ref{close_probability} by proving convergence not only to $\k$, but to the smaller set $\m$, defined in \eqref{M_def}.

\begin{thm} \label{close_M}
As $n \to \infty$, $\w(\mu_n,\m) \to 0$ almost surely.
\end{thm}

\begin{proof}
As $n\to\infty$, $\w(\mu_n,\k) \to 0$ almost surely (by Corollary \ref{close_probability}), and $\r(\mu_n)$ converges almost surely to $\inf_{\nu\in\k}\r(\nu)$ (by Proposition \ref{FR_prop} and Theorem \ref{upper_bound}).
Therefore, by continuity of $\r$ and compactness of $\p(\s)$, we have $\w(\mu_n,\k_\delta) \to 0$ almost surely, where $\delta$ is any positive number, and
\eq{
\k_\delta \coloneqq \Big\{\nu_0 \in \k : \r(\nu_0) < \inf_{\nu\in\k} \r(\nu)+\delta\Big\}.
} 

Now, given any $\eps > 0$, we can choose $\delta$ such that $\sup_{\nu\in\k_\delta}\w(\nu,\m) < \eps$.
(Indeed, if this were not the case, then one could find a sequence $(\nu_k)_{k \geq 1}$ in $\k$ such that
$\r(\nu_k) \searrow \inf_{\nu \in \k} \r(\nu)$ as $k \to \infty$, but $\w(\nu_k,\m) \geq \eps$ for all $k$.
Since $\k$ is compact, we may pass to a subsequence and assume $\nu_k$ converges to some $\nu_0 \in \k$.
In particular, $\w(\nu_0,\m) \geq \eps$.
But the continuity of $\r$ implies $\r(\nu_0) = \inf_{\nu \in \k} \r(\nu)$, meaning $\nu \in \m$, a contradiction.)
As $\w(\mu_n,\k_\delta) \to 0$ almost surely and $\eps > 0$ is arbitrary, we are done.
\end{proof}

\section{Limits of empirical measures} \label{empirical_limits}
In the first part of this section, we give a characterization of the low temperature regime in terms of the fixed points of the update map $\t$.
This is stated as Theorem \ref{characterization}(b). 
We know from Theorem \ref{close_M} that those fixed points minimizing the energy functional $\r$ constitute the possible limits of the empirical measure $\mu_n$, and so this characterization will allow us to study the Ces\`aro asymptotics of endpoint distributions.

In Section \ref{example_application}, we make this connection more concrete by outlining the general steps by which the theory developed in the previous three sections can be used to prove statements about directed polymers.
Indeed, the strategy described will be employed in Sections \ref{main_thm} and \ref{main_thm2}.
But to provide a more brief, instructional example in the present section, we give a new proof of Theorem \ref{characterization0},
which offered a first characterization of the low temperature phase.
Once appropriate definitions are made, only a short argument is needed to prove the statement.


\subsection{Characterizing high and low temperature phases} \label{norm_monotonicity}
Recall that $\vc{0}$ is the element of $\s$ whose unique representative in $S$ is the constant zero function,
and $\vc{1}$ is the element of $\s$ whose representatives in $S$ are the norm-1 functions supported on a single point.

\begin{lemma} \label{unique_max}
The function $R : \s \to \R$ from \eqref{R_def} achieves a unique minimum $R(\vc{1}) = \EE\log Z_1$,
and a unique maximum $R(\vc{0}) = \lambda(\beta)$.
\end{lemma}

\begin{proof}
The first statement is immediate from Lemma \ref{upper_bound_lemma} by taking $n = 1$.
The second statement follows from the Jensen inequality $R(f) \leq \log \EE(\wt F) -\log 2d= \lambda(\beta)$, where $\wt F$ is defined in \eqref{R_def}.
Moreover, the inequality is strict whenever $f \neq \vc{0}$, since then $\wt{F}$ is not an almost sure constant.
\end{proof}

Now we characterize the high and low temperature regimes by the elements of $\k$ and $\m$, defined by \eqref{K_def} and \eqref{M_def}, respectively. 
Notice that $\delta_{\vc 0}$ is always an element of $\k$; the temperature regime is determined by whether it is also an element of $\m$.

\begin{thm} \label{characterization}
Assume \eqref{mgf}. 
The following statements hold:
\begin{itemize}
\item[(a)] If $0 \leq \beta \leq \beta_{\mathrm{c}}$, then $\k = \m = \{\delta_{\vc{0}}\}$.
\item[(b)] If $\beta > \beta_{\mathrm{c}}$, then $\nu(\{f \in \s : \|f\| = 1\}) = 1$ for every $\nu \in \m$, and so $\t$ has more than one fixed point.
\end{itemize}
\end{thm}

\begin{proof}
Since the hypotheses of (a) and (b) are complementary, it suffices to prove their converses.
We always have $\t\delta_{\vc{0}} = \delta_{\vc{0}}$.
If $\k$ contains no other elements of $\p(\s)$, then $\m = \{\delta_{\vc{0}}\}$ and Theorem \ref{upper_bound} shows
\eq{
p(\beta) =  \lim_{n \to \infty} \EE(F_n) = \r(\delta_\vc{0}) = R(\vc{0}) = \lambda(\beta).
}
That is, $0 \leq \beta \leq \beta_{\mathrm{c}}$ when $\k = \m = \{\delta_{\vc{0}}\}$.

If instead there exists $\nu \in \k$ distinct from $\delta_\vc{0}$, then Proposition \ref{no_middle} implies that $\nu$ assigns positive mass to the set
\eq{
\u \coloneqq  \{f \in \s : \|f\| = 1\}, 
}
which is measurable by Lemma \ref{norm_equivalence}.
Moreover, Proposition \ref{no_middle} guarantees $\nu(\{\vc{0}\}) = 1 - \nu(\u)$,
and Lemma \ref{unique_max} shows $R(f) < R(\vc{0})$ for all $f \in \u$.
It follows that
\eq{
\lim_{n \to \infty} \EE(F_n) \leq \r(\nu) = \int R(f)\ \nu(\dd f) < R(\vc{0}).
}
In this case we have $p(\beta) < \lambda(\beta)$, meaning $\beta > \beta_{\mathrm{c}}$.
Furthermore, we can consider the probability measure
\eq{
\nu_\u(\a) \coloneqq  \frac{\nu(\a \cap \u)}{\nu(\u)}, \quad \text{Borel } \a \subset \s.
}
Notice that $\u$ and $\s\setminus\u$ are both closed under $\t$:
\eq{
f \in \u \quad \implies \quad Tf(\u) = 1, \\
f \notin \u \quad \implies \quad Tf(\u)  = 0. 
}
Indeed, this observation was made in the proof of Proposition \ref{no_middle}.
Therefore, $\t\nu = \nu$ implies $\nu_\u$ is also an element of $\k$.
If $\nu(\u) < 1$, then $\nu_\u$ satisfies
\eq{
\r(\nu_\u) = \frac{1}{\nu(\u)} \int_\u R(f)\ \nu(\dd f)
&= \int_\u R(f)\ \nu(\dd f) + \frac{1 - \nu(\u)}{\nu(\u)} \int_\u R(f)\ \nu(\dd f) \\
&< \int_\u R(f)\ \nu(\dd f) + \big(1 - \nu(\u)\big)R(\vc{0}) = \r(\nu).
}
It follows that $\nu \in \m$ only if $\nu(\u) = 1$.
\end{proof}

\subsection{An illustrative application} \label{example_application}
To demonstrate how the abstract setup can be employed to prove results on directed polymers, we will use it to prove Theorem \ref{characterization0}, which is restated here.
 
\begin{thm} \label{reprove_equivalence}
Assume \eqref{mgf}.
Then the following statements hold:
\begin{itemize}
\item[(a)] If $0 \leq \beta \leq \beta_{\mathrm{c}}$, then
\eq{
\lim_{n \to \infty} \frac{1}{n} \sum_{i = 0}^{n-1} \max_{x \in \Z^d} f_i(x) = 0 
 \quad \mathrm{a.s.}
 }
\item[(b)] If $\beta > \beta_{\mathrm{c}}$, then there exists $c > 0$ such that
\eeq{
\liminf_{n \to \infty} \frac{1}{n} \sum_{i = 0}^{n-1}  \max_{x \in \Z^d} f_i(x) \geq c \quad \mathrm{a.s.} \label{liminf_max}
}
 \end{itemize}
\end{thm}
The first step is to define the functional(s) relevant to the problem.
At first, definitions are made on the space $S$.
For instance, from \eqref{liminf_max} we are motivated to define $\max : S \to [0,1]$ by
\eq{
\max(f) \coloneqq  \max_{u \in \N \times \Z^d} f(u).
}
By appealing to Corollary \ref{defined_pspm}, we see that the functional is well-defined on the quotient space $\s$.
This should be true in order for the abstract machinery to be applicable.
Notice that when $f_i$ is seen as an element of $\s$,
\eeq{
\frac{1}{n} \sum_{i = 0}^{n-1} \max_{x \in \Z^d} f_i(x) = \int \max(f)\ \mu_{n}(\dd f). \label{max_abstract}
}
The second step is to prove a continuity condition so that Portmanteau (Lemma \ref{portmanteau}) may be applied.
Therefore, it is important to understand the topology induced by the metric $d$.
Both here and in later sections, this step is the most technical part of the process, but is generally straightforward and follows a similar proof strategy.

\begin{lemma} \label{max_lemma}
The function $\max : \s \to [0,1]$ given by
\eq{
\max(f) \coloneqq  \max_{u \in \N \times \Z^d} f(u)
}
is continuous and thus measurable.
\end{lemma}

\begin{proof}
Let $f \in \s$ be given, and fix a representative in $S$.
Choose $u \in \N \times \Z^d$ such that $f(u) = \max_{v \in \N \times \Z^d} f(v)$.
If $f(u) = 0$, then $f$ is identically zero, and
\eq{
d(f,g) < \delta \quad &\implies \quad \exists\ \phi : A \to \N \times \Z^d,\quad
\sum_{v \in A} g(\phi(v)) + \sum_{v \notin \phi(A)} g(v)^2 < \delta \\
&\implies \quad \max_{v \in \N \times \Z^d} g(v) < \max\{\delta,\sqrt{\delta}\}.
}
Otherwise, for any $\delta \in (0,f(u)^2)$, we have
\eq{
d(f,g) < \delta \quad &\implies \quad \exists\ \phi : A \to \N \times \Z^d,\quad
d_\phi(f,g) < \delta < f(u)^2 \\
&\implies \quad u \in A \text{ and } |f(v) - g(\phi(v))| < \delta \text{ for all $v \in A$}, \\
&\qquad \qquad \text{ and }
g(v)^2 < \delta \text{ for all $v \notin \phi(A)$}  \\
&\implies \quad \Big|\max_{v \in \N \times \Z^d} g(v) - \max_{v \in \N \times \Z^d} f(v)\Big| < \max\{\delta,\sqrt{\delta}\}.
}
From these two cases, we conclude that $f \mapsto \max f$ is continuous on $\s$.
\end{proof}

After the appropriate functionals have been defined, their limiting behavior (in a Ces\`aro sense) can be determined by applying Theorem \ref{characterization}.
Consequently, the next step in our program is to study how the functionals of interest behave according to the elements of $\m$.
This will depend on the value of $\beta$.
For instance, in the high temperature phase, $\m$ consists of only the point mass $\delta_\vc{0}$ at the zero function $\vc{0}$, and so trivially we have
\eeq{
0 \leq \beta \leq \beta_{\mathrm{c}} \quad \implies \quad \int \max(f)\ \nu(\dd f) = \max(\vc{0}) = 0 \quad \text{for all $\nu \in \m$.} \label{max_high}
}
On the other hand, in the low temperature phase, every $\nu \in \m$ is supported on those $f \in \s$ with $\|f\| = 1$.
So clearly
\eq{
\beta > \beta_{\mathrm{c}} \quad \implies \quad \int \max(f)\ \nu(\dd f) > 0 \quad \text{for all $\nu \in \m$.}
}
Furthermore, because of Lemma \ref{portmanteau}, Lemma \ref{max_lemma} shows $\nu \mapsto \int \max(f)\, \nu(\dd f)$ is continuous on the compact set $\m$, and so we actually have
\eeq{
\beta > \beta_{\mathrm{c}} \quad \implies \quad \int \max(f)\ \nu(\dd f) \geq c \quad \text{for all $\nu \in \m$,} \label{max_low}
}
for some $c > 0$.
Now we are poised to prove the desired result.
The final step is to interpret the above observations in terms of the directed polymer model, via the almost sure convergence to $\m$ that is stated in Theorem \ref{close_M}.
It is useful to remember that $\s$ and $\p(\s)$ are compact spaces.

\begin{proof}[Proof of Theorem \ref{reprove_equivalence}]
By the (uniform) continuity of the map 
\eq{
\mu \mapsto \int\max(f)\, \mu(\dd f)
}
on the compact space $\p(\s)$, we can find for any $\eps > 0$ some $\delta > 0$ such that
\eq{
\w(\mu,\m) < \delta \quad &\implies \quad 
\inf_{\nu \in \m} \int \max(f)\ \nu(\dd f) - \eps\\
&\qquad \qquad 
\leq \int \max(f)\ \mu(\dd f)
\leq \sup_{\nu \in \m} \int \max(f)\ \nu(\dd f) + \eps. 
}
Thus Theorem \ref{close_M} implies that almost surely,
\eq{
\inf_{\nu \in \m} \int \max(f)\ \nu(\dd f)
&\leq \liminf_{n \to \infty} \int \max(f)\ \mu_{n}(\dd f) \\
&\leq \limsup_{n \to \infty} \int \max(f)\ \mu_{n}(\dd f)
\leq \sup_{\nu \in \m} \int \max(f)\ \nu(\dd f).
}
When $0 \leq \beta \leq \beta_{\mathrm{c}}$, \eqref{max_high} says that the infimum and supremum appearing above are both equal to 0, and so
\eeq{
\lim_{n \to \infty} \int \max(f)\ \mu_{n}(\dd f) = 0 \quad \mathrm{a.s.} \label{max_claim2}
}
When $\beta > \beta_{\mathrm{c}}$, \eqref{max_low} shows that the infimum is bounded below by $c > 0$, in which case we have
\eeq{
\liminf_{n \to \infty} \int \max(f)\ \mu_{n}(\dd f) \geq c \quad \mathrm{a.s.} \label{max_claim1}
}
Recalling \eqref{max_abstract}, we see that \eqref{max_claim2} and \eqref{max_claim1} are exactly what we wanted to show.
\end{proof}

\section{Asymptotic pure atomicity} \label{main_thm}
Following the approach outlined in Section \ref{example_application}, we will prove that directed polymers are asymptotically purely atomic if and only if the parameter $\beta$ falls in the low temperature regime.
We say that the sequence of endpoint probability mass functions $(f_i)_{i \geq 0}$ is \textit{asymptotically purely atomic} 
if for every sequence $(\eps_i)_{i \geq 0}$ tending to $0$ as $i \to \infty$, we have
\eq{ 
\lim_{n \to \infty} \frac{1}{n} \sum_{i = 0}^{n-1} \rho_i(\omega_i \in \a_i^{\eps_i}) = 1 \quad \mathrm{a.s.},
}
where
\eq{
\a_i^\eps \coloneqq  \{x \in \Z^d : f_i(x) > \eps\}, \quad i \geq 0,\ \eps > 0.
}
In \cite{vargas07}, Vargas defines asymptotic pure atomicity by the same convergence condition, but in probability rather than almost surely.  \textit{A priori}, the definition considered here is a stronger one, but since \eqref{noVSD_claim} implies the negation of the ``in probability" definition, Theorem \ref{total_mass} shows the two definitions are equivalent given \eqref{mgf}.

We begin by making the necessary definitions in the abstract setting.

\subsection{Definitions of relevant functionals}
Observe that the quantity of interest, $\rho_i(\omega_i \in \a_i^{\eps})$, is a function of the mass function $f_i$.
Specifically,
\eq{
\rho_i(\omega_i \in \a_i^{\eps}) = \sum_{x \in \Z^d\, : f_i(x) > \eps} f_i(x).
}
We are thus motivated to define $\|\cdot\|_\eps : S \to [0,1]$ by
\eq{
\|f\|_\eps \coloneqq  \sum_{u\in\N\times\Z^d\, :\, f(u) > \eps} f(u).
}
For any $\eps \in (0,1)$, the map $f \mapsto \|f\|_\eps$ satisfies (i)--(iii) of Corollary \ref{defined_pspm} and thus induces a well-defined function on $\s$.
To establish asymptotic pure atomicity, it will be important that this function is lower semi-continuous, a fact we prove in the lemma below.
Another useful functional will be $\vc{I}_\eps : S \to \{0,1\}$ given by
\eq{
\vc{I}_\eps(f) = \begin{cases}
1 &\textup{if } \max_{u \in \N \times \Z^d} f(u) \geq \eps, \\
0 &\textup{else}.
\end{cases}
}
Once more, it is clear that $\vc{I}_\eps$ satisfies the hypotheses of Corollary \ref{defined_pspm}, and so it is well-defined on~$\s$.

\begin{lemma} \label{eps_norm_equivalence}
For any $\eps \in (0,1)$, the following statements hold:
\begin{itemize}
\item[(a)] The map $\|\cdot\|_\eps : \s \to [0,1]$ is lower semi-continuous and thus measurable.
\item[(b)] The map $\vc{I}_\eps : \s \to \{0,1\}$ is upper semi-continuous and thus measurable.
\end{itemize}
\end{lemma}

\begin{proof}
The proof of (a) is similar to that of Lemma \ref{norm_equivalence}; see Appendix \ref{proof_eps_norm}.
For claim (b) we need only to consider the case when $(f_n)_{n \geq 1}$ is a sequence in $\s$ satisfying $\vc{I}_\eps(f_n) = 1$ for infinitely many $n$, and $d(f_n,f) \to 0$ as $n \to \infty$.
We must show $\vc{I}_\eps(f) = 1$.
By passing to a subsequence, we may assume $\vc{I}_\eps(f_n) = 1$ for all $n$.
That is, $\max(f_n) \geq \eps$ for all $n$, and so Lemma \ref{max_lemma} forces $\max(f) \geq \eps$, meaning $\vc{I}_\eps(f) = 1$.
\end{proof}

\subsection{Proof of asymptotic pure atomicity at low temperature} \label{characterize_proof}
First we simplify the problem by providing a sufficient condition for asymptotic pure atomicity.
The statement below can be easily verified by the reader.

\begin{lemma} \label{equiv_apa}
If for every $c < 1$, there is $\eps > 0$ such that
\eq{
\liminf_{n \to \infty} \frac{1}{n} \sum_{i = 0}^{n-1} \rho_i(\omega_i \in \a_i^\eps) > c \quad \mathrm{a.s.}, 
}
then $(f_i)_{i \geq 0}$ is asymptotically purely atomic.
\end{lemma}

We are now ready to prove Theorem \ref{intro_result1}. For the convenience of the reader, we will restate the result here before giving the proof.

\begin{thm} \label{total_mass}
Assume \eqref{mgf}.
Then the following statements hold:
\begin{itemize}
\item[(a)] If $\beta > \beta_{\mathrm{c}}$, then $(f_i)_{i \geq 0}$ is asymptotically purely atomic.
\item[(b)] If $0 \leq \beta \leq \beta_{\mathrm{c}}$, then there is a sequence $(\eps_i)_{i \geq 0}$ tending to 0 as $i \to \infty$, such that
\eeq{
\lim_{n \to \infty} \frac{1}{n} \sum_{i = 0}^{n-1} \rho_{i}(\omega_i \in \a_i^{\eps_i}) = 0 \quad \mathrm{a.s.}\label{noVSD_claim}
}
\end{itemize}
\end{thm}

While the functionals $\|\cdot\|_\eps$ and $\vc{\i}_\eps(\cdot)$ are not continuous, their semi-continuity will be sufficient to make a limiting argument, thanks to the following generalization of Dini's first theorem.

\begin{lemma} \label{dini}
Let $(\x,\tau)$ be a compact metric space.
If $(F_n)_{n \geq 1}$ is a non-decreasing sequence of lower semi-continuous functions $\x \to \R$ converging pointwise to an upper semi-continuous function $F$, then the convergence is uniform.
\end{lemma}

\begin{proof}
Let $G_n \coloneqq  F - F_n$ so that $G_n \searrow 0$ pointwise on $\x$, and $G_n$ is upper semi-continuous.
For given $\eps > 0$, consider the set
\eq{
U_n \coloneqq  \{x \in \x : G_n(x) < \eps\},
}
which is open because $G_n$ is upper semi-continuous.
Since $G_n(x) \to 0$ for every $x \in \x$, we have $\bigcup_n U_n = \x$.
By compactness of $\x$, there is a finite list $n_1 < n_2 < \dots < n_k$ such that
$\bigcup_{j = 1}^k U_{n_j} = \x$.
But the monotonicity assumption guarantees $U_n$ is an ascending sequence.
Hence $\x = \bigcup_{j = 1}^k U_{n_j} = U_{n_k} = U_n$ for all $n \geq n_k$, meaning
\eq{
n \geq n_k \quad &\implies \quad
U_n = \x \\
&\implies \quad
|F(x) - F_n(x)| = F(x) - F_n(x) < \eps \quad \text{for all $x \in \x$.}
}
That is, $F_n \nearrow F$ uniformly on $\x$.
\end{proof}

\begin{proof}[Proof of Theorem \ref{total_mass}]
We first prove (a).
For $\eps > 0$, define $\|\cdot\|_\eps$ on $\p(\s)$ by
\eq{
\|\nu\|_\eps \coloneqq  \int \|f\|_\eps\ \nu(\dd f)
= \int \sum_{u\, : f(u) > \eps} f(u)\ \nu(\dd f).
}
In light of Lemma \ref{portmanteau}, Lemma \ref{eps_norm_equivalence}(a) implies $\|\cdot\|_\eps : \p(\s) \to \R$ is lower semi-continuous.
For any $f \in \s$ such that $\|f\| = 1$, we have $\|f\|_\eps \nearrow 1$ as $\eps \to 0$.
In particular, under the assumption that $\beta > \beta_{\mathrm{c}}$,
Theorem \ref{characterization}(b) and monotone convergence imply that for any $\nu \in \m$,
\eq{
\|\nu\|_\eps \nearrow 1 \quad \text{as $\eps \to 0$.} 
}
Since $\m$ is compact, Lemma \ref{dini} strengthens this pointwise convergence 
to uniform convergence.
That is, for any $c < 1$, there is $\eps > 0$ such that $\|\nu\|_\eps > c$ for all $\nu\in\m$.
Furthermore, by compactness of $\p(\s)$ and lower semi-continuity of $\|\cdot\|_\eps$, we can find $\delta > 0$ such that for any $\mu \in \p(\s)$,
\eq{
\w(\mu,\m) < \delta \quad \implies \quad \|\mu\|_\eps > c. 
}
Now Theorem \ref{close_M} implies
\eq{
\liminf_{n\to\infty} \frac{1}{n} \sum_{i = 0}^{n-1} \rho_i(\omega_i \in \a_i^\eps) 
= \liminf_{n\to\infty} \|\mu_{n}\|_\eps \geq c \quad \mathrm{a.s.}
}
We have thus verified the hypothesis of Lemma \ref{equiv_apa}, and so $(f_i)_{i \geq 0}$ is asymptotically purely atomic.

For (b), we assume $0 \leq \beta \leq \beta_{\mathrm{c}}$.
For $\eps > 0$, define $\vc{\i}_\eps : \p(\s) \to \R$ by
\eq{
\vc{\i}_\eps(\nu) \coloneqq  \int \vc{I}_\eps(f)\ \nu(\dd f)
= \int \one_{\{\max_{u \in \N \times \Z^d} f(u) \geq \eps\}}\ \nu(\dd f).
}
Considering Lemma \ref{portmanteau}, we see from Lemma \ref{eps_norm_equivalence}(b) that $\vc{\i}_\eps$ is an upper semi-continuous map.
By Theorem \ref{characterization} and Theorem \ref{close_M}, $\w(\mu_n,\delta_\vc{0}) \to 0$ almost surely as $n \to \infty$, and so
\eq{
\limsup_{n \to \infty} \vc{\i}_\eps(\mu_n) \leq \vc{\i}_\eps(\delta_\vc{0}) = \vc{I}_\eps(\vc{0}) = 0 \quad \mathrm{a.s.} \quad \text{for any $\eps>0$}.
}
In particular, for any $j \in \N$,
\eq{
\lim_{N \to \infty} \PP\bigg(\bigcap_{n = N}^\infty \{\vc{\i}_\eps(\mu_{n}) < 2^{-j}\}\bigg)=1. 
}
Now let $(\kappa_j)_{j \geq 1}$ be any decreasing sequence tending to $0$ as $j \to \infty$.
Set $M_0 \coloneqq  0$.
By taking complements in the above display, we can inductively choose $M_j > M_{j-1}$ such that
\eeq{ \label{bad_probability}
\PP\bigg(\bigcup_{n = M_j}^\infty \{\vc{\i}_{\kappa_{j+1}}(\mu_{n}) \geq 2^{-(j+1)}\} \bigg) &< 2^{-(j+1)}, \quad j\geq1.
}
Define $\eps_i \coloneqq  \kappa_j$ when $M_{j-1} \leq i \leq M_{j} - 1$.
For $M_{\ell-1} \leq n \leq M_{\ell}$ and any $k<\ell$, the monotonicity of $\kappa_j$ gives
\eq{
\frac{1}{n} \sum_{i = 0}^{n-1} \vc{I}_{\eps_i}(f_i)
&= \frac{1}{n} \bigg[\sum_{i = 0}^{M_k-1} \vc{I}_{\eps_i}(f_i) + \sum_{j = k+1}^{\ell-1} \sum_{i = M_{j-1}}^{M_{j}-1} \vc{I}_{\kappa_j}(f_i) + \sum_{i = M_{\ell-1}}^{n-1} \vc{I}_{\kappa_{\ell}}(f_i)\bigg]  \\
&\leq \frac{1}{n}\bigg[\sum_{i = 0}^{M_k-1} \vc{I}_{\kappa_{k}}(f_i) + \sum_{j = k+1}^{\ell} \sum_{i = 0}^{n-1} \vc{I}_{\kappa_j}(f_i)\bigg]\\
&\leq \sum_{j = k}^{\ell} \frac{1}{n} \sum_{i = 0}^{n-1} \vc{I}_{\kappa_j}(f_i)
= \sum_{j = k}^\ell \vc{\i}_{\kappa_j}(\mu_{n}),
}
where we have identified $f_i$ with an element of $\s$ (this identification is measurable by Lemma~\ref{S_meas}, cf.~the discussion following Corollary \ref{increment_cor}).
Writing a more general inequality, we can say that for all $n \geq M_k$,
\eq{
\frac{1}{n} \sum_{i = 0}^{n-1} \vc{I}_{\eps_i}(f_i)
\leq \sum_{j \geq k :\ n \geq M_{j-1}} \vc{\i}_{\kappa_j}(\mu_{n}).
}
It follows from a union bound and \eqref{bad_probability}  that
\eq{
&\PP\bigg(\bigcup_{n = M_k}^\infty \bigg\{\frac{1}{n} \sum_{i = 0}^{n-1} \vc{I}_{\eps_i}(f_i) \geq \sum_{j \geq k} 2^{-j}\bigg\}\bigg)\\
&\leq \PP\bigg(\bigcup_{n = M_{k-1}}^\infty \bigcup_{j \geq k\, :\, n \geq M_{j-1}} \{\vc{\i}_{\kappa_j}(\mu_{n}) \geq 2^{-j}\}\bigg) \\
&= \PP\bigg(\bigcup_{j \geq k} \bigcup_{n= M_{j-1}}^\infty \{\vc{\i}_{\kappa_j}(\mu_{n}) \geq 2^{-j}\}\bigg) \\
&\leq \sum_{j \geq k} \PP\bigg(\bigcup_{n = M_{j-1}}^\infty \{\vc{\i}_{\kappa_j}(\mu_{n}) \geq 2^{-j}\}\bigg) 
< \sum_{j \geq k} 2^{-j} = 2^{-k+1}.
}
By Borel--Cantelli, the following is true with probability one: For only finitely many $k$ is
\eq{
\frac{1}{n} \sum_{i = 0}^{n-1} \vc{I}_{\eps_i}(f_i) \geq 2^{-k+1} \quad \text{for some $n \geq M_k$.}
}
This implies
\eeq{
\lim_{n \to \infty} \frac{1}{n} \sum_{i = 0}^{n-1} \vc{I}_{\eps_i}(f_i) = 0 \quad \mathrm{a.s.} \label{bigger_to_0}
}
Finally, note that $\rho_i(\omega_i \in \a_i^\eps)$ is nonzero only when $\a_i^\eps$ is nonempty, in which case $\vc{I}_\eps(f_i) = 1$.
Hence
\eq{
\rho_i(\omega_i \in \a_i^\eps) \leq \vc{I}_\eps (f_i) \quad \text{for any $i \geq 0$, $\eps > 0$,}
}
and so \eqref{noVSD_claim} follows from \eqref{bigger_to_0}.
\end{proof}

\section{Geometric localization} \label{main_thm2}
Recall the following definition from Section \ref{results}.
We say that the sequence $(f_i)_{i \geq 0}$ exhibits \textit{geometric localization with positive density} if for every $\delta > 0$, there is $K < \infty$ and $\theta  > 0$ such that
\eq{
\liminf_{n \to \infty} \frac{1}{n} \sum_{i = 0}^{n-1} \one_{\{f_i \in \g_{\delta,K}\}} \geq \theta \quad \mathrm{a.s.},
}
where
\eeq{ \label{g_set_def}
&\g_{\delta,K} \coloneqq \Big\{f : \Z^d \to [0,1]\ \Big|\ \|f\| = 1, \exists\, D\subset \Z^d \text{ such that} \\ 
&\phantom{\g_{\delta,K} = \Big\{f : \Z^d \to [0,1]\ \bigg|\ }\diam(D)\leq K\text{ and }\sum_{x\in D}f(x) > 1-\delta\Big\},
} 
and
\eq{
\diam(D) \coloneqq  \sup\{\|x - y\|_1 : x,y \in D\}.
}
If $K$ can always be chosen so that $\theta\geq1-\delta$, then the sequence is ``geometrically localized with \textit{full} density."
The main goal of this section is to establish that under \eqref{mgf}, there is positive density geometric localization if and only if $\beta > \beta_{\mathrm{c}}$.

\subsection{Definitions of relevant functionals}
To employ the theory of partitioned subprobability measures, we make a few definitions.
For $f \in S$, let $H_f$ denote $\N$-support of $f$:
\eq{
H_f \coloneqq  \{n \in \N : f(n,x) > 0 \text{ for some $x \in \Z^d$}\}.
}
The first functional of interest is the size of the $\N$-support of $f$,
\eq{
N(f) \coloneqq  |H_f|,
}
which we call the \textit{support number} of $f$.
Second, for $f\in \s$ and $\delta \in (0,1)$, it is natural in studying geometric localization to consider the quantity
\eq{
W_\delta(f) &\coloneqq  \inf\Big\{\diam(D) : D \subset \Z^d,\  \sum_{x \in D} f(n,x) > 1 - \delta \text{ for some $n \in \N$}\Big\}.
}
In words, $W_\delta(f)$ is the smallest $K$ such that a region of diameter $K$ in $\N \times \Z^d$ has $f$-mass strictly greater than $1-\delta$.
Here we follow the convention that the infimum of an empty set is infinity, meaning $W_\delta(f) = \infty$ whenever there is no one copy of $\Z^d$ on which $f$ has mass greater than $1-\delta$.
We will write
\eq{
\v_{\delta,K} \coloneqq  \{f \in \s : W_\delta(f) \leq K\},
}
so that we can naturally identify $\g_{\delta,K}$ in $\s$ by
\eeq{
\g_{\delta,K} = \v_{\delta,K} \cap \{f \in \s : N(f) = 1,\, \|f\| = 1\}. \label{g_intersection}
}
Next define $q_n : S \to [0,1]$ by
\eq{
q_n(f) \coloneqq  \sum_{x\in \Z^d} f(n,x),
}
and let
\eq{
m(f) \coloneqq  \max_{n\in \N} q_n(f).
}
Finally, it will also be useful to analyze the function
\eq{
Q(f) \coloneqq  \sum_{n \in \N} \frac{q_n(f)}{1 - q_n(f)},
}
where $1/0 = \infty$.
Observe that $Q(f) = \infty$ if and only if $N(f) = 1$ and $\|f\| = 1$.

It is easy to see from Corollary \ref{defined_pspm} that $N, W_\delta$, $m$, and $Q$ are well-defined on $\s$.
The fact that $N : \s \to \N \cup \{0,\infty\}$ is measurable is surprisingly non-trivial.
The proof, however, is peripheral to our present focus, and so we defer this argument to Appendix \ref{measurability_support_number}.
A more central part of our methods is the semi-continuity of relevant functionals, which is 
the content of
the next lemma.

\begin{lemma} \label{more_equivalences}
The following statements hold:
\begin{itemize}
\item[(a)] For any $\delta \in (0,1)$, $W_\delta: \s \to \N \cup \{0,\infty\}$ is upper semi-continuous and thus measurable.
\item[(b)] $m : \s \to [0,1]$ is lower semi-continuous and thus measurable.
\item[(c)] $Q : \s \to [0,\infty]$ is lower semi-continuous and thus measurable.
\end{itemize}
\end{lemma}

These continuity properties can be understood intuitively.
For instance, (a) is simply a consequence of the fact that if a measure on $\Z^d$ places mass greater than $1-\delta$ on a compact set $K$, then any sufficiently similar measure (in the weak or vague topologies) must do the same.
On the other hand, (b) and (c) come from the observation that a converging sequence of partitioned subprobability measures can divide mass on one copy of $\Z^d$ between several copies in the limit when large parts of the mass are drifting away from each other.
But the reverse is not true, because vertices in distinct copies of $\Z^d$ are considered infinitely far apart.
The proofs that make these ideas precise are tedious and thus postponed to Appendix \ref{proof_3_functionals}.
They are similar in spirit the proof of Lemma \ref{norm_equivalence} but substantially more involved.

An observation that will be useful is that at low temperature, the functional $Q$ must have infinite expectation according to any element of $\m$.

\begin{lemma} \label{Q_lemma}
Assume $\beta > \beta_{\mathrm{c}}$.
Then for any $\nu \in \m$,
\eq{
\int Q(f)\ \nu(\dd f) = \infty.
}
\end{lemma}

\begin{proof}
Consider any $\nu \in \m$.
By Theorem \ref{characterization}, the assumption of $\beta > \beta_{\mathrm{c}}$ implies 
$\nu(\{f \in \s: \|f\| = 1\}) = 1$.
Suppose toward a contradiction that $\nu \in \m$ satisfies
\eeq{
\int Q(f)\ \nu(\dd f) < \infty. \label{finite_Q_integral}
}
It must then be the case that $\nu(\{f \in \s: N(f) = 1\}) = 0$, since $Q(f) = \infty$ whenever $N(f) = 1$ and $\|f\| = 1$.
Let $f \in \s$ be random with law $\nu$.
By the previous observation, $1-q_n(f) > 0$ for all $n \in \N$ with $\nu$-probability 1.
Consider an environment $(Y_u)_{u \in \N \times \Z^d}$ of i.i.d.~random variables with law $\mathfrak{L}$, that is independent of $f$.
When viewed as an element of $\s$, the law of the function
\eq{
F(u) \coloneqq  \frac{\sum_{v \sim u} f(v)e^{\beta Y_u}}{\sum_{w \in \N \times \Z^d} \sum_{v \sim w} f(v)e^{\beta Y_w}}, \quad u \in \N \times \Z^d,
}
is $\t\nu = \nu$.
Observe that because $\|F\| = 1$,
\eq{
Q(F) = \sum_{n \in \N} \frac{q_n(F)}{1-q_n(F)}
&= \sum_{n \in \N} \frac{\sum_{x \in \Z^d} F(n,x)}{\sum_{k \neq n} \sum_{x \in \Z^d} F(k,x)}\\
&= \sum_{n \in \N} \frac{\sum_{x \in \Z^d} \sum_{y \sim x} f(n,y)e^{\beta Y_{n,\, x}}}{\sum_{k \neq n} \sum_{x \in \Z^d} \sum_{y \sim x} f(k,y)e^{\beta Y_{k,\, x}}}.
}
Conditioned on $f$, the numerator and denominator of
\eq{
\frac{\sum_{x \in \Z^d} \sum_{y \sim x} f(n,y)e^{\beta Y_{n,\, x}}}{\sum_{k \neq n} \sum_{x \in \Z^d} \sum_{y \sim x} f(k,y)e^{\beta Y_{k,\, x}}}
}
are independent.
Hence
\eq{
\EE\givenk{Q(F)}{f} &= \sum_{n \in \N} \EE\givenk[\bigg]{\frac{\sum_{x \in \Z^d} \sum_{y \sim x} f(n,y)e^{\beta Y_{n,\, x}}}{\sum_{k \neq n} \sum_{x \in \Z^d} \sum_{y \sim x} f(k,y)e^{\beta Y_{k,\, x}}}}{f} \\
&= \sum_{n \in \N} \EE\givenk[\bigg]{\sum_{x \in \Z^d} \sum_{y \sim x} f(n,y)e^{\beta Y_{n,\, x}}}{f}\\
&\qquad \qquad \cdot
\EE\givenk[\bigg]{\frac{1}{\sum_{k \neq n} \sum_{x \in \Z^d} \sum_{y \sim x} f(k,y)e^{\beta Y_{k,\, x}}}}{f} \\
&> \sum_{n \in \N} \frac{\EE\givenk[\big]{\sum_{x \in \Z^d} \sum_{y \sim x} f(n,y)e^{\beta Y_{n,\, x}}}{f}}{\EE\givenk[\big]{\sum_{k \neq n} \sum_{x \in \Z^d} \sum_{y \sim x} f(k,y)e^{\beta Y_{k,\, x}}}{f}}
\\
&= \sum_{n \in \N} \frac{2d \cdot e^{\lambda(\beta)} \cdot q_n(f)}{2d \cdot e^{\lambda(\beta)}(1-q_n(f))}
= Q(f),
}
where the strict inequality is due to the strict convexity of $t \mapsto 1/t$ on $(0,\infty)$ and the non-degeneracy of $\mathfrak{L}$.
But now \eqref{finite_Q_integral} allows us to write
\eq{
\int Q(f)\ \nu(\dd f) &= \iint Q(F)\ Tf(\dd F)\, \nu(\dd f)\\ 
&= \int \EE\givenk{Q(F)}{f}\, \nu(\dd f)
> \int Q(f)\, \nu(\dd f),
}
yielding the desired contradiction.
\end{proof}

\subsection{Proof of geometric localization with positive density}
We can now prove the main theorem of Section \ref{main_thm2}. 
This theorem, which encompasses Theorem \ref{intro_result2}, shows that positive density geometric localization is equivalent to $\beta > \beta_{\mathrm{c}}$. 
It also proves that the single-copy condition (stated as \eqref{single_copy_assumption} below) implies full density geometric localization.

\begin{thm} \label{localized_subsequence}
Assume \eqref{mgf}.
Let $f_i(\cdot) \coloneqq  \rho_i(\omega_i = \cdot)$.
Then the following statements hold:
\begin{itemize}
\item[(a)] If $\beta > \beta_{\mathrm{c}}$, then $(f_i)_{i \geq 0}$ is geometrically localized with positive density. Moreover, the quantities $K$ and $\theta$ in the definition of positive density geometric localization are deterministic, and depend only on the choice of $\delta$, as well as $\mathfrak{L}$, $\beta$, and $d$. 
\item[(b)] If $\beta > \beta_{\mathrm{c}}$ and 
\eeq{
\nu(\u_1) = \nu(\{f \in \s: N(f) = 1\}) = 1 \quad \text{for all $\nu \in \m$}, \label{single_copy_assumption}
}
then $(f_i)_{i \geq 0}$ is geometrically localized with full density.
\item[(c)] If $0 \leq \beta \leq \beta_{\mathrm{c}}$, then for any $\delta \in (0,1)$ and any $K$,
\eeq{
\lim_{n \to \infty} \frac{1}{n} \sum_{i = 0}^{n-1} \one_{\{f_i \in \g_{\delta,K}\}} = 0 \quad \mathrm{a.s.} \label{no_localization}
}
\end{itemize}
\end{thm}

\begin{proof}
Fix $\delta\in (0,1)$ throughout. 
Recall that
\eq{
\v_{\delta,K} = \{f \in \s : W_\delta(f) \leq K\} = \{f \in \s : W_\delta(f) < K+1\},
}
which is open by the upper semi-continuity of $W_\delta$.
For (a) and (b), we assume $\beta > \beta_{\mathrm{c}}$.
Given $\delta > 0$, define the set
\eq{
\u_\delta \coloneqq  \{f\in \s: m(f) > 1-\delta\} = \bigcup_{K = 0}^\infty \v_{\delta,K}.
}
For (b) only, note the following consequence of Theorem \ref{characterization}(b):
\eq{
 \eqref{single_copy_assumption} \quad \implies \quad \nu(\u_\delta) = 1 \quad \text{for all $\nu \in \m$.} 
} 
Otherwise, we refer to Lemma \ref{Q_lemma} which tells us that for any $\nu\in  \m$,
\eq{
\int Q(f)\ \nu(\dd f) =\infty.
}
It follows that $\nu(\u_\delta) > 0$,
since otherwise the $\nu$-essential supremum of $m(f)$ would be strictly less than 1, forcing the $\nu$-essential supremum of $Q(f)$ to be finite.
By Lemma \ref{more_equivalences}(b), $\u_\delta$ is an open set. Consequently, the map $\mu \mapsto \mu(\u_\delta)$ is lower semi-continuous on $\p(\s)$. 
Since a lower semi-continuous function on a compact set attains its minimum value, we must have
\eq{
\Theta\coloneqq  \inf_{\nu\in\m} \nu(\u_\delta) > 0.
}
For any $f \in \u_\delta$, the quantity $W_\delta(f)$ is finite.
Therefore, for any $\nu \in \m$,
\eq{
\nu(\v_{\delta,K}) \wedge \Theta \nearrow \nu(\u_\delta) \wedge \Theta = \Theta \quad \text{as $K \nearrow \infty$.}
}
Since $\nu \mapsto \nu(\v_{\delta,K}) \wedge \Theta$ is lower semi-continuous, Lemma \ref{dini} upgrades this convergence to be uniform on the compact set $\m$.
That is, for any $\theta<\Theta$, we can find $K$ large enough that $\nu(\v_{\delta,K}) > \theta$ for every $\nu \in \m$.
Furthermore, by compactness of $\p(\s)$ there exists $\xi > 0$ such that
\eq{
\w(\mu,\m) < \xi \quad \implies \quad \mu(\v_{\delta,K}) > \theta.
}
We may now conclude from Theorem \ref{close_M} and \eqref{g_intersection} that
\eq{
\theta \leq \liminf_{n\to\infty} \mu_{n}(\v_{\delta,K}) 
&= \liminf_{n\to\infty} \frac{1}{n}\sum_{i=0}^{n-1} \one_{\{f_i \in \v_{\delta,K}\}} \\
&= \liminf_{n\to\infty} \frac{1}{n}\sum_{i=0}^{n-1} \one_{\{f_i \in \g_{\delta,K}\}} \quad \mathrm{a.s.},
}
which completes the proof of (a) and (b).

For claim (c), suppose $0 \leq \beta \leq \beta_{\mathrm{c}}$ so that Theorem \ref{characterization} gives $\m = \{\delta_\vc{0}\}$.
Fix $K > 0$.
Recall from the proof of Theorem \ref{reprove_equivalence} that in this high temperature case,
\eeq{
\lim_{n \to \infty} \int \max(f)\ \mu_{n}(\dd f) = \lim_{n \to \infty} \frac{1}{n} \sum_{i = 0}^{n-1} \max_{x \in \Z^d} f_i(x) = 0 \quad \mathrm{a.s.} \label{max0}
}
Notice that if $\eps > 0$ is sufficiently small that
\eeq{
D \subset \Z^d,\, \diam(D) \leq K \quad \implies \quad \eps|D| < 1 - \delta, \label{diam_constraint}
}
then
\eeq{
\frac{1}{n} \sum_{i = 0}^{n-1} \max_{x \in \Z^d} f_i(x) < \eps^2 \quad \implies \quad
\frac{1}{n} \sum_{i = 0}^{n-1} \one_{\{f_i \in \g_{\delta,K}\}} < \eps. \label{max_diam}
}
Indeed, the hypothesis in \eqref{max_diam} implies there are fewer than $\eps n$ numbers $i$ between 0 and $n-1$ such that $\max_{x \in \Z^d} f_i(x) \geq \eps$, and \eqref{diam_constraint} implies all the remaining $i$ must satisfy $f_i \notin \g_{\delta,K}$.
Therefore, \eqref{max_diam} is true, and so \eqref{max0} implies \eqref{no_localization}.
\end{proof}

\subsection{Localization in a favorite region}
In this final section we prove that single-copy condition~\eqref{single_copy_assumption} implies localization in a ``favorite region" of size $O(1)$.
Recall that the mode of a probability mass function is a location where the function attains its maximum. For any $n\ge 0$ and $K\ge 0$, let $\c_{n}^K$ be the set of all $x\in \Z^d$ that are within $\ell^1$ distance $K$ from {\it every} mode of the random probability mass function $f_n$. Note that $\c_n^K$ is a set with diameter at most $2K$. The following theorem establishes localization of the endpoint in $\c_n^K$, if the single-copy condition holds.

\begin{prop}\label{localization_thm}
Assume  \eqref{mgf} and \eqref{single_copy_assumption}.
Then
\eeq{
\lim_{K \to \infty} \liminf_{n \to \infty} \frac{1}{n} \sum_{i = 0}^{n-1} \rho_i(\omega_i \in \c_{i}^K) = 1\quad \mathrm{a.s.} \label{favorite_region}
}
\end{prop}

\begin{proof}
For the sake of completeness, we first verify that $\rho_i(\omega_i \in \c_i^K)$ is a measurable function (with respect to $\f_i$).
For any Borel set $B \subset \R$, the event $\{\rho_i(\omega_i \in \c_i^K) \in B\}$ can be expressed as
\eq{
\bigcup_{\substack{A \subset \Z^d \\ |A| \leq (2d)^i}} \bigg[\bigg(\bigcap_{x \in A} \{f_i(x) = \max_{y\in\Z^d} f_i(y)\}\bigg) \cap
\bigg\{\sum_{\substack{y \in \Z^d\, :\, \|x - y\|_1 \leq K\, \forall\, x \in A }} f_i(y) \in B\bigg\}\bigg].
}
It is clear that the above display is a measurable event, and so $\rho_i(\omega_i \in \c_i^K)$ is measurable.

Assume \eqref{single_copy_assumption}, so that $(f_i)_{i \geq 0}$ is geometrically localized with full density.
(In particular, by Theorem \ref{localized_subsequence}(c), we must have $\beta > \beta_{\mathrm{c}}$.)
As in the proof of Theorem \ref{total_mass}(a), the assumption $\beta > \beta_{\mathrm{c}}$ ensures that the conclusion of Lemma \ref{equiv_apa} holds.
Consequently, given any $\theta<1$, we can choose $\eps \in (0,1-\theta)$ such that
\eq{
\liminf_{n \to \infty} \frac{1}{n} \sum_{i = 0}^{n-1} \one_{\{\a_i^\eps\text{ is nonempty}\}} > \theta \quad \mathrm{a.s.}
}
And by geometric localization with full density, there is $K$ such that
\eq{
\liminf_{n \to \infty} \frac{1}{n} \sum_{i = 0}^{n-1} \one_{\{f_i \in \g_{\eps,K}\}} > \theta \quad \mathrm{a.s.} 
}
Therefore, there almost surely exists some $N$ satisfying
\eq{
n \geq N \quad &\implies \quad \frac{1}{n} \sum_{i = 0}^{n-1} \one_{\{\a_i^\eps \text{ is nonempty}\}} > \theta \quad \text{and} \quad
\frac{1}{n}\sum_{i = 0}^{n-1} \one_{\{f_i \in \g_{\eps,K}\}} > \theta.
}
%
%
So for $n \geq N$, there are at least $(2\theta - 1)n$ numbers $i$ between 0 and $n-1$ for which both of the following statements are true:
First, $f_i(x) > \eps$ for some (and thus any) mode $x \in \Z^d$ of $f_i$. 
Second $\rho_i(\omega_i \in D_i) > 1 - \eps$ for some $D_i \subset \Z^d$ with $\diam(D_i) \leq K$.
For such $i$, \textit{all} modes must belong to $D_i$, which implies $D_i \subset \c_i^K$.
In particular, 
\eq{
\rho_i(\omega_i \in \c_i^K) \geq \rho_i(\omega_i \in D_i) > 1 - \eps > \theta,
}
and so
\eq{
n \geq N \quad \implies \quad \frac{1}{n} \sum_{i = 0}^{n-1} \rho_i(\omega_i \in \c_i^K) > \theta(2\theta - 1).
}
As $\rho_i(\omega_i \in \c_i^L) \geq \rho_i(\omega_i \in \c_i^{K})$ for $L \geq K$, we in fact have
\eq{
\lim_{K \to \infty} \liminf_{n \to \infty} \frac{1}{n} \sum_{i = 0}^{n-1}\rho_i(\omega_i \in \c_i^K) \geq \theta(2\theta-1) \quad \mathrm{a.s.}
}
As $\theta < 1$ is arbitrary, we can take a countable sequence $\theta_k \to 1$ to conclude \eqref{favorite_region}.
\end{proof}

\section{Proof of Proposition 3.2} \label{proof_continuity}
\addtocontents{toc}{\protect\setcounter{tocdepth}{1}}
Since $(\s,d)$ is a compact metric space, uniform continuity of $f \mapsto Tf$ and $f \mapsto \EE(\log^q\wt F)$ will be implied by continuity.
So it suffices to prove continuity at a fixed $f \in \s$.
Given $\eps > 0$, choose $\kappa_1 > 0$ small enough that
\eeq{
\big((2d)^{-1}\kappa_1 + d^{-1/2}\sqrt{2\kappa_1}\,\big) \sqrt{e^{\lambda(2\beta)}e^{\lambda(-2\beta)}} < \frac{\eps}{8}.\label{kappa1_1}
}
Next choose $\kappa_2 > 0$ small enough that all of the following statements are true:
\begin{subequations} \label{kappa2}
\begin{align}
\kappa_2\, e^{\lambda(2\beta)}e^{\lambda(-2\beta)} &< \frac{\eps}{4} \label{kappa2_3} \\
(4d+1)\kappa_2 &< \kappa_1 \label{kappa2_1} \\
2d\sqrt{\kappa_2} &< \kappa_1. \label{kappa2_2}
\end{align}
\end{subequations}
Given a representative $f \in S$, we may take $A \subset \N \times \Z^d$ to be finite but large enough that
\eeq{
\sum_{u \notin A} f(u) < \kappa_2. \label{notAsum}
}
By possibly omitting some elements of $A$, we may assume $f(u) > 0$ for all $u \in A$, and then take $\delta > 0$ small enough that
\eeq{
f(u)^2 > \delta \quad \text{for all $u \in A$} \label{pointwiseA}.
}
Considering \eqref{kappa2}, we may choose $\delta$ so that all of the following statements are also true:
\begin{subequations}
\begin{align}
(\delta+\kappa_2) e^{\lambda(2\beta)}e^{\lambda(-2\beta)} &< \frac{\eps}{4} \label{delta6} \\
(4d+1)(\delta+\kappa_2) &< \kappa_1\label{delta3} \\
2d\sqrt{\delta+\kappa_2} &< \kappa_1. \label{delta5}
\intertext{In addition, we will assume}
(2d)^2 |A| \sqrt{\delta} < \kappa_2, \label{borderA}
\intertext{and finally}
\delta < \min\Big(\frac{\eps}{16},2^{-3}\Big). \label{delta4}
\end{align}
\end{subequations}
Now suppose $g \in \s$ satisfies $d(f,g) < \delta$.
We will show $\w(Tf,Tg) < \eps$.

Given a representative $g \in S$, we can choose an isometry $\psi : C \to \N \times \Z^d$ such that $d_\psi(f,g) < \delta$.
From \eqref{pointwiseA}, we deduce that $A \subset C$, since otherwise $d_\psi(f,g) \geq f(u)^2 > \delta$ for some $u \in A \setminus C$. 
Let $\phi$ be the restriction of $\psi$ to $A$.
It follows that
\eeq{
\sum_{v \in A} |f(v) - g(\phi(v))| \leq \sum_{v \in C} |f(v) - g(\psi(v))| < \delta. \label{Asum}
}
Since $d_\psi(f,g) < \delta < 2^{-3}$, we necessarily have $\deg(\phi) \geq \deg(\psi) \geq 4$.
Let $\Phi: B \to \N \times \Z^d$ be defined as in Lemma \ref{extension} so that $\deg(\Phi) \geq 2$.
Now take $(Y_u)_{u \in \Z^d}$ to be i.i.d.~random variables with shared law $\mathfrak{L}$.
Define $Z_u \coloneqq  Y_{\Phi^{-1}(u)}$ if $u \in \Phi(B)$; otherwise let $Z_u$ be an independent copy of $Y_u$.
Define $F,G\in\s$ as in Proposition \ref{same_law} so that the laws of $F$ and $G$ are $Tf$ and $Tg$, respectively, and we may consider the coupling $(F,G)$ to determine an upper bound on $\w(Tf,Tg)$.
That is, $\w(Tf,Tg) \leq \EE[d(F,G)] \leq \EE[d_\Phi(F,G)]$.

To simplify notation, we will write
\eq{
\wt{f}(u) &= \sum_{v \sim u} f(v)e^{\beta Y_u} & \wt{F} &= \sum_{u \in \N \times \Z^d} \wt{f}(u) + 2d(1 - \|f\|){e^{\lambda(\beta)}} \\
\wt{g}(u) &= \sum_{v \sim u} g(v)e^{\beta Z_u} & \wt{G} &= \sum_{u \in \N \times \Z^d} \wt{g}(u) + 2d(1 - \|g\|){e^{\lambda(\beta)}}
}
so that $F(u) = \wt{f}(u)/\wt{F}$ and $G(u) = \wt{g}(u)/\wt{G}$.
For any $u \in B$,
\eeq{
\wt{f}(u) - \wt{g}(\Phi(u)) &= \sum_{v \sim u} f(v)e^{\beta Y_{u}} - \sum_{v \sim \Phi(u)} g(v)e^{\beta Z_{\Phi(u)}}  \\
&= e^{\beta Y_u} \bigg(\sum_{v \sim u} f(v) - \sum_{v \sim \Phi(u)} g(v)\bigg). \label{first_observation}
}
Since $\deg(\Phi) \geq 2$, the two sets
\eq{
R_u \coloneqq  \{v: v \sim u,\ v \in A\} \quad \text{and} \quad R_u' \coloneqq  \{v: v \sim \Phi(u),\ v \in \phi(A)\}
}
satisfy 
$\phi(R_u) = R_u'$.
Indeed, if $v \in R_u$, then $\Phi(u) - \phi(v) = \Phi(u) - \Phi(v) = u - v$, and so $\phi(v)$ belongs to $R_u'$.  
Conversely, if $\phi(v) = \Phi(v) \in R_u'$, then $u - v = \Phi(u) - \Phi(v) = \Phi(u) - \phi(v)$, and so $v$ belongs to $R_u$.
We also define
\eq{
S_u \coloneqq  \{v: v \sim u,\ v \notin A\} \quad \text{and} \quad S'_u \coloneqq  \{v : v \sim \Phi(u),\ v \notin \psi(C)\},
}
and finally
\eq{
T_u' \coloneqq  \{v: v \sim \Phi(u),\ v \in \psi(C \setminus A)\}.
}
We thus have
\eeq{
\bigg| \sum_{v \sim u} f(v) - \sum_{v \sim \Phi(u)} g(v) \bigg|
&\leq \sum_{v \in R_u} |f(v) - g(\phi(v))| + \sum_{v \in S_u} f(v) \\
&\qquad + \sum_{v \in S_u'} g(v) + \sum_{v \in T_u'} g(v). \label{Bsum0}
}
When summed over $u \in B$, each of the four terms on the right-hand side of \eqref{Bsum0} can be separately bounded from above.
We do so in the next paragraph.

Any $v \in \N \times \Z^d$ is the neighbor of at most $2d$ elements of $B$.
Consequently, the inclusion $R_u \subset A$ and the assumption \eqref{Asum} together imply
\eeq{
\sum_{u \in B} \sum_{v \in R_u} |f(v) - g(\phi(v))| \leq 2d \cdot \delta. \label{Bsum1}
}
Next, $S_u \subset A^c$ and \eqref{notAsum} imply
\eeq{
\sum_{u \in B} \sum_{v \in S_u} f(v) \leq 2d \cdot \kappa_2.
}
Since $|B| \leq 2d |A|$, the inclusion $S_u' \subset \psi(C)^c$ and \eqref{borderA} imply
\eeq{
\sum_{u \in B} \sum_{v \in S_u'} g(v) 
&\leq (2d)^2 |A| \sup_{v \notin \psi(C)} g(v) 
\leq (2d)^2 |A| \sqrt{\sum_{v \notin \psi(C)} g(v)^2}  \\ 
&\leq (2d)^2 |A| \sqrt{d_\psi(f,g)} < (2d)^2 |A| \sqrt{\delta} < \kappa_2.
}
Finally, one more application of \eqref{notAsum} gives
\eeq{
\sum_{u \in B} \sum_{v \in T_u'} g(v) &\leq \sum_{u \in B} \sum_{v \in T_u'} \bigl(|f(\psi^{-1}(v)) - g(v)| + f(\psi^{-1}(v))\bigr)  \\
&\leq 2d \sum_{u \in C \setminus A} |f(u) - g(\psi(u))| + 2d \sum_{u \notin A} f(u)  \\
&< 2d(d_\psi(f,g) + \kappa_2) < 2d(\delta + \kappa_2). \label{Bsum4}
}
Having established \eqref{Bsum0}--\eqref{Bsum4} and assumed \eqref{delta3}, we have shown
\eeq{
\sum_{u \in B} \bigg| \sum_{v \sim u} f(v) - \sum_{v \sim \Phi(u)} g(v) \bigg| < (4d + 1)(\delta + \kappa_2) < \kappa_1. \label{sumB}
}
This inequality will be pivotal in obtaining an upper bound on $\EE[d_\Phi(F,G)]$. 

Four other inequalities stated below will also be useful, namely \eqref{nodenominator}, \eqref{nodenominator2}, \eqref{Fbound}, and \eqref{Gbound}.
Let us now prove each of them.
First, for any $u \notin B$, an application of Cauchy--Schwarz gives
\eq{
\sum_{u \notin B} \wt{f}(u)^{2} &=  
\sum_{u \notin B}\bigg(\sum_{v \sim u} f(v)e^{\beta Y_{u}}\bigg)^{2}
\leq \sum_{u \notin B} 2d \cdot e^{2\beta Y_u} \sum_{v \sim u} f(v)^2.
}
Notice that if $u \notin B$, and $v \sim u$, then $v \notin A$.
It follows that
\eeq{
\EE\bigg[\sum_{u \notin B} \wt{f}(u)^{2}\bigg]
&\leq 2d \cdot e^{\lambda(2\beta)} \sum_{u \notin B} \sum_{v \sim u} f(v)^2  \\
&\leq (2d)^2 \cdot e^{\lambda(2\beta)}  \sum_{v \notin A} f(v)^2   \\
&\leq (2d)^2 \cdot e^{\lambda(2\beta)}  \sum_{v \notin A} f(v)
< \kappa_2 (2d)^2 e^{\lambda(2\beta)}   \label{nodenominator},
}
where we have used \eqref{notAsum} in the last inequality.
Since $\deg(\Phi) \geq 2$, the same reasoning as above 
gives
\eeq{
\EE\bigg[\sum_{u \notin \Phi(B)} \wt{g}(u)^{2}\bigg]
\leq (2d)^2 e^{\lambda(2\beta)}  \sum_{v \notin \phi(A)} g(v)^2, \label{notPhiB}
}
and we can bound the last sum by again applying \eqref{notAsum}:
\eeq{
\sum_{v \notin \phi(A)} g(v)^2 &\leq \sum_{v \in \psi(C \setminus A)} g(v) + \sum_{v \notin \psi(C)} g(v)^2\\
&\leq \sum_{v \in C \setminus A} \bigl(|f(v) - g(\psi(v))| + f(v)\bigr) + \sum_{v \notin \psi(C)} g(v)^2  \\
&\leq \sum_{v \in C} |f(v) - g(\psi(v))| + \sum_{v \notin \psi(C)} g(v)^2 + \sum_{v \notin A} f(v)  \\
&< d_\psi(f,g) + \kappa_2 < \delta + \kappa_2. \label{notphiA}
}
In light of \eqref{notphiA}, now \eqref{notPhiB} becomes
\eeq{
\EE\bigg[\sum_{u \notin \Phi(B)} \wt{g}(u)^{2}\bigg] < (\delta + \kappa_2) (2d)^2 e^{\lambda(2\beta)}. \label{nodenominator2}
}
The third and fourth inequalities to show consider the quantities $\wt{F}$ and $\wt{G}$.
Suppose $q \in (-\infty,0]\cup[1,\infty]$.
Since
\eq{
\sum_{u\in\N\times\Z^d}\sum_{v\sim u} f(v) + 2d(1-\|f\|) = 2d,
}
Jensen's inequality applied to the convex function $t\mapsto t^q$ gives
\eeq{ \label{Fbound}
\EE(\wt F^q)
&\leq (2d)^{q}\EE\bigg[\frac{1}{2d}\sum_{u\in\N\times\Z^d}\sum_{v\sim u}f(v) e^{q\beta Y_u} + (1-\|f\|){e^{q\lambda(\beta)}}\bigg] \\
&= (2d)^{q}(\|f\|e^{\lambda(q\beta)} + (1-\|f\|){e^{q\lambda(\beta)}}) \leq (2d)^{q}e^{\lambda(q\beta)}.
}
By the same argument with $g$ in place of $f$, we obtain the fourth desired inequality:
\eeq{
\EE(\wt{G}^{q}) \leq (2d)^qe^{\lambda(q\beta)}. \label{Gbound}
}
We will only use \eqref{Fbound} and \eqref{Gbound} when $|q|\leq2$, in accordance with our assumption \eqref{mgf}.

We now turn in earnest to bounding $\EE[d_\Phi(F,G)]$, which is the sum of four expectations corresponding to the four summands in definition \eqref{d_def}.
Seeking an estimate on the first summand, we observe that for $u \in B$,
\eeq{
|F(u) - G(\Phi(u))| &= \bigg|\frac{\wt{f}(u)}{\wt{F}} - \frac{\wt{g}(\Phi(u))}{\wt{G}}\bigg| \\
&\leq \bigg|\frac{\wt{f}(u)}{\wt{F}} - \frac{\wt{g}(\Phi(u))}{\wt{F}}\bigg| + \bigg|\frac{\wt{g}(\Phi(u))}{\wt{F}}-\frac{\wt{g}(\Phi(u))}{\wt{G}}\bigg|  \\
&= \frac{|\wt{f}(u) - \wt{g}(\Phi(u))|}{\wt{F}} + \frac{\wt{g}(\Phi(u))}{\wt{G}}\cdot\bigg|\frac{\wt{G}}{\wt{F}} - 1\bigg|. \label{two_parts}
}
Summing over $B$ and taking expectation, we have for the first term in \eqref{two_parts}:
\eeq{ 
\EE\sum_{u \in B} \frac{|\wt{f}(u) - \wt{g}(\Phi(u))|}{\wt{F}}
&= \sum_{u \in B} \EE\bigg[\frac{|\wt{f}(u) - \wt{g}(\Phi(u))|}{\wt{F}}\bigg]  \\
&\leq \sum_{u \in B} \sqrt{\EE(|\wt{f}(u) - \wt{g}(\Phi(u))|^2)\EE(\wt{F}^{-2})} \\
&\leq (2d)^{-1}\sum_{u \in B} \sqrt{e^{\lambda(2\beta)}e^{\lambda(-2\beta)}}\bigg| \sum_{v \sim u} f(v) - \sum_{v \sim \Phi(u)} g(v) \bigg|  \\
&< (2d)^{-1}\kappa_1 \sqrt{e^{\lambda(2\beta)}e^{\lambda(-2\beta)}},  \label{first_part}
}
where have applied Cauchy--Schwarz to obtain the first inequality, \eqref{first_observation} and \eqref{Fbound} to obtain the second, and \eqref{sumB} for the third.
Meanwhile, the second term in \eqref{two_parts} satisfies
\eeq{ 
\EE\bigg[\sum_{u \in B} \frac{\wt{g}(\Phi(u))}{\wt{G}}\cdot\Big|\frac{\wt{G}}{\wt{F}}-1\Big|\bigg]
&\leq \EE\, \Big|\frac{\wt{G}}{\wt{F}}-1\Big| \\
&= \EE\, \Big|\frac{\wt{G}-\wt{F}}{\wt{F}}\Big|\\
&\leq \sqrt{\EE(\wt{F}^{-2})\EE(|\wt{G} - \wt{F}|^2)}  \\
&\leq (2d)^{-1}\sqrt{e^{\lambda(-2\beta)}\EE(|\wt{G} - \wt{F}|^2)}, \label{second_part}
}
where we have again applied \eqref{Fbound}.
Notice that
\eq{
&\wt{F} - 2d \cdot {e^{\lambda(\beta)}} = \sum_{u \in \N \times \Z^d} \wt{f}(u) - 2d \cdot {e^{\lambda(\beta)}} \|f\| \\
&= \sum_{u \in \N \times \Z^d} \sum_{v \sim u} f(v) e^{\beta Y_u} - \sum_{u \in \N \times \Z^d} \sum_{v \sim u} f(v) {e^{\lambda(\beta)}} \\
&= \sum_{u \in \N \times \Z^d} \sum_{v \sim u} f(v)\big[e^{\beta Y_u} - {e^{\lambda(\beta)}}\big] \\
&= \sum_{u \in B} \sum_{v \sim u} f(v)\big[e^{\beta Y_u} - {e^{\lambda(\beta)}}\big]
+ \sum_{u \notin B} \sum_{v \sim u} f(v)\big[e^{\beta Y_u} - {e^{\lambda(\beta)}}\big],
}
and similarly
\eq{
&\wt{G} - 2d \cdot {e^{\lambda(\beta)}} \\
&= \sum_{u \in \Phi(B)} \sum_{v \sim u} g(v)\big[e^{\beta Z_u} - {e^{\lambda(\beta)}}\big] + \sum_{u \notin \Phi(B)} \sum_{v \sim u} g(v)\big[e^{\beta Z_u} - {e^{\lambda(\beta)}}\big] \\
&= \sum_{u \in B} \sum_{v \sim \Phi(u)} g(v)\big[e^{\beta Y_u} - {e^{\lambda(\beta)}}\big]  + \sum_{u \notin \Phi(B)} \sum_{v \sim u} g(v)\big[e^{\beta Z_u} - {e^{\lambda(\beta)}}\big].
}
Hence
\eq{
\wt{F} - \wt{G}
&= \sum_{u \in B}\bigg( \sum_{v \sim u} f(v) - \sum_{v \sim \Phi(u)} g(v)\bigg)\big[e^{\beta Y_u} - {e^{\lambda(\beta)}}\big] \\
&\phantom{=} + \sum_{u \notin B} \sum_{v \sim u} f(v)\big[e^{\beta Y_u} - {e^{\lambda(\beta)}}\big]
- \sum_{u \notin \Phi(B)} \sum_{v \sim u} g(v)\big[e^{\beta Z_u} - {e^{\lambda(\beta)}}\big].
}
We thus consider the following i.i.d.~random variables:
\eq{
Z_{u} &\coloneqq  e^{\beta Y_u} - {e^{\lambda(\beta)}}, \quad u \in B \\
Z'_{u} &\coloneqq  e^{\beta Y_u} - {e^{\lambda(\beta)}}, \quad u \notin B \\
Z''_{u} &\coloneqq  e^{\beta Z_u} - {e^{\lambda(\beta)}}, \quad u \notin \Phi(B).
} 
Furthermore, their coefficients
\eq{
\alpha_{u} &\coloneqq  \sum_{v \sim u} f(v) - \sum_{v \sim \Phi(u)} g(v), \quad u \in B\\
\alpha'_{u} &\coloneqq  \sum_{v \sim u} f(v), \quad u \notin B \\
\alpha''_{u} &\coloneqq  -\sum_{v \sim u} g(v), \quad u \notin \Phi(B)
}
satisfy
\eq{
|\alpha_u| &< \kappa_1 \quad &&\text{by \eqref{sumB},} \\
|\alpha_u'| &= \alpha'_u \leq \sum_{v \notin A} f(v) < \kappa_2 < \kappa_1 \quad &&\text{by \eqref{notAsum} and \eqref{kappa2_1},} \\
|\alpha_u''| &= -\alpha''_u \leq 2d \sup_{v \notin \phi(A)} g(v) \leq 2d\sqrt{\delta + \kappa_2} < \kappa_1 \quad &&\text{by \eqref{notphiA} and \eqref{delta5}}.
}
Since $\|f\|,\|g\| \leq 1$, one also has 
\eq{
\sum_{u \in B} |\alpha_u| + \sum_{u \notin B} \alpha'_u - \sum_{u \notin \Phi(B)} \alpha''_u \leq (1+1)\cdot 2d + 2d + 2d = 8d.
}
Therefore, 
\eq{ 
\EE\big[|\wt{F}-\wt{G}|^2\big] 
&= \EE\bigg[\bigg(\sum_{u \in B} \alpha_uZ_u + \sum_{u \notin B} \alpha'_uZ_u' + \sum_{u \notin \Phi(B)} \alpha''_uZ_u''\bigg)^2\bigg] \\
&= \bigg(\sum_{u \in B} \alpha_u^2+ \sum_{u \notin B} (\alpha'_u)^2+ \sum_{u \notin \Phi(B)} (\alpha''_u)^2\bigg)(e^{\lambda(2\beta)}-e^{2\lambda(\beta)})\\
&\leq \kappa_1 \bigg(\sum_{u\in B}|\alpha_u|+ \sum_{u \notin B} \alpha'_u - \sum_{u \notin \Phi(B)} \alpha''_u\bigg)(e^{\lambda(2\beta)}-e^{2\lambda(\beta)}) \\
&\leq \kappa_1 \cdot 8d \cdot e^{\lambda(2\beta)}.
}
Upon using the above display in \eqref{second_part}, we obtain
\eeq{
\EE\bigg[\sum_{u \in B} \frac{\wt{g}(\Phi(u))}{\wt{G}}\cdot\Big|\frac{\wt{G}}{\wt{F}}-1\Big|\bigg]
&\leq \EE\, \Big|\frac{\wt{G}}{\wt{F}}-1\Big|
\leq d^{-1/2}\sqrt{2\kappa_1{e^{\lambda(2\beta)}}{e^{\lambda(-2\beta)}}} \label{second_part2}.
}
Then applying \eqref{first_part} and \eqref{second_part2} with \eqref{two_parts}, we find
\eq{
\EE\bigg[\sum_{u \in B} |F(u) - G(\Phi(u))| \bigg] \leq 
\big((2d)^{-1}\kappa_1 + d^{-1/2}\sqrt{2\kappa_1}\,\big)\sqrt{{e^{\lambda(2\beta)}}{e^{\lambda(-2\beta)}}}.
}
Now \eqref{kappa1_1} gives us control on the first summand in $\EE[d_\Phi(F,G)]$:
\eeq{
\EE\bigg[2\sum_{u \in B} |F(u) - G(\Phi(u))| \bigg] < \frac{\eps}{4} \label{final_bound_1}.
}
From what we have argued, the second and third summands in $\EE[d_\Phi(F,G)]$ are not difficult to bound.
Notice that for any $u \in \N \times \Z^d$, the random variables $\wt{f}(u)$, $\wt{g}(u)$, $\wt{F}$, and $\wt{G}$ are non-decreasing functions of $Y_u$ or $Z_u$. 
Therefore, we can use the FKG inequality.
For the second summand, we apply FKG, then \eqref{nodenominator} and \eqref{Fbound}, and finally \eqref{kappa2_3}:
\eeq{
\EE\sum_{u \notin B} F(u)^{2} &\leq 
\EE\bigg[\sum_{u \notin B} \wt{f}(u)^{2}\bigg] \EE(\wt{F}^{-2}) 
< \kappa_2\, {e^{\lambda(2\beta)}}{e^{\lambda(-2\beta)}}
< \frac{\eps}{4}. \label{final_bound_2}
}
For the third summand, we appeal to FKG, then to \eqref{nodenominator2} and \eqref{Gbound}, and finally to \eqref{delta6}:
\eeq{
\EE\sum_{u \notin \Phi(B)} G(u)^{2} 
&\leq \EE\bigg[\sum_{u \notin \Phi(B)} \wt{g}(u)^{2}\bigg] \EE(\wt{G}^{-2}) \\
&< (\delta + \kappa_2){e^{\lambda(2\beta)}}{e^{\lambda(-2\beta)}}
< \frac{\eps}{4}. \label{final_bound_3}
}
The fourth and final summand in $\EE[d_\Phi(F,G)]$ depends on the maximum degree of the isometry $\Phi$.
From \eqref{delta4} we know $d_\psi(f,g) < \delta < \eps/16$, and so
$2^{-\deg(\psi)} < \eps/16$.
Since $\phi$ is a restriction of $\psi$, we have $\deg(\phi) \geq \deg(\psi)$.
And by Lemma \ref{extension}, $\deg(\Phi) \geq \deg(\phi) - 2$.
As a result,
\eeq{
2^{-\deg(\Phi)} \leq 2^{-\deg(\phi)+2} \leq 4 \cdot 2^{-\deg(\psi)} < \frac{\eps}{4} \label{final_bound_4}.
}
Summing the left-hand sides of \eqref{final_bound_1}--\eqref{final_bound_4}, we reach the desired conclusion for part (a):
\eq{
\w(Tf,Tg) \leq \EE[d_\Phi(F,G)] < \eps.
}

For part (b), we assume $q$ is a positive integer.
By several applications of Cauchy--Schwarz, we have
\eeq{ \label{part_b_many_CS} 
\big|\EE(\log^q \wt F) - \EE(\log^q \wt G)\big|
&\leq \EE\, |\log^q \wt F - \log^q \wt G| \\
&= \EE\, \bigg|\Big(\log\frac{\wt F}{\wt G}\Big)\sum_{i=0}^{q-1}(\log^{q-1-i} \wt F)(\log^i \wt G)\bigg| \\
&\leq \sqrt{\EE\log^2\frac{\wt F}{\wt G}}\sqrt{\EE\bigg[\bigg(\sum_{i=0}^{q-1}(\log^{q-1-i} \wt F)(\log^i \wt G)\bigg)^2\bigg]} \\
&\leq \sqrt{\EE\log^2\frac{\wt F}{\wt G}}\sqrt{q\sum_{i=0}^{q-1}\sqrt{\EE(\log^{4(q-1-i)} \wt F)\EE(\log^{4i} \wt G)}}. 
}
Now, by the inequality $\log^2 x \leq 2(|x-1| +|x^{-1}-1|)$ for $x>0$, we have
\eq{
\EE\log^2\frac{\wt F}{\wt G} \leq 2\Big(\EE\, \Big|\frac{\wt F}{\wt G}-1\Big| + \EE\, \Big|\frac{\wt G}{\wt F}-1\Big|\Big).
}
From \eqref{second_part2} and \eqref{kappa1_1}, we know $\EE\: \big|\frac{\wt{G}}{\wt{F}}-1\big| < \eps/8$.
A parallel analysis --- exchange $\EE(\wt{F}^{-2})$ and $\EE(\wt{G}^{-2})$ in the second expression of \eqref{second_part}, and then apply \eqref{Gbound} instead of \eqref{Fbound} ---  
also gives $ \EE\:\big|\frac{\wt{F}}{\wt{G}}-1\big| < \eps/8$.
On the other hand, for each integer $i\geq0$, there exists a large enough constant $C_i$ so that $\log^{4i} x \leq C_i(x+x^{-1})$ for all $x>0$.
Therefore, we can appeal to \eqref{Fbound} and \eqref{Gbound} with $q=\pm1$ in order to upper bound the second square root in the final expression of \eqref{part_b_many_CS}.
In summary,
\eq{
\big|\EE(\log^q \wt F) - \EE(\log^q \wt G)\big| \leq C\frac{\sqrt{\eps}}{2},
}
where $C$ is a constant that depends only on $\mathfrak{L}$, $\beta$, $d$, and $q$.
The continuity of the map $f \mapsto \EE(\log^q \wt F)$ is now clear.

\section{Open problems}
\label{open_problems}
Given the subsequent applications \cite{koenig-mukherjee15,bolthausen-koenig-mukherjee17} of Mukherjee and Varadhan's construction in \cite{mukherjee-varadhan16} and of a variant construction in \cite{mukherjee15}, as well as the success of the concentration compactness phenomenon in general, we are optimistic that the ideas presented in this study can fuel future work.
In fact, after the preprint release of this manuscript, our ``cavity approach" inspired similar works by Br{\"o}ker and Mukherjee for the stochastic heat equation and Gaussian multiplicative chaos \cite{broker-mukherjee19II}, and by Bakhtin and Seo \cite{bakhtin-seo20} for directed polymers with a continuous-space reference walk.

Below we mention some remaining questions concerning the methods of this manuscript.
For a broader discussion of  open problems on directed polymers, we refer the reader to Section 12.9 of~\cite{denHollander09}.
\begin{enumerate}
\item For the (1+1)-dimensional log-gamma model, localization occurs around a single favorite region \cite{comets-nguyen16}. Does this hold more generally, specifically in the low temperature regime? 
We showed via Theorem \ref{localized_subsequence}(b) and Proposition \ref{localization_thm} that the answer is yes if the single-copy condition holds.
Nevertheless, it is challenging to rule out the possibility that the endpoint distribution maintains an unbounded number of favorite sites, so that some $\nu \in \m$ may put mass on partitioned subprobability measures supported on infinitely many copies of $\Z^d$. Such behavior has been observed for other types of Gibbs measures \cite{comets-dembo01,barral-rhodes-vargas12}.
\item The set $\m$ defined by \eqref{M_def} is closed and convex in $\p(\s)$.
By Theorem \ref{characterization}, $\m$ is a singleton at high temperature.  
Is the same true in the low temperature phase?  If true, this would imply that the empirical measure converges to a deterministic limit, instead of the subsequential convergence that is established in this paper.
If false, can one describe the extreme points of $\m$?
\item Does the endpoint distribution converge in law? In other words, can we go beyond the Ces\`aro averages and prove limiting results for the actual endpoint distribution? Note that if this is true, then the limiting law must be an element of $\m$.
\item Can the methods of this paper be extended to understand the distinction between $(\SD)$ and $(\VSD)$? Such an extension may lead to the resolution of some longstanding questions in this area, such as the following. (a) For $d \geq 3$, do the critical values $\beta_{\mathrm{c}}$ and $\wt{\beta_{\mathrm{c}}}$ coincide?
This is known to be the case for $d = 1$ \cite{comets-vargas06} and for $d = 2$ \cite{lacoin10}.
If the answer is no, then there would exist \textit{critical strong disorder}, in which $(\SD)$ holds but not $(\VSD)$. (b) For $d \geq 3$, is there strong disorder at inverse temperature $\wt{\beta}_{\mathrm{c}}$?
For analogous models on self-similar trees, the answer is yes \cite{kahane-peyriere76}.
\end{enumerate}

\section*{Acknowledgements}
We thank David Aldous, Alexei Borodin, Persi Diaconis, Jafar Jafarov, Chiranjib Mukherjee, and Clifford Weil for fruitful discussions.
We are also grateful to  Erwin Bolthausen, Francis Comets, Tom Spencer, Srinivasa Varadhan, Vincent Vargas, and Zhengqing Zhou for useful comments on our preliminary draft.
We are also grateful to the referees for many useful suggestions and corrections.


%


\appendix

\section{Remaining technical details} \label{remaining_technical_details}
\subsection{Proof of Proposition \ref{superisometry}} \label{proof_superisometry}
The ``if" direction is straightforward to prove.
Suppose such a map $\psi$ exists.
For any $\eps > 0$, properties (ii) and (iii) allow us to take $A \subset B$ to be finite but large enough that
\eq{
\sum_{u \notin A} f(u)^2 + \sum_{u \notin \psi(A)} g(u)^2 < \eps. 
}
Let $\phi: A \to \N \times \Z^d$ be the restriction of $\psi$ to $A$.
Then (i) says $f(u) - g(\phi(u)) = 0$ for all $u \in A$, and (iv) shows $\deg(\phi) = \infty$.
Therefore, 
\eq{
d(f,g) \leq d_\phi(f,g) = \sum_{u \notin A} f(u)^2 + \sum_{u \notin \phi(A)} g(u)^2 < \eps.
}
As $\eps$ is arbitrary, we must have $d(f,g) = 0$.

Now we prove the converse.
Assume $d(f,g) = 0$.
Let $u_0$ be any element of $\N\times\Z^d$ with $f(u_0)>0$. 
With $\alpha \coloneqq  f(u_0)$, consider the sets $A \coloneqq  \{u \in\N\times\Z^d:f(u)=\alpha\}$ and $A' \coloneqq  \{u \in\N\times\Z^d:g(u)=\alpha\}$.
Because $\|g\|\le 1$, there is some $\delta \in(0,\alpha)$ such that 
if $g(u)\neq\alpha$, 
then $|\alpha-g(u)| \ge\delta$. Let $\eps = \min\{\alpha^2,2\delta\}$ and choose any isometry $\phi$ with finite domain $D$ such that $d_\phi(f,g)<\eps$. 
Note that $A \subset D$, since otherwise there would exist some $u \in A\setminus D$ so that $\eps >d_\phi(f,g) \geq f(u)^2=\alpha^2\ge \eps$. 
Moreover, every $u \in A$ must satisfy $g(\phi(u))=f(u)$, since otherwise
$\eps >d_\phi(f,g)\ge 2|f(u)-g(\phi(u))| \ge 2\delta\ge\eps$. 
Hence $\phi(A) \subset A'$, implying by injectivity of $\phi$ that the cardinality of $A$ is at most that of $C$.
Reversing the roles of $f$ and $g$ shows that these two sets actually have the same (finite) cardinality, meaning $\phi(A)=A'$. 
Since $u_0$ is arbitrary, it follows that the set of positive values of $f$ and those of $g$ are identical.
Denote this set by $V$.

First suppose $V$ is a finite set and let $\alpha$ denote the smallest element of $V$.
Take $\delta\in(0,\alpha)$ such that the distance between any two distinct elements of $V$ is at least $\delta$. 
Because $\|f\| \le 1$ and $\|g\|\le 1$, the sets
$B\coloneqq \{u : f(u)\in V\}$ and $B' \coloneqq  \{u : g(u)\in V\}$ are also finite. 
Therefore, there is an $m\in \N$ large enough that if $u,v\in B \cup B'$ with $\|u-v\|_1<\infty$, then $\|u-v\|_1<m$. 
Let $\eps = \min\{\alpha^2, 2\delta,2^{-m}\}$ and
choose any isometry $\phi$ with finite domain $D$ such that $d_\phi(f,g)<\eps$. 
By the same arguments as in the first paragraph, it follows that $B \subset D$, $g(\phi(u)) = f(u)$ for every $u \in B$, and $\phi(B) = B'$.
Hence $\psi \coloneqq  \phi\big|_{B}$ satisfies (i)--(iii).
Furthermore, it is immediate from the choice of $\eps$ that $\deg(\psi)\geq\deg(\phi)>m$. 
Consequently, whenever $u,v\in B$ satisfy $\|u-v\|_1<\infty$ (and thus $\|u-v\|_1<m$) or $\|\phi(u)-\phi(v)\|_1 < \infty$ (and thus $\|\phi(u)-\phi(v)\|_1 < m$), we have $\phi(u)-\phi(v) = u-v$. 
So $\psi$ also satisfies (iv).

Now suppose $V = \{\alpha_1 > \alpha_2 > \cdots\}$ is an infinite set.
For each $n\in \N$, let $V_n=\{\alpha_i : i=1,\dots, n\}$. 
Arguing as in the previous paragraph, we can find an isometry $\psi_n$ with domain $B_n\coloneqq \{u: f(u)\in V_n\}$ satisfying (i) and (iv) with $B$ replaced by $B_n$, and also satisfying $\psi_n(B_n) = B_n' \coloneqq \{u : g(u)\in V_n\}$. 
To complete the proof, we will identify a subsequence of $(\psi_n)_{n=1}^\infty$ that converges to the desired $\psi$. 
For each $n\ge1$, the function $\psi_n\big|_{B_1}$ maps $B_1$ bijectively onto $B_1'$.
Because there are only finitely many such mappings, there must be infinitely many $n_k$ that produce the same mapping.
Let $n_1$ denote the first such index (that is at least $2$) and take $\psi\big|_{B_1}$ to be equal to $\psi_{n_1}\big|_{B_1}$. 
Then repeat the argument along the subsequence $(\psi_{n_k})_{k=1}^\infty$ to define $\psi\big|_{B_2}$ as an extension of $\psi\big|_{B_2}$. 
Continued indefinitely, this procedure yields a map $\psi$ defined on all of $B \coloneqq  \bigcup_{n=1}^\infty B_n$, for which (i) and (iv) hold by construction.
Considering that $B=\{u : f(u)>0\}$ and 
\eq{
\psi(B)=\bigcup_{n=1}^\infty \psi(B_n) = \bigcup_{n=1}^\infty B_n' = \{u : g(u) > 0\},
}
(ii) and (iii) obviously hold as well.

\subsection{Proof of Theorem \ref{compactness_result}} \label{proof_compactness}
For each $n$, fix a representative $f_n \in S$. 
The elements of $\N \times \Z^d$ can be ordered so that $f_n$ attains its $k^{\text{th}}$ largest value at the $k^{\text{th}}$ element on the list.
Ties are broken according to some arbitrary but fixed rule (e.g.~assigning an order to $\Z^d$ and following the lexicographic ordering induced on $\N \times \Z^d$).
This information is recorded by defining
\eq{
u_{n, k} \coloneqq  u \in \N \times \Z^d \text{ at which $f_n$ attains its $k^{\text{th}}$ largest value}.
}
Of interest will be the quantities
\eq{
t_n(k,\ell) \coloneqq  u_{n, k} - u_{n, \ell} \in \Z^d \cup \{\infty\}.
}
Here we consider $\Z^d \cup \{\infty\}$ as the one-point compactification of $\Z^d$ (cf.~Steen and Seebach \cite{seebach-steen78}, page 63). 
Since this space is second countable, compactness is equivalent to sequential compactness (\cite{seebach-steen78}, page 22).
Consequently, by passing to a subsequence, we may assume that for every $k,\ell \in \N$, there is $t(k,\ell) \in \Z^d \cup \{\infty\}$ so that
\eq{
t_n(k,\ell) \to t(k,\ell) \quad \text{as } n \to \infty.
}
Because convergence takes place in the discrete space $\Z^d \cup \{\infty\}$, either $t(k,\ell) \in \Z^d$ or $t(k,\ell) = \infty$.
In the former case, $t_n(k,\ell) = t(k,\ell)$ for all $n$ sufficiently large.
On the other hand, $t(k,\ell) = \infty$ when either $t_n(k,\ell)$ is finite infinitely often but $\liminf_n \|t_n(k,\ell)\|_1 = \infty$, or $t_n(k,\ell) = \infty$ for all $n$ sufficiently large.

Since $\|f_n\| \leq 1$, we necessarily have
\eeq{
0 \leq f_n(u_{n, k}) \leq \frac{1}{k} \quad \text{for all } n,k. \label{f_n_bound}
}
So by passing to a further subsequence, it may be assumed that
\eeq{
\eps_k \coloneqq  \lim_{n\to\infty} f_n(u_{n, k}) \quad \text{exists for all } k, \label{eps_convergence}
}
where $\eps_k$ satisfies
\eeq{
0 \leq \eps_k \leq \frac{1}{k}. \label{eps_bound}
}
Next, inductively define
\eq{
\ell(1) &\coloneqq  1 \\
\ell(r) &\coloneqq  \min\{k > \ell(r-1)\ :\ t(k,\ell(s)) = \infty \text{ for all } s \leq r - 1\}, \quad r \geq 2.
}
It may be the case that only finitely many $\ell(r)$ can be defined, if the set considered in the definition is empty for some $r$.
In any case, let
\eq{
R \coloneqq  \text{number of $r$ for which $\ell(r)$ is defined} \leq \infty.
}
Clearly each $\ell(r)$ is distinct.
In words, $\ell(r)$ is the next integer past $\ell(r-1)$ such that the distance between $u_{n, \ell(r)}$ and each of $u_{n, \ell(1)}, u_{n, \ell(2)}, \dots ,u_{n, \ell(r-1)}$ is tending to $\infty$ with $n$.

Consider any fixed $k \in \N$.
There is some $r$ such that $t(k,\ell(r))$ is finite.
For instance, if $k = \ell(r)$, then $t(k,\ell(r)) = t(\ell(r),\ell(r)) = 0$.
If $\ell(r) < k < \ell(r+1)$, then $t(k,\ell(s))$ is finite for some $s \leq r$, by the definition of $\ell(r+1)$.
Similarly, if $R < \infty$ and $k > \ell(R)$, then $t(k,\ell(s))$ is finite for some $s \leq R$.
Now given the existence of $r$ with $t(k,\ell(r))$ finite, we claim there is a unique such $r$.
Indeed, if $t(k,\ell(r))$ and $t(k,\ell(r'))$ are both finite, then
\eq{
t(\ell(r),\ell(r')) 
&= \lim_{n \to \infty} (u_{n, \ell(r)} - u_{n, \ell(r')})\\
&= \lim_{n \to \infty} \bigl(-(u_{n, k} - u_{n, \ell(r)}) + (u_{n, k} - u_{n, \ell(r')}) \bigr)\\
&= -t(k,\ell(r)) + t(k,\ell(r'))
\ne \infty,
}
which forces $r = r'$.
This discussion shows it is possible to define
\eq{
r_k \coloneqq  \text{unique $r$ such that } t(k,\ell(r)) \text{ is finite.}
}
Let us summarize the construction thus far.
For each $k \in \N$, consider the point
\eq{
v_k \coloneqq  \Big(r_k, t\big(k,\ell(r_k)\big)\Big) \in \N \times \Z^d.
}
That is, $v_k$ is in the $r_k^{\text{th}}$ copy of $\Z^d$, at the finite limit point of $u_{n, k} - u_{n, \ell(r_k)}$.
Moreover, $r_k$ is the unique $r$ such that $u_{n, k} - u_{n, \ell(r)}$ converges to a finite limit.
So there are $R$ copies of $\Z^d$ populated by the $v_k$, the $r^{\text{th}}$ copy containing those $v_k$ for which $u_{n, k}$ and $u_{n, \ell(r)}$ remain close as $n$ grows large.
The actual point in that copy of $\Z^d$, at which $v_k$ is located, encodes the limiting difference $t(k,\ell(r))$, prescribing exactly how those two locations are separated for large $n$.
More precisely, there is some $N_k$ such that
\eeq{
u_{n, k} - u_{n, \ell(r_k)} = t\big(k,\ell(r_k)\big) \in \Z^d \quad \text{for all $n \geq N_k$}. \label{relative_finite}
}
Finally, recall that these locations give the order statistics on the values of $f_n$.
The first coordinate of $v_k$ indicates that the location of the $k^{\text{th}}$ largest value of $f_n$ remains close to the location of the $\ell(r_k)^{\text{th}}$ largest value of $f_n$, as $n$ tends to infinity.
The $R$ different locations of the $\ell(r)^{\text{th}}$ largest values of $f_n$ may not converge, but they serve as moving reference points to which all other locations remain close.

With the definitions made above, it is now possible to construct the limit function $f$.
Set
\eq{
f(u) \coloneqq  \begin{cases} 
\eps_k &\text{if } u = v_k, \\
0 &\text{else.}
\end{cases}
}
In order for $f$ to be well-defined, it must be checked that the $v_k$ are distinct.
Suppose $v_k = v_{k'}$.
That is, $r_{k} = r_{k'}$ and $t\big(k,\ell(r_k)\big) = t\big(k',\ell(r_{k'})\big)$, meaning that for all $n \geq \max\{N_k,N_{k'}\}$,
\eq{
u_{n, k} - u_{n, \ell(r_k)} 
= t\big(k,\ell(r_k)\big)
&= t\big(k',\ell(r_{k'})\big) \\
&= u_{n, k'} - u_{n, \ell(r_{k'})} 
= u_{n, k'} - u_{n, \ell(r_k)}.
}
It follows that $u_{n, k} = u_{n, k'}$ for all sufficiently large $n$.
That this holds for even one value of $n$ implies $k = k'$.

In order for $f$ to be an element of $S$, it must be checked that $\|f\| \leq 1$
($f$ is nonnegative due to \eqref{eps_bound}), which is a straightforward task:
Let $K \in \N$ be fixed.
Given $\eps > 0$, from \eqref{eps_convergence} we may choose $N$ large enough that
\eq{
|f_N(u_{N, k}) - \eps_k| \leq\frac{\eps}{K} \quad \text{for } k = 1,2,\dots,K.
}
Then
\eq{
\sum_{k = 1}^K \eps_k \leq \sum_{k = 1}^K \bigg(f_N(u_{N, k}) + \frac{\eps}{K}\bigg) \leq \|f_N\| + \eps \leq 1 + \eps.
}
As $\eps$ is arbitrary, it follows that
\eq{
\sum_{k = 1}^K f(v_k) = \sum_{k=1}^K \eps_k \leq 1.
}
Letting $K$ tend to infinity yields the bound $\|f\| \leq 1$.

In some sense, the remaining goal of the proof is to show that $f_n(v_k) \to \eps_k$.
Of course, this is not true.
Rather, the pointwise convergence of $(f_n)_{n \geq 1}$ is given by \eqref{eps_convergence}.
Nevertheless, we have the convergence \eqref{relative_finite} of \textit{relative} coordinates.
Furthermore, for any $k$ and $k'$ such that $r_k \neq r_{k'}$, we have $t(k,k') = \infty$. 
Therefore, for any $M > 0$, there is $N_{k, k'}$ such that
\eeq{
\|u_{n, k} - u_{n, k'}\|_1 \geq M \quad \text{for all $n \geq N_{k, k'}$.} \label{relative_infinite}
}
So the desired convergence will hold after pre-composing $f$ with a suitable isometry.
The correct choice of isometry is described below.

Temporarily fix $K \in \N$.
For $n \in \N$, define $\phi_{n, K} : \{u_{n, 1},u_{n, 2},\dots,u_{n, K}\} \to \N \times \Z^d$ by
\eq{
\phi_{n, K}(u_{n, k}) \coloneqq  v_k, \quad k = 1,2,\dots,K.
}
Consider the case when $n \geq N_k$ for all $k = 1,2,\dots,K$.
Then \eqref{relative_finite} guarantees that $u_{n, k} - u_{n, \ell(r_k)} = t\big(k,\ell(r_k)\big)$ for each $k = 1,2,\dots,K$.
Consequently, for $1 \leq k,k' \leq K$ we have
\eq{
r_k = r_{k'} \quad &\implies \quad u_{n, k} - u_{n, k'} = t\big(k,\ell(r_k)\big) - t\big(k',\ell(r_{k'})\big)\\
&\phantom{\implies \quad u_{n, k} - u_{n, k'}}\hspace{0.7ex} = v_k - v_{k'} 
= \phi_{n, K}(u_{n, k}) - \phi_{n, K}(u_{n, k'}).
}
If it is further the case that $n \geq N_{k, k'}$ for all $1 \leq k,k' \leq K$ with $r_k \neq r_{k'}$, then \eqref{relative_infinite} shows
\eq{
r_k \neq r_{k'} \quad &\implies \quad 
\|u_{n, k} - u_{n, k'}\|_1 \geq M \text{ and } \\
&\qquad \qquad \|\phi_{n, K}(u_{n, k}) - \phi_{n, K}(u_{n, k'})\|_1 = \|v_k - v_{k'}\|_1 = \infty.
}
Together, the previous two displays ensure $\deg(\phi_{n, K}) \geq M$ for large $n$.
Moreover, since $M$ in \eqref{relative_infinite} can be chosen arbitrarily large, we conclude that for fixed $K$,
\eeq{
\lim_{n \to \infty} \deg(\phi_{n, K}) = \infty. \label{degtoinf}
}
One can now verify convergence of $f_n$ to $f$ under the metric $d$.
To do so, one must exhibit, for every $\eps > 0$, some $N \in \N$ satisfying the following condition:
For every $n \geq N$, there is some finite $A_n \subset \N \times \Z^d$ and some isometry $\phi_n : A_n\to \N \times \Z^d$ satisfying $d_{\phi_n}(f_n,f) < \eps$.

So fix $\eps > 0$.
For each $K \in \N$, write $A_{n, K} = \{u_{n, 1},u_{n, 2},\dots,u_{n, K}\}$ for the domain of $\phi_{n, K}$.
Choose $K$ large enough that
\eq{
\sum_{k > K} \frac{1}{k^2} < \frac{\eps}{4}.
}
It follows from \eqref{f_n_bound} that
\eeq{
\sum_{k > K} f_n(u_{n, k})^2
< \frac{\eps}{4} \quad \text{for all $n$.} \label{final_bound1}
}
Similarly, from \eqref{eps_bound} we have
\eeq{
\sum_{k > K} f(v_k)^2
= \sum_{k > K} \eps_k^2
< \frac{\eps}{4}. \label{final_bound2}
}
Finally, in light of \eqref{eps_convergence} and \eqref{degtoinf}, it is possible to choose $N$ large enough that
\eeq{
\sum_{k = 1}^K |f_n(u_{n, k}) - \eps_k| < \frac{\eps}{8} \qquad \text{for all } n \geq N,\label{final_bound3}
}
and that
\eeq{
2^{-\deg(\phi_{n, K})} < \frac{\eps}{4} \quad \text{for all $n \geq N$}. \label{final_bound4}
}
Using \eqref{final_bound1}--\eqref{final_bound4}, one arrives at
\eq{
d_{\phi_{n, K}}(f_n,f) &=
2\sum_{u \in A_{n,K}} |f_n(u) - f(\phi_{n, K}(u))|
+ \sum_{u \notin A_{n, K}} f_n(u)^2\\
&\qquad  + \sum_{v \notin \phi_{n, K}(A_{n, K})} f(v)^2 + 2^{-\deg(\phi_{n, K})} \\
&= 2\sum_{k=1}^K |f_n(u_{n, k}) - \underbrace{f(v_k)}_{= \eps_k}|
+ \sum_{k > K} f_n(u_{n, k})^2 \\
&\qquad + \sum_{k > K} f(v_k)^2 + 2^{-\deg(\phi_{n, K})}  \\
&< \frac{\eps}{4} + \frac{\eps}{4} + \frac{\eps}{4} + \frac{\eps}{4} = \eps
}
for all $n \geq N$.
In particular,
\eq{
d(f_n,f) \leq d_{\phi_{n, K}}(f_n,f) < \eps \quad \text{for all } n \geq N.
}
One concludes $d(f_n,f) \to 0$ as $n \to \infty$, as desired.

\subsection{Proof of Proposition \ref{continuous2}} \label{continuous2_subsection}
Given $\eps > 0$, we may choose by Proposition \ref{continuous1} some $\delta > 0$ such that
\eeq{
d(f,g) < \delta \quad \implies \quad \w(Tf,Tg) < \frac{\eps}{2}. \label{from_before}
}
Now suppose $\mu,\nu \in \p(\s)$ satisfy $\w(\mu,\nu) < \delta \cdot \eps/4$.
Then there exists a coupling $(f,g) \in \s \times \s$ of the distributions $\mu$ and $\nu$ such that
$\EE[d(f,g)] < \delta \cdot \eps/4$.
Denote the joint distribution of $(f,g)$ by $\pi \in \Pi(\mu,\nu)$.
Markov's inequality gives
\eeq{
\PP(d(f,g) \geq \delta) \leq \delta^{-1} \EE[d(f,g)] < \frac{\eps}{4}. \label{measure_bound}
}
To bound $\w(\t\mu,\t\nu)$, we will use definition \eqref{kantorovich}. For any Lipschitz function $\vphi : \s \to \R$, we have
$\sup_{f \in \s} |\vphi(f)| < \infty$ by compactness of $\s$.
Considering such $\vphi$ with minimal Lipschitz constant at most 1, we can invoke \eqref{T_fubini} to write
\eeq{
&\int \vphi(h)\ \t\mu(\dd h) - \int \vphi(h)\ \t\nu(\dd h)\\ 
&= \iint \vphi(h)\ Tf(\dd h)\, \mu(\dd f) - \iint \vphi(h)\ Tg(\dd h)\, \nu(\dd g) \\
&= \int\bigg(\int \vphi(h)\ Tf(\dd h) - \int \vphi(h)\ Tg(\dd h)\bigg)\, \pi(\dd f,\dd g)  \\
&\leq \int \w(Tf,Tg)\ \pi(\dd f,\dd g) 
= \EE[\w(Tf,Tg)]. \label{each_vphi}
}
We now consider the expectation of $\w(Tf,Tg)$ over the set $\{d(f,g) < \delta\}$, where we can apply \eqref{from_before}, and separately over the complement $\{d(f,g) \geq \delta\}$, which has $\pi$-measure less than $\eps/4$ by \eqref{measure_bound}.
As the bound \eqref{each_vphi} holds for every $\vphi$, we have
\eq{
\w(\t\mu,\t\nu) 
&\leq \EE\big[\w(Tf,Tg)\big] \\
&\leq \EE\givenk[\big]{\w(Tf,Tg)}{d(f,g) < \delta} + \frac{\eps}{4}\, \EE\givenk[\big]{\w(Tf,Tg)}{d(f,g) \geq \delta} \\
&< \frac{\eps}{2} + \frac{\eps}{2} = \eps,
}
where we have applied Lemma \ref{trivial_bound} to bound the latter conditional expectation.

\subsection{Proof of Lemma \ref{eps_norm_equivalence}(a)} \label{proof_eps_norm}

It suffices to fix $f \in \s$, let $\delta_2 > 0$ be arbitrary, and find $\delta_1 > 0$ sufficiently small that
\eq{
d(f,g) < \delta_1 \quad \implies \quad \|g\|_\eps > \|f\|_\eps - \delta_2.
}
This task is trivial if $\|f\|_\eps = 0$.
Otherwise, we pick a representative $f \in S$, and consider the nonempty set
$A \coloneqq  \{u \in \N \times \Z^d : f(u) > \eps\}$.
Since $A$ is finite, $f$ achieves its minimum $t>\eps$ over $A$.
Now choose $\delta_1>0$ such that
\eq{
\delta_1 < \min\{\delta_2,\eps^2,t-\eps\}.
}
If $d(f,g) < \delta_1$, then there is a representative $g \in S$ and an isometry $\phi : C \to \N \times \Z^d$ such that $d_\phi(f,g) < \delta_1$.
It follows that $A \subset C$, since otherwise $d_\phi(f,g) \geq f(u)^2 > \eps^2 > \delta_1$ for some $u \in A \setminus C$.
Consequently,
\eq{
\sum_{u \in A} |f(u) - g(\phi(u))| \leq d_\phi(f,g) < \delta_1 < t - \eps.
}
In particular, for every $u \in A$,
\eq{
g(\phi(u)) \geq f(u) - |f(u) - g(\phi(u))| > t - (t - \eps) = \eps.
}
Hence
\eq{
\|g\|_\eps \geq \sum_{u \in A} g(\phi(u)) &\geq \sum_{u \in A} f(u) - \sum_{u \in A} |f(u) - g(\phi(u))| \\
&\geq \|f\|_\eps - d_\phi(f,g) > \|f\|_\eps - \delta_1 > \|f\|_\eps - \delta_2,
}
which completes the proof.

\subsection{Measurability of support number} \label{measurability_support_number}
For $f \in S$, let 
\eq{
H_f \coloneqq  \{n \in \N : f(n,x) > 0 \text{ for some $x \in \Z^d$}\}
}
denote the $\N$-support of $f$.
Recall that the support number of $f$ is $N(f) \coloneqq  |H(f)|$.

\begin{lemma} \label{N_meas1}
$N(\cdot) : S \to \N \cup \{0,\infty\}$ is measurable.
\end{lemma}

\begin{proof}
Fix $N \in \{0\} \cup \N$.
Let $U_N \coloneqq  \{f \in S : N(f) = N\}$ so that
\eq{
U_N &= \bigcup_{A \subset \N\, :\, |A| = N} \bigg[\bigg(\bigcap_{n \in A} \bigcup_{x \in \Z^d} \{f \in S : f(n,x) > 0\} \bigg) \\
&\qquad \qquad \qquad \qquad \cap \bigg(\bigcap_{n \in \N \setminus A} \bigcap_{x \in \Z^d} \{f \in S : f(n,x) = 0\}\bigg)\bigg].
}
For any $(n,x) \in \N \times \Z^d$, the map $f \mapsto f(n,x)$ is continuous and thus measurable.
Furthermore, all intersections and unions in the above display are taken over countable sets, and so $U_N$ is measurable.
It follows that $U_\infty \coloneqq  \{f \in S : N(f) = \infty\} = S \setminus \big(\bigcup_{N \geq 0} U_N\big)$ is also measurable.
We conclude that $N(\cdot) : S \to \N \cup \{0,\infty\}$ is a measurable function.
\end{proof}

By Corollary \ref{defined_pspm}, $N(\cdot)$ is well-defined on $\s$.
That is, if we define for $f \in S$ the set
\eq{
A_f \coloneqq \{g \in S : d(f,g) = 0\},
}
then $N(g) = N(f)$ for all $g \in A_f$.
It remains to show, however, that $N(\cdot) : \s \to \N \cup \{0,\infty\}$ is a measurable map.

\begin{lemma} \label{at_most_countable}
If $N(f) < \infty$, then the cardinality of $A_f$ is at most countably infinite.
\end{lemma}

\begin{proof}
Fix $f \in S$ such that $N(f) < \infty$.
By Corollary \ref{better_def_cor}, for any $g \in A_f$ there is $\sigma : H_f \to H_g$ and a set of vectors $(x_n)_{n \in H_f}$ such that \eqref{better_def} holds.
Notice that because $|H_f| < \infty$, there are only countably many choices for the set $\sigma(H_f) = H_g$, and only countably many possibilities for the $x_n$.
As $g$ is determined from $f$ by $H_g$ and $(x_n)_{n \in H_f}$, we conclude that $A_f$ is at most countably infinite.
\end{proof}

We will use the following fact about countable-to-one Borel maps.
These are (Borel) measurable functions $f : \x\to\y$ between two topological spaces such that for every $y\in\y$, $f^{-1}(\{y\})$ is at most countably infinite.

\begin{lemma}[see Srivastava \cite{srivastava98}, Theorem 4.12.4] \label{countable_to_one}
Suppose $\x,\y$ are Polish spaces and $F : \x \to \y$ is a countable-to-one Borel map. 
Then $F(B)$ is Borel for every Borel set $B$ in $\x$. 
\end{lemma}

Recall the quotient map $\iota : S \to \s$ defined in Lemma \ref{S_meas}(b), sending $f \in S$ to its equivalence class in $\s$.

\begin{prop} \label{N_meas2}
The map $N(\cdot) : \s \to \N \cup \{0,\infty\}$ is measurable.
\end{prop}

\begin{proof}
It suffices to show that for every $N \in \N \cup \{0,\infty\}$, the set
\eq{
\u_N \coloneqq  \{f \in \s : N(f) = N\}
}
is measurable.
We first restrict to the finite case.
For any nonnegative integer $N$, observe that $\x \coloneqq  \bigcup_{i = 0}^N U_i$ is a Polish space (under the $\ell^1$ metric) and, by Lemma \ref{N_meas1}, a measurable subset of $S$.
Lemma \ref{at_most_countable} says that $\iota |_{\x} : \x \to \s$ is a countable-to-one map.
Furthermore, this map is measurable by Lemma \ref{S_meas}(b).
Now Lemma \ref{countable_to_one} shows $\iota(U_N) = \u_N$ is a measurable subset of $\s$.

Given that $\u_N$ is measurable for every finite $N$, the set $\u_\infty = \s \setminus \big(\bigcup_{N \geq 0} \u_N\big)$ is also measurable.
We have thus shown $N(\cdot) : \s \to \N \cup \{0,\infty\}$ is measurable.
\end{proof}

\subsection{Proof of Lemma \ref{more_equivalences}} \label{proof_3_functionals}

For (a), we wish to show that for any $f \in \s$, there is $\eps > 0$ such that
\eq{
d(f,g) < \eps \quad \implies \quad W_\delta(g) \leq W_\delta(f).
}
It is only necessary to consider the case when $K \coloneqq  W_\delta(f)$ is finite (in particular, $f$ is nonzero).
Then $m(f) > 1 - \delta$, and we can select a representative $f \in S$ and $D \subset \Z^d$ such that
\eq{
\sum_{x \in D} f(1,x) > 1 - \delta, \qquad \diam(D) \leq K.
}
By possibly omitting some elements of $D$, we may assume $f$ is strictly positive on $\{1\} \times D$.
Choose $\eps > 0$ to satisfy the following three inequalities:
\begin{subequations}
\begin{align}
\eps &< \inf_{x \in D} f(1,x)^2 \label{eps_choice_1} \\
\eps &< 2^{-K} \label{eps_choice_2} \\
\sum_{x \in D} f(1,x) &> 1 - \delta + \eps. \label{eps_choice_3}
\end{align}
\end{subequations}
Suppose $g \in \s$ has $d(f,g) < \eps$.
Then there is some representative $g \in S$ and some isometry $\phi : A \to \N \times \Z^d$ so that $d_\phi(f,g) < \eps$.
It follows from \eqref{eps_choice_1} that $A \supset \{1\} \times D$.
Furthermore, \eqref{eps_choice_2} guarantees $\deg(\phi) > K$, implying
\eeq{
x,y \in D \quad \implies \quad \phi(1,x) - \phi(1,y) = x - y. \label{perfect_on_D}
}
By \eqref{eps_choice_3}, we now have
\eq{
\sum_{x \in D} g(\phi(1,x)) &\geq \sum_{x \in D} f(1,x) - \sum_{x \in D} |f(1,x) - g(\phi(1,x))| \\
&> 1 - \delta + \eps - \sum_{u \in A} |f(u) - g(\phi(u))| \\
&\geq 1 - \delta + \eps - d_\phi(f,g) > 1 - \delta.
}
Because of \eqref{perfect_on_D}, the set $\{\phi(1,x) : x \in D\}$ has diameter equal to $\diam(D)$, which is at most $K$, and so the above inequality shows $W_\delta(g) \leq K = W_\delta(f)$.

Next we prove (b) and (c) together, for which
we consider two cases.
First, if $m(f) = 1$, then $Q(f) = \infty$, and we must show that for any $\eps > 0$ and any $L > 0$, there is $\delta > 0$ such that
\eq{
d(f,g) < \delta \quad \implies \quad m(g) > 1 - \eps, \quad Q(g) > L.
}
In this case, it suffices to prove $\delta$ may be chosen so that $m(g) > 1 - \eps$, since any $\eps \leq 1/(L+1)$ gives
\eq{
m(g) > 1 - \eps \quad \implies \quad Q(g) > \frac{1-\eps}{1-(1-\eps)} \geq \frac{L/(L+1)}{1/(L+1)}= L.
}
The first step in doing so is to choose $A \subset \N \times \Z^d$ finite but sufficiently large that
\eeq{
\sum_{u \notin A} f(u) < \frac{\eps}{2}. \label{really_small_sum}
}
On the other hand, if $m(f) < 1$, so that $Q(f) < \infty$, then we shall find $\delta > 0$ satisfying
\eq{
d(f,g) < \delta \quad \implies \quad m(g) > m(f) - \eps, \quad Q(g) > Q(f) - \eps.
}
This is trivial if $f = \vc{0}$.
Otherwise, fix a representative $f \in S$, and again choose a finite $A \subset \N \times \Z^d$, but now satisfying a slightly different condition:
\eeq{
\frac{\sum_{u \notin A} f(u)}{(1-m(f))^2} < \frac{\eps}{2}. \label{really_small_sum2}
}
In either case --- \eqref{really_small_sum} or \eqref{really_small_sum2} --- we may assume $f$ is strictly positive on $A$ by possibly omitting some elements.
Next consider the integer
\eq{
K \coloneqq  \sup\{\|x-y\|_1 : (n,x),(n,y) \in A \text{ for some $n \in \N$}\},
}
and take $\delta > 0$ satisfying the following conditions:
\begin{subequations}
\begin{align}
\delta &< \inf_{u \in A} f(u)^2 \label{delta_condition1} \\
\delta &< 2^{-K} \label{delta_condition2} \\
\delta &< \frac{\eps}{2}. \label{delta_condition3}
\intertext{If $m(f) < 1$, we will further assume}
0 < \hspace{0.55in}&\hspace{-0.55in}\frac{\delta}{(1-m(f))(1-m(f)-\delta)} < \frac{\eps}{2}.
\label{delta_condition4}
\end{align}
\end{subequations}
Now, if $d(f,g) < \delta$, then there is a representative $g \in S$ and an isometry $\phi : C \to \N \times \Z^d$ such that $d_\phi(f,g) < \delta$.
By \eqref{delta_condition1}, we must have $A \subset C$.
Furthermore, \eqref{delta_condition2} ensures $\deg(\phi) > K$, meaning the following implication holds:
\eq{
(n,x),(n,y) \in A \quad &\implies \quad \|(n,x)-(n,y)\|_1 \leq K \\
&\implies \quad
\|\phi(n,x) - \phi(n,y)\|_1 = \|x-y\|_1 < \infty.
}
Consequently, for any $n$ such that $A \cap (\{n\} \times \Z^d)$ is nonempty, we can define $\sigma(n)$ to be the unique integer such that 
\eq{
\phi(n,x) \in \{\sigma(n)\} \times \Z^d \quad \text{whenever $(n,x) \in A$}.
}
Note that $n \mapsto \sigma(n)$ may not be injective.
We consider the quantities
\eq{
r_n \coloneqq  \sum_{x\, :\, (n,x) \in A} f(n,x), \qquad 
u_n &\coloneqq  \sum_{x\, :\, (n,x) \in A} g(\phi(n,x)), \\
U_n &\coloneqq  \sum_{m\, :\, \sigma(m) = \sigma(n)} u_m,
}
the first two of which satisfy
\eeq{
\sum_{n \in \N}|r_n-u_n|
&= \sum_{n \in \N}\bigg|\sum_{x\, :\, (n,x) \in A} f(n,x) - g(\phi(n,x))\bigg|\\
&\leq \sum_{u \in A} |f(u) - g(\phi(u))| < \delta. \label{first_two}
}
Also notice that \eqref{really_small_sum} or \eqref{really_small_sum2} implies
\eq{
m(f) - \max_{n \in \N} r_n < \frac{\eps}{2},
}
and so \eqref{first_two} and \eqref{delta_condition3} now give
\eq{
m(g) \geq \max_{n \in \N} U_n \geq \max_{n \in \N} u_n \geq \max_{n \in \N} r_n - \delta
> m(f) - \eps.
}
This completes the proof in the case $m(f) = 1$.
Otherwise, \eqref{delta_condition4} yields
\eq{ 
\sum_{n \in \N} \frac{r_n}{1-r_n} - \sum_{n \in \N} \frac{u_n}{1-u_n}
&= \sum_{n \in \N} \frac{r_n - u_n}{(1-r_n)(1-u_n)}\\
&< \frac{1}{(1-m(f))(1-m(f)-\delta)}\sum_{n \in \N}|r_n-u_n| \\
&< \frac{\delta}{(1-m(f))(1-m(f)-\delta)}
< \frac{\eps}{2}.
}
Also observe that
\eq{ 
\frac{U_n}{1-U_n} = \frac{\sum_{m\, :\, \sigma(m) = \sigma(n)} u_m}{1-\sum_{m\, :\, \sigma(m) = \sigma(n)} u_m}
\geq \sum_{m\, :\, \sigma(m) = \sigma(n)} \frac{u_m}{1-u_m}.
}
Together, the previous two displays show
\eeq{ \label{t_to_r}
\sum_{n \in \N} \frac{q_n(g)}{1-q_n(g)} \geq \sum_{n \in \N} \frac{U_n}{1-U_n} \geq \sum_{n \in \N} \frac{u_n}{1-u_n} > \sum_{n \in \N} \frac{r_n}{1-r_n} - \frac{\eps}{2}.
}
Finally, using \eqref{really_small_sum2} we find
\eeq{ \label{s_to_r}
\sum_{n \in \N} \frac{q_n(f)}{1-q_n(f)} - \sum_{n \in \N} \frac{r_n}{1-r_n} 
&= \sum_{n \in \N} \frac{q_n(f) - r_n}{(1 - q_n(f))(1-r_n)}\\
&\leq \frac{1}{(1-m(f))^2} \sum_{n \in \N} (q_n(f) - r_n )\\
&= \frac{\sum_{u \notin A} f(u)}{(1-m(f))^2}
< \frac{\eps}{2}.
}
Applying \eqref{s_to_r} in \eqref{t_to_r}, we obtain the desired result:
\eq{
Q(g) = \sum_{n \in \N} \frac{q_n(g)}{1-q_n(g)} > \sum_{n \in \N} \frac{q_n(f)}{1-q_n(f)} - \eps = Q(f) - \eps.
}

\subsection{Equivalent notions of asymptotic pure atomicity} \label{equivalent_notions_apa}
\begin{prop} \label{apa_either_way}
Assume \eqref{mgf}. Let $g_i(x) \coloneqq  \rho_{i-1}(\omega_i =x)$. 
Then the sequence $(g_i)_{i \geq 1}$ is asymptotically purely atomic if and only if $(f_i)_{i \geq 0}$ is asymptotically purely atomic.
\end{prop}

\begin{proof}
Let us write
\eq{
\b_i^\eps \coloneqq  \{x \in \Z^d : \rho_{i-1}(\omega_i = x) > \eps\}, \quad i \geq 1,\ \eps > 0.
}
Suppose $(f_i)_{i \geq 0}$ is asymptotically purely atomic.
Let $(\eps_i)_{i \geq 0}$ be any sequence tending to 0 as $i \to \infty$.
Then $2d \cdot \eps_i$ also tends to 0 as $i \to \infty$, and so
\eeq{
\lim_{n \to \infty} \frac{1}{n} \sum_{i = 0}^{n-1} \rho_i(\omega_i \in \a_i^{2d \cdot \eps_i}) = 1 \quad \mathrm{a.s.} \label{b_apa}
}
The observation
\eeq{
\rho_{i-1}(\omega_i = x) = \frac{1}{2d} \sum_{y \sim x} \rho_{i-1}(\omega_{i-1} = y) \label{next_step_repeat}
}
shows
\eq{
y \in \a_{i-1}^{2d\cdot\eps} \quad \implies \quad x \in \b_i^\eps \text{ for all $x$ such that $\|x-y\|_1 = 1$.}
}
Consequently, for any $\eps > 0$ we have the bound
\eq{
\rho_{i-1}(\omega_i \in \b_i^\eps) 
&\geq \sum_{y \in \a_{i-1}^{2d\cdot \eps}} \sum_{x \sim y} \rho_{i-1}(\omega_i = x,\omega_{i-1} = y)\\
&= \sum_{y \in \a_{i-1}^{2d\cdot \eps}} \rho_{i-1}(\omega_{i-1} = y) = \rho_{i-1}(\omega_{i-1} \in \a_{i-1}^{2d\cdot\eps}).
}
It now immediately follows from \eqref{b_apa} that
\eq{
\lim_{n \to \infty} \frac{1}{n} \sum_{i = 1}^{n} \rho_{i-1}(\omega_i \in \b_i^{\eps_i}) = 1 \quad \mathrm{a.s.}
}
So $(g_i)_{i \geq 1}$ is asymptotically purely atomic.

Conversely, assume $(f_i)_{i \geq 0}$ is not asymptotically purely atomic.
By Theorem \ref{total_mass}, we must have $0 \leq \beta \leq \beta_{\mathrm{c}}$, and so \eqref{bigger_to_0} holds for some sequence $(\eps_i)_{i \geq 0}$ tending to 0 as $i \to \infty$.
From \eqref{next_step_repeat}, we see that $\b_i^\eps$ is nonempty only when $\a_{i-1}^\eps$ is nonempty, in which case $\vc{I}_\eps(f_{i-1}) = 1$.
As a result,
\eq{
\rho_{i-1}(\omega_i \in \b_i^\eps) \leq \vc{I}_{\eps}(f_{i-1}) \quad \text{for any $i \geq 1$, $\eps > 0$,}
}
and so \eqref{bigger_to_0} forces
\eq{
\lim_{n \to \infty} \frac{1}{n} \sum_{i = 1}^{n} \rho_{i-1}(\omega_i \in \b_i^{\eps_{i-1}}) = 0 \quad \mathrm{a.s.}
}
Indeed, $(g_i)_{i \geq 1}$ is not asymptotically purely atomic.
\end{proof}

\section{Comparison to the Mukherjee--Varadhan topology} 
\label{compare_topologies}
\addtocontents{toc}{\protect\setcounter{tocdepth}{1}}
This appendix accomplishes two goals: (i) adapt the compactification technique of Mukherjee and Varadhan \cite{mukherjee-varadhan16} to measures on $\Z^d$, and (ii) prove that the metric in this adaptation, which is defined in terms of suitable test functions, is equivalent to the metric $d$.
None of the facts proved here are needed in the rest of our study, and the reader will not encounter any lapse of presentation by skipping this section entirely. 
Rather, the results of this section are included to verify that our methods may achieve the same effect as those initiated in \cite{mukherjee-varadhan16}, while offering a more tractable metric with which to work.
Indeed, this is one way our approach capitalizes on the countability of $\Z^d$, although
the discussion that follows underscores the possibility that the abstract machinery can be made more general.

\subsection{Adaptation to the lattice}
\addtocontents{toc}{\protect\setcounter{tocdepth}{1}}
In this preliminary section, we convert the Mukherjee--Varadhan setup to the discrete setting.
Aside from the proof of Proposition \ref{not_pseudo}, the construction is completely parallel to the one in \cite{mukherjee-varadhan16}.
For the sake of the ambitious reader, we mirror the notation of that manuscript as closely as possible.
We will use boldface $\vc{x} = (x_1,x_2,\dots,x_k)$ to denote a vector in $(\Z^d)^k$.
For $\vc{x} \in (\Z^d)^k$ and $z \in \Z^d$, we use the notation
\eq{
\vc{x}+z \coloneqq  (x_1+z,x_2+z,\ldots,x_k+z).
}
For an integer $k \ge 2$, call a function $W : (\Z^d)^k \to \R$ \textit{translation invariant} if
\eeq{ \label{trans_invariant}
W(\vc{x}+z) = W(\vc{x}) \quad \text{for all $\vc{x} \in (\Z^d)^k,\ z \in \Z^d$}.
}
We will say such a function $W$ \textit{vanishes at infinity} if
\eeq{ \label{vanish_inf}
\lim_{\max_{i \neq j} \|x_i - x_j\|_1 \to \infty} W(\vc{x}) = 0.
}
Now let $\i_k$ denote the set of functions $W : (\Z^d)^k \to \R$ that are both translation invariant and vanishing at infinity.
The space $\i_k$ is naturally equipped with a metric by the uniform norm,
\eq{
\|W\|_\infty \coloneqq  \sup_{\vc{x} \in (\Z^d)^k} |W(\vc{x})|.
}
The condition \eqref{trans_invariant} means that $W$ depends only on the $k-1$ variables $x_2-x_1,x_3-x_1,\dots,x_{k}-x_{1}$.
That is,
\eq{
W(x_1,x_2,\dots,x_k) = w(x_2-x_1,x_3-x_1,\dots,x_{k}-x_{1}),
}
where $w : (\Z^d)^{k-1} \to \R$ is given by
\eq{
w(y_1,y_2,\dots,y_{k-1}) = W(0,y_1,y_2,\dots,y_{k-1}).
}
Then \eqref{vanish_inf} is equivalent to
\eeq{ \label{vanish_inf_2}
\lim_{\max_{1 \le i \le k-1} \|y_i\|_1 \to \infty} w(\vc{y}) = 0.
}
Since the space of functions $w : (\Z^d)^{k-1} \to \R$ satisfying \eqref{vanish_inf_2} is separable, it follows that $\i_k$ is separable.

The space of test functions will be 
\eq{
\i \coloneqq  \bigcup_{k \ge 2} \i_k.
}
As each $\i_k$ is separable, we can find
a countable dense subset $(W_r)_{r \in \N}$ of $\i$,
where $W_r \in \i_{k_r}$.
Notice that for any $W \in \i_k$ and any $f \in \s$, the quantity
\eq{
I(W,f) \coloneqq  \sum_{n \in \N} \sum_{\vc{x} \in (\Z^d)^k} W(\vc{x}) \prod_{i = 1}^k f(n,x_i)
}
does not depend on the representative $f$ chosen from $S$.
(For the sake of exposition, we note that $I(W,f)$ is simply the sum of countably many integrals of $W$, the $n^{\text{th}}$ integral occurring  on the product space $((\Z^d)^k,f^{\otimes k}(n,\cdot))$,
where $f^{\otimes k}(n,\cdot)$ is the product measure whose every marginal has $f(n,\cdot)$ as its probability mass function.)
Indeed, if $f, g \in S$ are such that $d(f,g) = 0$, then by Lemma \ref{better_def_cor} there exists a bijection $\sigma$ between their $\N$-supports (denoted $H_f$ and $H_g$, respectively) and a collection $(z_n)_{n \in H_f}$ in $\Z^d$ such that
\eq{
f(n,x) = g(\sigma(n),x-z_n), \quad x \in \Z^d.
}
In this case, \eqref{trans_invariant} gives
\eq{
\sum_{n \in \N} \sum_{\vc{x} \in (\Z^d)^k} W(\vc{x}) \prod_{i = 1}^k f(n,x_i)
&= \sum_{n \in H_f} \sum_{\vc{x} \in (\Z^d)^k} W(\vc{x}) \prod_{i = 1}^k g(\sigma(n),x_i-z_n) \\
&= \sum_{n \in H_f} \sum_{\vc{x} \in (\Z^d)^k} W(\vc{x}+z_n) \prod_{i = 1}^k g(\sigma(n),x_i)\\
&= \sum_{n \in \N} \sum_{\vc{x} \in (\Z^d)^k} W(\vc{x}) \prod_{i = 1}^k g(n,x_i).
}
So $I(W,\cdot)$ is well-defined on $\s$.
We can thus define the following metric on $\s$:
\eq{
D(f,g) \coloneqq  \sum_{r = 1}^\infty \frac{1}{2^r}\frac{1}{1+\|W_r\|_\infty} |I(W_r,f) - I(W_r,g)|.
}
Since the family $\{I(\cdot,f): f \in \s\}$ is uniformly equicontinuous,
\eq{
|I(W_1,f) - I(W_2,f)| &= \bigg| \sum_{n \in \N} \sum_{\vc{x} \in (\Z^d)^k} (W_1(\vc{x}) - W_2(\vc{x})) \prod_{i = 1}^k f(n,x_i)\bigg| \\
&\leq \|W_1-W_2\|_\infty \sum_{n \in \N} \sum_{\vc{x} \in (\Z^d)^k} \prod_{i = 1}^k f(n,x_i) \\
&= \|W_1-W_2\|_\infty \sum_{n \in \N} \bigg(\sum_{x \in \Z^d} f(n,x)\bigg)^k \\
&\leq \|W_1-W_2\|_\infty \sum_{n \in \N} \sum_{x \in \Z^d} f(n,x) 
\leq \|W_1-W_2\|_\infty,
}
and $(W_r)_{r \in \N}$ is dense in $\i$, convergence in this metric implies convergence for \textit{all} test functions:
\eeq{ \label{convergence_criterion}
\lim_{j \to \infty} D(f_j,f) = 0  \iff
\lim_{j \to \infty} I(W,f_j) = I(W,f) \quad \text{for all $W \in \i$.}
}
It is clear that $D$ is reflexive and satisfies the triangle inequality.
It is nontrivial, however, that $D$ separates points.

\begin{prop} \label{not_pseudo}
For $f,g \in S$, $D(f,g) = 0$ if and only if $d(f,g) = 0$.
\end{prop}

Therefore, $D$ is indeed a metric on $\s$. The topology induced by $D$ on $\s$ is the lattice analog of the topology defined in Mukherjee and Varadhan~\cite{mukherjee-varadhan16}.

To prove Proposition \ref{not_pseudo}, we will use the lemma below.
Recall that a sequence of real numbers $(a_j)$ \textit{lexicographically dominates} $(a_j')$ 
(which we denote by $(a_j) \succeq (a_j')$) 
if $a_j > a_j'$ for the smallest $j$ for which $a_j \neq a_j'$. 
Of course, if there is no such $j$, then the two sequences are equal. If the two sequences are not equal, then we will say that $(a_j)$ strictly dominates $(a_j')$, and write $(a_j) \succ (a_j')$.

We will say a collection of sequences $\{(a_{i,\, j})\}$ is lexicographically descending ``in $i$" if
\eq{
i \leq i' \quad \implies \quad (a_{i,\, j}) \succeq (a_{i',\, j}).
}
Given an infinite collection of sequences $\{(a_{i,\, j}) : i \in \N\}$, it is not always possible to rearrange the $i$-indices so that the collection is lexicographically descending.
One can easily check, however, that rearrangement  is possible for \textit{nonnegative} sequences satisfying the following condition:
\eeq{
|\{i : a_{i,\, j} > \eps\}|,\,  |\{i : b_{i,\, j} > \eps\}| < \infty \quad \text{for all $\eps > 0$, $j \geq 1$.} \label{all_j_zero}
}

\begin{lemma} \label{powers_lemma}
Let $\{(a_{i,\, j})_{j = 1}^\infty : 1 \leq i \leq N_1\}$ and $\{(b_{i,\, j})_{j = 1}^\infty : 1 \leq i \leq N_2\}$ be two collections of sequences in $[0,1]$, where $N_1,N_2 \in \N \cup \{\infty\}$.
Suppose that \eqref{all_j_zero} holds and
\eeq{
a_{i,\, 1},\, b_{i,\, 1} > 0 \quad \text{for all $i$,} \label{first_terms_pos}
}
so that we may assume each collection is lexicographically descending in $i$.
If, for every $\ell \in \N$,
\eeq{
&\sum_{i = 1}^{N_1} a_{i,\, 1}\prod_{j = 1}^\ell a_{i,\, j}^{p_j} = \sum_{i = 1}^{N_2} b_{i,\, 1}\prod_{j = 1}^\ell b_{i,\, j}^{p_j} < \infty\\ 
&\quad \text{for all integers $p_1,\dots,p_\ell \geq 0$ with $\sum_{j=1}^\ell p_j \geq 1$,} \label{tests_agree}
}
then $N_1 = N_2$, and $a_{i,\, j} = b_{i,\, j}$ for every $i$ and $j$.
\end{lemma}

\begin{proof}
We will show by induction that for each finite $\ell$, $a_{i,\, \ell} = b_{i,\, \ell}$ for all $i$.
First consider the case when $\ell = 1$.
Since $a_{1,\, 1} = \max_i a_{i,\, 1}$ and $b_1 = \max_i b_{i,\, 1}$, we have
\eq{
a_{1,\, 1} = \lim_{p \to \infty} \bigg(\sum_{i = 1}^{N_1} a_{i,\, 1}^p\bigg)^{1/p}
\stackrel{\eqref{tests_agree}}{=} \lim_{p \to \infty} \bigg(\sum_{i = 1}^{N_2} b_{i,\, 1}^p\bigg)^{1/p}
= b_{1,\, 1}.
}
But then this argument can be repeated with the sequences $\{a_{i,\, 1} : 2 \leq i \leq N_1\}$ and $\{b_{i,\, 1} : 2 \leq i \leq N_2\}$ to obtain $a_{2,\, 1} = b_{2,\, 1}$.
Continuing in this way, one exhaustively determines that $N_1 = N_2 = N$, and $a_{i,\, 1} = b_{i,\, 1}$ for every $i$.

For $\ell \geq 2$, 
assume by induction that for each $j \leq \ell-1$, we have $a_{i,\, j} = b_{i,\, j}$ for every $i$.
By hypothesis \eqref{tests_agree}, for any nonnegative integers $q_1,q_2,\dots,q_{\ell-1}$ with $q_1 \geq 1$,
\eeq{ \label{max_agree}
\max_{i} \bigg(a_{i,\, \ell} \prod_{j = 1}^{\ell-1} a_{i,\, j}^{q_j}\bigg)
&= \lim_{p \to \infty} \bigg[\sum_{i = 1}^N \bigg( a_{i,\, \ell} \prod_{j = 1}^{\ell-1} a_{i,\, j}^{q_j}\bigg)^p\bigg]^{1/p} \\
&= \lim_{p \to \infty} \bigg[\sum_{i = 1}^N \bigg( b_{i,\, \ell} \prod_{j = 1}^{\ell-1} b_{i,\, j}^{q_j}\bigg)^p\bigg]^{1/p}\\
&= \max_{i} \bigg(b_{i,\, \ell} \prod_{j = 1}^{\ell-1} b_{i,\, j}^{q_j}\bigg).
}
We now specify $q_1,q_2,\dots,q_{\ell-1}$ via the following backward induction:
\begin{itemize}
\item If $a_{1,\, \ell-1} = 0$, then set $q_{\ell-1} = 0$.
Otherwise, take $q_{\ell-1} = 1$.
\item Given $q_{\ell-1},q_{\ell-2},\dots,q_{k+1}$, choose $q_k$ as follows:
\begin{itemize}
\item If $a_{1,\, k} = 0$, then set $q_{k} = 0$.
\item Otherwise, consider all $i$ such that the sequences $(a_{i,\, j})_{j = 1}^{\ell-1}$ and $(a_{1,\, j})_{j=1}^{\ell-1}$ first differ when $j = k$.
If it exists, the smallest such $i$, call it $i_k$, will maximize $a_{i,\, k} < a_{1,\, k}$  (in particular, $a_{1,\, k} > 0$).
We then take $q_k$ sufficiently large that
\eeq{
\Big(\frac{a_{i_k,\,k}}{a_{1,\, k}}\Big)^{q_k} < \prod_{j = k+1}^{\ell-1} a_{1,\, j}^{q_j}.
\label{prod_construction}
}
If no such $i$ exists, then set $q_k = 0$.
Notice that \eqref{first_terms_pos} and \eqref{all_j_zero} force $q_1 \geq 1$.
\end{itemize}
\end{itemize}
Having defined $q_1,q_2,\dots,q_{\ell-1}$ to satisfy \eqref{prod_construction}, we obtain the following implication:
\begin{align}
&(a_{1,\, j})_{j = 1}^{\ell-1} \succ (a_{i,\, j})_{j = 1}^{\ell-1} \nonumber \\
&\implies \quad
\exists\ k,\text{ $a_{1,\, j}$ and $a_{i,\, j}$ first differ at $j = k$} \label{differ}  \\
&\implies \quad
 \prod_{j = 1}^{\ell-1} \frac{a_{i,\, j}^{q_j}}{a_{1,\, j}^{q_j}}
 =  \prod_{j = k}^{\ell-1} \frac{a_{i,\, j}^{q_j}}{a_{1,\, j}^{q_j}}
\leq \Big(\frac{a_{i_k,\, k}}{a_{1,\, k}}\Big)^{q_k} \prod_{j = k+1}^{\ell-1} \frac{1}{a_{1,\, j}^{q_j}} < 1 \label{k_bound} \\
&\implies \quad \lim_{p \to \infty} \bigg(\prod_{j = 1}^{\ell-1} \frac{a_{i,\, j}^{q_j}}{a_{1,\, j}^{q_j}}\bigg)^p = 0. \label{ratio_to_0}
\end{align}
Moreover, because the first inequality in \eqref{k_bound} holds for all $i$ for which \eqref{differ} is true, the convergence in \eqref{ratio_to_0} is uniform over such $i$.
It follows that
\eq{
\lim_{p \to \infty} \max_{i} \bigg[a_{i,\, \ell} \bigg(\prod_{j = 1}^{\ell-1} \frac{a_{i,\, j}^{q_j}}{a_{1,\, j}^{q_j}}\bigg)^p\bigg] = a_{1,\, \ell}.
}
Since the choice of $q_1,q_2,\dots,q_{\ell-1}$ depended only on the $a_{i,\, j}$ with $j \leq \ell-1$, the induction hypothesis gives the same result for the $b$-collection:
\eq{
\lim_{p \to \infty} \max_{i} \bigg[b_{i,\, \ell}\bigg(\prod_{j = 1}^{\ell-1} \frac{b_{i,\, j}^{q_j}}{b_{1,\, j}^{q_j}}\bigg)^p\bigg] = b_{1,\, \ell}.
}
Now \eqref{max_agree}, with the fact that $a_{1,\, j} = b_{1,\, j}$ for $j \leq \ell-1$, allows us to conclude $a_{1,\, \ell} = b_{1,\, \ell}$.
As in the $\ell = 1$ case, we can repeat the above argument with the collections $\{(a_{i,\, j})_{j = 1}^\ell : 2 \leq i \leq N\}$ and $\{(b_{i,\, j})_{j = 1}^\ell : 2 \leq i \leq N\}$ to determine $a_{2,\, \ell} = b_{2,\, \ell}$.
Indeed, induction gives $a_{i,\, \ell} = b_{i,\, \ell}$ for every $i$.
\end{proof}

\begin{proof}[Proof of Proposition \ref{not_pseudo}]
The ``if" direction follows from the fact that $D$ is well-defined.
For the converse, we assume $D(f,g) = 0$ and prove $d(f,g) = 0$.
Let $\{z_1,z_2,\dots\}$ be any enumeration of $\Z^d$.
Given integers $\ell \geq 1$ and $p_1,p_2,\dots,p_\ell \geq 0$ with $k\coloneqq  1 + \sum_{j=1}^\ell p_j \geq 2$, consider the following member of $\i_k$:
\eq{
W(\vc{x}) \coloneqq  \begin{cases}
1 &\text{if }x_{i} - x_1 = z_j  \text{ for all $1+\sum_{t = 1}^{j-1} p_t < i \leq 1+ \sum_{t = 1}^{j} p_t$,}\\
& \quad 1 \leq j \leq \ell, \\
0 &\text{else}.
\end{cases}
}
Since $D(f,g) = 0$, we have
\eeq{
I(W,f) &= \sum_{u \in \N \times \Z^d} f(u) \prod_{j = 1}^\ell f(u+z_j)^{p_j}\\
&= \sum_{v \in \N \times \Z^d} g(v) \prod_{j = 1}^\ell g(v+z_j)^{p_j}
= I(W,g). \label{tests_agree_2}
}
Furthermore, these quantities are finite:
\eeq{ \label{finiteness}
\sum_{u \in \N \times \Z^d} f(u) \prod_{j = 1}^\ell f(u+z_j)^{p_j}
\leq \sum_{u \in \N \times \Z^d} f(u) \leq 1.
}
Now we lexicographically order (descending in $i$) the sequences given by
\eq{
a_{i,\, j} \coloneqq  f(u_i+z_j), \quad f(u_i) > 0, \qquad
b_{i,\, j} \coloneqq  g(v_i+z_j), \quad g(v_i) > 0,
}
for which \eqref{all_j_zero} is true because $\|f\|, \|g\| \leq 1$, and \eqref{tests_agree} is equivalent to \eqref{tests_agree_2}--\eqref{finiteness}.
Therefore, Lemma \ref{powers_lemma} shows
\eeq{
f(u_i+z_j) = g(v_i+z_j) \quad \text{for all $i$ and $j$}. \label{lemma_consequence}
}
Let $H_f$ and $H_g$ denote the $\N$-supports of $f$ and $g$, respectively.
Since $z_j$ ranges over all of $\Z^d$, \eqref{lemma_consequence} implies that for every $n \in N_f$, there is $m \in H_g$ such that $f(n,\cdot)$ and $g(m,\cdot)$ are translates of each another.
The proof that $d(f,g) = 0$ will be complete if we can show that for each $n \in H_f$, a \textit{distinct} $m \in H_g$ can be chosen, since then we would have an injection $\sigma : H_f \to H_g$ such that $f(n,\cdot)$ and $g(\sigma(n),\cdot)$ are always translates.
By interchanging $f$ and $g$, we would then see that $\sigma$ is necessarily a bijection, and so Lemma \ref{better_def_cor} gives $d(f,g) = 0$.

We now verify the final fact needed from above: $m \in H_g$ can be chosen distinctly for each $n \in H_f$.
Suppose that $f(n_1,\cdot),f(n_2,\cdot),\dots,f(n_K,\cdot)$ are translates of one another, where $n_1,\dots,n_K$ are distinct elements of $H_f$.
Even when $K$ is chosen maximally, $\|f\| \leq 1$ forces $K$ to be finite.
We can choose indices $i_1,i_2,\dots,i_K$ to simultaneously ``align" all these translates:
\begin{align}
&u_{i_k} \in \{n_k\} \times \Z^d  \text{ and } 
f(u_{i_k}+z_j) = f(u_{i_1}+z_j) \nonumber \\
&\qquad \text{for all $j \geq 1$, $k = 1,2,\dots,K$.} \nonumber
\end{align}
By \eqref{lemma_consequence}, we then have
\begin{align}
g(v_{i_k}+z_j) &= g(v_{i_1}+z_j)  \quad \text{for all $j \geq 1$, $k = 1,2,\dots,K$.} \label{duplicate_copies}
\end{align}
Since $n_1,\dots,n_K$ are distinct, so too are $i_1,\dots,i_K$, and therefore $v_{i_1},\dots,v_{i_K}$ are distinct.
Upon writing $v_{i_k} = (m_k,y_k)$, we claim that $m_k \neq m_\ell$ for $k \neq \ell$.
Indeed, if $m_k = m_\ell$, then distinctness forces $y_k \neq y_\ell$.
Therefore, $z\coloneqq y_\ell-y_k$ is nonzero and satisfies $v_{i_\ell} = v_{i_k}+z$.
In particular,
\eq{
g(v_{i_\ell}) = g(v_{i_k}+z) \stackrel{\eqref{duplicate_copies}}{=} g(v_{i_\ell}+z).
}
More generally, for any integer $q$,
\eq{
g(v_{i_\ell}+qz) = g(v_{i_k} + (q+1)z) \stackrel{\eqref{duplicate_copies}}{=} g(v_{i_\ell}+(q+1)z).
}
Since $g(v_{i_\ell}) > 0$, it follows that
\eq{
\sum_{q = 0}^\infty g(v_{i_\ell}+qz) = \sum_{q = 0}^\infty g(v_{i_\ell}) = \infty,
}
an obvious contradiction to $\|g\| \leq 1$.
Now each $f(n_k,\cdot)$ is a translate of $g(m_k,\cdot)$, and $m_1,m_2,\dots,m_K$ are all distinct, as desired.
\end{proof}

\subsection{Equivalence of the metrics}
The metric $D$ does, in fact, give rise to a compact topology on $\s$.
Rather than prove this directly, though, we first show that $d$ induces a topology at least as fine as the one induced by $D$.
As the continuous image of a compact set is compact, this immediately implies $D$ also induces a compact topology.  
But the result actually implies more: The topologies are necessarily equivalent.

\begin{prop} \label{equal_topologies}
$D(f_j,f) \to 0$ if and only if $d(f_j,f) \to 0$.
\end{prop}

The following topological fact reduces the proof of Proposition \ref{equal_topologies} to showing only one direction of the equivalence.

\begin{lemma}[see \cite{munkres00}, Theorem 26.6] \label{topology_fact}
Suppose $\x$ is a compact topological space, and $\y$ is a Hausdorff topological space. 
If $F : \x \to \y$ is bijective and continuous, then $F$ must be a homeomorphism.
\end{lemma}

In the present setting, we consider $\x = (\s,d)$ and $\y = (\s,D)$.
With $F$ equal to the identity map, Lemma \ref{topology_fact} says the following: If $d(f_j,f) \to 0$ implies $D(f_j,f) \to 0$, then the converse is also true.

\begin{proof}[Proof of Proposition \ref{equal_topologies}]
By Lemma \ref{topology_fact}, it suffices to show that if $d(f_j,f)$ converges to $0$, then $D(f_j,f) \to 0$.
And by \eqref{convergence_criterion}, it suffices to check that given any $W \in \i$, we have
$I(W,f_j) \to I(W,f)$.
So consider any $W \in \i_k$, and let $\eps > 0$ be given.
We seek a number $\delta > 0$ such that for any $g \in \s$,
\eq{
d(f,g) < \delta \quad \implies \quad |I(W,f) - I(W,g)| < \eps. 
}
This is trivial if $W$ is constant zero, and so we will henceforth assume $\|W\|_\infty > 0$.
By \eqref{vanish_inf}, there is $K$ large enough that
\eeq{ \label{K_choice}
\max_{i \neq j} \|x_i - x_j\|_1 \geq K \quad \implies \quad |W(\vc{x})| < \frac{\eps}{8}.
}
Upon fixing a representative $f \in S$, we can take $N \in \N$ such that
\eeq{ \label{small_after_N}
\sum_{n = N+1}^\infty \sum_{x \in \Z^d} f(n,x) < \frac{\eps}{8k\|W\|_\infty}.
} 
Next, for each $1 \leq n \leq N$, we choose $A_n \subset \Z^d$ finite but large enough that
\eeq{ \label{small_before_N}
\sum_{x \notin A_n} f(n,x) < \frac{\eps}{8kN\|W\|_\infty}.
}
Now define $A \coloneqq  \bigcup_{n = 1}^N (\{n\} \times A_n)$, so that \eqref{small_after_N} and \eqref{small_before_N} together show
\eeq{ \label{small_total}
\sum_{u \notin A} f(u) < \frac{\eps}{4k\|W\|_\infty}.
}
By possibly omitting some elements of $A$ and/or taking $N$ smaller, we may assume $f$ is strictly positive on $A$.
We may also assume 
\eq{
K \geq \sup_{1 \leq n \leq N} \diam(A_n),
}
since \eqref{K_choice} still holds if $K$ is made larger.
Now we choose $\delta > 0$ satisfying
\begin{align}
\delta &< \inf_{u \in A} f(u)^2, \label{delta_condition_1} \\
\delta &< 2^{-K}, \label{delta_condition_2} \\
\delta &< \frac{\eps}{8k\max\{1,|A|^{k-1}\}\|W\|_\infty}, \label{delta_condition_3} \\
\sqrt{\delta} &< \frac{\eps}{8k(2K)^{(k-1)d}\|W\|_\infty} \label{delta_condition_4}.
\end{align}
If $d(f,g) < \delta$, then there is a representative $g \in S$ and an isometry $\phi : C \to \N \times \Z^d$ such that $d_\phi(f,g) < \delta$.
The condition \eqref{delta_condition_1} implies $A \subset C$, while \eqref{delta_condition_2} guarantees that 
\eeq{
\deg(\phi) > K \geq \diam(A_n) \quad \text{for any $n \leq N$.} \label{deg_big_enough}
}
Consequently, $\phi$ acts by translation on $A_n$. 
That is, for each $n = 1,2,\dots,N$, there is $\sigma(n) \in \N$ and $z_n \in \Z^d$ so that
\eeq{
\phi(n,x) = (\sigma(n),x + z_n) \quad \text{for all $x \in A_n$.} \label{An_translate}
}
(Here $n \mapsto \sigma(n)$ may not be injective.) We thus have
\eeq{ \label{triple_sum}
|I(W,f) - I(W,g)| &=
\bigg|\sum_{n \in \N} \sum_{\vc{x} \in (\Z^d)^k} W(\vc{x}) \prod_{i = 1}^k f(n,x_i)\\
&\qquad - \sum_{n \in \N} \sum_{\vc{x} \in (\Z^d)^k} W(\vc{x}) \prod_{i = 1}^k g(n,x_i)\bigg| \\
&\leq D_1 + D_2 + D_3,
}
where
\begin{align}
D_1 &\coloneqq  \bigg|\sum_{n \in \N} \sum_{\vc{x} \in (\Z^d)^k} W(\vc{x}) \prod_{i = 1}^k f(n,x_i)
- \sum_{n = 1}^N \sum_{\vc{x} \in A_n^k} W(\vc{x}) \prod_{i = 1}^k f(n,x_i)\bigg|, \label{D1} \\
D_2 &\coloneqq  \bigg|\sum_{n = 1}^N \sum_{\vc{x} \in A_n^k} W(\vc{x}) \prod_{i = 1}^k f(n,x_i)\nonumber\\
&\qquad 
- \sum_{n = 1}^N \sum_{\vc{x} \in A_n^k} W(\vc{x}+z_n) \prod_{i = 1}^k g(\sigma(n),x_i+z_n)\bigg|, \nonumber \\
D_3 &\coloneqq  \bigg|\sum_{n = 1}^N \sum_{\vc{x} \in A_n^k} W(\vc{x}+z_n) \prod_{i = 1}^k g(\sigma(n),x_i+z_n)\nonumber \\
&\qquad - \sum_{n \in \N} \sum_{\vc{x} \in (\Z^d)^k} W(\vc{x}) \prod_{i = 1}^k g(n,x_i)\bigg|. \nonumber
\end{align}
We shall produce an upper bound for each of $D_1$, $D_2$, and $D_3$, and then use \eqref{triple_sum} to yield the desired result.

First, for $D_1$, notice that the summand  $W(\vc{x}) \prod_{i = 1}^k f(n,x_i)$ will appear in the first sum of \eqref{D1} but not the second if and only if $\vc{x} \notin A_{n}^k$ or $n > N$.
Considering the first case, we observe that $\vc{x} \notin A_n^k$ if and only if some $x_j$ does not belong to $A_n$. 
Hence
\eeq{ \label{bound_1_1}
&\bigg| \sum_{n = 1}^N \sum_{\vc{x} \notin A_n^k} W(\vc{x}) \prod_{i = 1}^k f(n,x_i) \bigg|\\
&\leq \|W\|_\infty \sum_{n = 1}^N \sum_{\vc{x} \notin A_n^k} \prod_{i = 1}^k f(n,x_i) \\
&= \|W\|_\infty \sum_{n = 1}^N \sum_{j = 1}^k \sum_{x_j \notin A_n} f(n,x_j) \sum_{\vc{y} \in (\Z^d)^{k-1}} \prod_{i = 1}^{k-1} f(n,y_i) \\
&= \|W\|_\infty \sum_{n = 1}^N \sum_{j = 1}^k \sum_{x \notin A_n} f(n,x) \bigg(\sum_{y \in \Z^d} f(n,y)\bigg)^{k-1} \\
&\leq k\|W\|_\infty \sum_{n = 1}^N \sum_{x \notin A_n} f(n,x)
< \frac{\eps}{8},
}
where the final inequality is a consequence of \eqref{small_before_N}. 
Considering the second case, we have
\eeq{ \label{bound_1_2}
&\bigg| \sum_{n = N+1}^\infty \sum_{\vc{x} \in (\Z^d)^k} W(\vc{x}) \prod_{i = 1}^k f(n,x_i) \bigg|\\
&\leq \|W\|_\infty \sum_{n = N+1}^\infty \sum_{\vc{x} \in (\Z^d)^k} \prod_{i = 1}^k f(n,x_i) \\
&= \|W\|_\infty \sum_{n = N+1}^\infty \bigg(\sum_{x \in \Z^d} f(n,x)\bigg)^k \\
&\leq \|W\|_\infty \sum_{n = N+1}^\infty \sum_{x \in \Z^d} f(n,x) < \frac{\eps}{8k} < \frac{\eps}{8},
}
where we have used \eqref{small_after_N} in the penultimate inequality.
Together, \eqref{bound_1_1} and \eqref{bound_1_2} yield
\eeq{ \label{bound_1}
D_1 < \frac{\eps}{4}.
}
Next we analyze the second difference, $D_2$, from \eqref{triple_sum}.
Here we use the translation invariance of $W$:
Making use of \eqref{An_translate}, we determine that
\eeq{ \label{before_telescope}
D_2 &= \bigg|\sum_{n = 1}^N \sum_{\vc{x} \in A_n^k} W(\vc{x}) \prod_{i = 1}^k f(n,x_i)\\
&\qquad - \sum_{n = 1}^N \sum_{\vc{x} \in A_n^k} W(\vc{x}+z_n) \prod_{i = 1}^k g(\sigma(n),x_i+z_n)\bigg| \\
&= \bigg|\sum_{n = 1}^N \sum_{\vc{x} \in A_n^k} W(\vc{x}) \prod_{i = 1}^k f(n,x_i)
- \sum_{n = 1}^N \sum_{\vc{x} \in A_n^k} W(\vc{x}) \prod_{i = 1}^k g(\phi(n,x_i))\bigg| \\
&\leq \|W\|_\infty \sum_{n = 1}^N \sum_{\vc{x} \in A_n^k} \bigg|\prod_{i = 1}^k f(n,x_i) - \prod_{i = 1}^k g(\phi(n,x_i))\bigg|.
}
For $1 \leq n \leq N$ and $\vc{x} \in A_n^k$, we can use a telescoping sum to write
\eq{
&\bigg|\prod_{i = 1}^k f(n,x_i) - \prod_{i = 1}^k g(\phi(n,x_i))\bigg| \\
&= \bigg|\sum_{i = 1}^k f(n,x_1)\cdots f(n,x_{i-1})\big(f(n,x_i) - g(\phi(n,x_i))\big) \\
&\qquad \qquad \cdot g(\phi(n,x_{i+1}))\cdots g(\phi(n,x_k))\bigg| \\
&\leq\sum_{i = 1}^k |f(n,x_i) - g(\phi(n,x_i))|.
}
Therefore, \eqref{before_telescope} becomes 
\eq{
D_2 \leq \|W\|_\infty \sum_{n = 1}^N \sum_{\vc{x} \in A_n^k} \sum_{i = 1}^k |f(n,x_i) - g(\phi(n,x_i))|.
}
Now, given any $u = (n,x) \in A$, the summand $|f(u) - g(\phi(u))|$ will appear $k|A_n|^{k-1}$ times in the above sum: There must be some $i$ for which $x_i = x$, and the remaining $k-1$ coordinates of $\vc{x}$ can be any elements of $A_n$.
Using this fact and \eqref{delta_condition_3}, we arrive at
\eeq{ \label{bound_2}
D_2 &\leq k|A|^{k-1}\|W\|_\infty\sum_{u \in A} |f(u) - g(\phi(u))| < \frac{\eps}{8}.
}
Finally, we need to bound the third difference, $D_3$, in \eqref{triple_sum}.
Recall the map $n \mapsto \sigma(n)$ from \eqref{An_translate}.
For each $\ell \in \N$, consider the partition of $(\Z^d)^k = J_1(\ell) \cup J_2(\ell)$, where
\eq{
J_1(\ell) &\coloneqq  \bigcup_{1 \leq n \leq N\, :\, \sigma(n) = \ell} \{\vc{x} +z_n : \vc{x} \in A_n^k\}, \qquad
J_2(\ell) \coloneqq  (\Z^d)^k \setminus J_1(\ell).
}
In this notation, the product $W(\vc{y}) \prod_{i = 1}^k g(\ell,y_i)$ appears in the sum
\eq{
\sum_{n = 1}^N \sum_{\vc{x} \in A_n^k} W(\vc{x}+z_n) \prod_{i = 1}^k g(\sigma(n),x_i+z_n)
}
if and only if $\vc{y} \in J_1(\ell)$.
Furthermore, in this case it will appear exactly once, since
\eq{
&\big\{(\sigma(n),x+z_n):x \in A_n\big\} \cap \big\{(\sigma(m),x+z_m):x \in A_m\big\} \\
&= \phi(\{n\} \times A_n) \cap \phi(\{m\} \times A_m) = \varnothing
} 
for $n \neq m$.
Therefore,
\eeq{ \label{bound_3_rewrite}
D_3 &= \bigg|\sum_{n = 1}^N \sum_{\vc{x} \in A_n^k} W(\vc{x}+z_n) \prod_{i = 1}^k g(\sigma(n),x_i+z_n)\\
&\qquad - \sum_{n \in \N} \sum_{\vc{x} \in (\Z^d)^k} W(\vc{x}) \prod_{i = 1}^k g(n,x_i)\bigg| \\
&= \bigg|\sum_{\ell \in \N} \sum_{\vc{x} \in J_2(\ell)} W(\vc{x}) \prod_{i = 1}^k g(\ell,x_i)\bigg|.
}
To analyze this quantity, we consider the further partition $J_2(\ell) = J_3(\ell) \cup J_4(\ell)$, where
\eq{
J_3(\ell) \coloneqq  \Big\{\vc{x} \in J_2(\ell) : \max_{i \neq j} \|x_i - x_j\|_1 \geq K\Big\}, \qquad
J_4(\ell) \coloneqq  J_2(\ell) \setminus J_3(\ell).
}
The sum over the various $J_3(\ell)$ is easy to control because of \eqref{K_choice}:
\eeq{ \label{bound_3_1}
\bigg|\sum_{\ell \in \N} \sum_{\vc{x} \in J_3(\ell)} W(\vc{x}) \prod_{i = 1}^k g(\ell,x_i)\bigg|
&< \frac{\eps}{8} \sum_{\ell \in \N} \sum_{\vc{x} \in J_3(\ell)} \prod_{i = 1}^k g(\ell,x_i)\\
&\leq \frac{\eps}{8} \sum_{\ell \in \N} \sum_{\vc{x} \in (\Z^d)^k} \prod_{i = 1}^k g(\ell,x_i) \\
&= \frac{\eps}{8} \sum_{\ell \in \N} \bigg(\sum_{x \in \Z^d} g(\ell,x)\bigg)^k \\
&\leq \frac{\eps}{8} \sum_{\ell \in \N} \sum_{x \in \Z^d} g(\ell,x)
\leq \frac{\eps}{8}.
}
Next considering $J_4(\ell)$, we make the following observation.
If $\vc{x} \in J_4(\ell)$, then there must be some coordinate $x_i$ for which $(\ell,x_i) \notin \phi(A)$.
Indeed, if $(\ell,x_i) = \phi(n,x_i')$ and $(\ell,x_j) = \phi(m,x_j')$ both belong to $\phi(A)$ (that is, $x_i' \in A_n$ and $x_j' \in A_m$), then
\eq{
\vc{x} \in J_4(\ell) \quad &\stackrel{\phantom{\eqref{deg_big_enough}}}{\implies} \quad
\|(\ell,x_i) - (\ell,x_j)\|_1 < K \\
&\stackrel{\eqref{deg_big_enough}}{\implies}\quad \|(n,x_i') - (m,x_j')\|_1 < K < \infty
\quad \implies \quad n = m.
}
Therefore, if it were the case that every coordinate $x_i$ satisfied $(\ell,x_i) \in \phi(A)$, then $\vc{x}$ would belong to $A_n^k + z_n$ for some $n$ such that $\ell = \sigma(n)$.
This contradicts $J_4(\ell) \cap J_1(\ell) = \varnothing$.

We now consider one final (non-disjoint) partition: $J_4(\ell) = J_5(\ell) \cup J_6(\ell)$, where
\eq{
J_5(\ell) &\coloneqq  \{\vc{x} \in J_4(\ell) : (\ell,x_i) \in \phi(C \setminus A) \text{ for some $i$}\},
\intertext{and}
J_6(\ell) &\coloneqq  \{\vc{x} \in J_4(\ell) : (\ell,x_i) \notin \phi(C) \text{ for some $i$}\}.
}
\textit{A priori}, these definitions only imply $J_5(\ell) \cup J_6(\ell) \subset J_4(\ell)$, but the observation of the previous paragraph ensures that $J_5(\ell) \cup J_6(\ell) = J_4(\ell)$.
For the sum over the $J_5(\ell)$, there is a straightforward upper bound:
\eeq{
&\bigg|\sum_{\ell \in \N} \sum_{\vc{x} \in J_5(\ell)} W(\vc{x}) \prod_{i = 1}^k g(\ell,x_i)\bigg|\\
&\leq \|W\|_\infty \sum_{\ell \in \N} \sum_{j = 1}^k \sum_{\vc{x} \in (\Z^d)^k\, :\, (\ell,x_j) \in \phi(C\setminus A)} \prod_{i = 1}^k g(\ell,x_i) \\
&= \|W\|_\infty \sum_{\ell \in \N} \sum_{j = 1}^k \sum_{x_j \in \Z^d\, :\, (\ell,x_j) \in \phi(C \setminus A)} g(\ell,x_j) \sum_{\vc{y} \in (\Z^d)^{k-1}} \prod_{i = 1}^{k-1} g(\ell,y_i) \\
&= k\|W\|_\infty \sum_{\ell \in \N} \sum_{x \in \Z^d\, :\, (\ell,x) \in \phi(C \setminus A)} g(\ell,x)\bigg(\sum_{y \in \Z^d} g(\ell,y)\bigg)^{k-1} \\
&\leq k\|W\|_\infty \sum_{\ell \in \N} \sum_{x \in \Z^d\, :\, (\ell,x) \in \phi(C \setminus A)} g(\ell,x) \\
&= k\|W\|_\infty \sum_{u \in C \setminus A} g(\phi(u)) \\
&\leq k\|W\|_\infty \bigg(\sum_{u \in C \setminus A} f(u) + \sum_{u \in C \setminus A} |f(u) - g(\phi(u))|\bigg) \\
&\leq k\|W\|_\infty\Big(\frac{\eps}{4k\|W\|_\infty} + d_\phi(f,g)\Big)
< \frac{3\eps}{8} \label{bound_3_2},
}
where we have used \eqref{small_total} and \eqref{delta_condition_3} to establish the final two inequalities.
Now turning our focus to the sum over the $J_6(\ell)$, we define for each $x_1 \in \Z^d$ the set
\eq{
\n(x_1) \coloneqq  \Big\{(x_2,\dots,x_{k}) \in (\Z^d)^{k-1} : \max_{1 \leq i < j \leq k} \|x_i - x_j\|_1 < K \Big\}.
}
Note that
\eq{
|\n(x_1)| \leq (2K)^{(k-1)d},
}
since there are no more than $(2K)^d$ elements of $\Z^d$ at distance less than $K$ from $x_1$, and each of $x_2,\dots,x_{k}$ must satisfy this property.
By definition, if $\vc{x} \in J_4(\ell)$, then for every $j$ we have 
$(x_1,\dots,x_{j-1},x_{j+1},\dots,x_k) \in \n(x_j)$.
Therefore, we obtain the bound
\eeq{ \label{first_attack}
&\bigg|\sum_{\ell \in \N} \sum_{\vc{x} \in J_6(\ell)} W(\vc{x}) \prod_{i = 1}^k g(\ell,x_i)\bigg|\\
&\leq \|W\|_\infty \sum_{\ell \in \N} \sum_{j = 1}^{k} \sum_{\vc{x} \in J_6(\ell)\, :\, (\ell,x_{j}) \notin \phi(C)}\prod_{i = 1}^{k} g(\ell,x_i) \\
&\leq \|W\|_\infty \sum_{\ell \in \N} \sum_{j = 1}^k \sum_{x_j \in \Z^d\, :\, (\ell,x_j) \notin \phi(C)} \sum_{\vc{y} \in \n(x_j)} g(\ell,x_j) \prod_{i=1}^{k-1} g(\ell,y_i) \\
&= k\|W\|_\infty \sum_{\ell \in \N} \sum_{x \in \Z^d\, :\, (\ell,x) \notin \phi(C)} \sum_{\vc{y} \in \n(x)} g(\ell,x) \prod_{i = 1}^{k-1} g(\ell,y_i).
}
Now notice that
\eq{
\vc{y} = (y_1,y_2,\dots,y_{k-1}) \in \n(x) \quad \iff \quad
(x,y_2,\dots,y_{k-1}) \in \n(y_1),
}
and that
\eq{
(\ell,x) \notin \phi(C) \quad \implies \quad g(\ell,x) \leq \sqrt{\sum_{u \notin \phi(C)} g(u)^2} \leq \sqrt{d_\phi(f,g)} < \sqrt{\delta}.
}
Therefore, we can rewrite \eqref{first_attack} as
\eeq{
&\bigg|\sum_{\ell \in \N} \sum_{\vc{x} \in J_6(\ell)} W(\vc{x}) \prod_{i = 1}^k g(\ell,x_i)\bigg|\\
&\leq k\|W\|_\infty \sum_{\ell \in \N} \sum_{y \in \Z^d} \sum_{\vc{x} \in \n(y)\, :\, (\ell,x_1) \notin \phi(C)} g(\ell,y) \prod_{i = 1}^{k-1} g(\ell,x_i) \\
&\leq k\|W\|_\infty \sum_{\ell \in \N} \sum_{y \in \Z^d}g(\ell,y) \sum_{\vc{x} \in \n(y)} \sqrt{\delta}\prod_{i = 2}^{k-1} g(\ell,x_i) \\
&\leq k\|W\|_\infty (2K)^{(k-1)d}\sqrt{\delta} \sum_{\ell \in \N} \sum_{y \in \Z^d} g(\ell,y) \\
&\leq k\|W\|_\infty (2K)^{(k-1)d}\sqrt{\delta} < \frac{\eps}{8}, \label{bound_3_3}
}
where \eqref{delta_condition_4} yields the last inequality.
In light of \eqref{bound_3_rewrite}, \eqref{bound_3_1}--\eqref{bound_3_3} now yield
\begin{align}
D_3 &=\bigg|\sum_{\ell \in \N} \sum_{\vc{x} \in J_2(\ell)} W(\vc{x}) \prod_{i = 1}^k g(\ell,x_i)\bigg| \nonumber \\
&\leq \bigg|\sum_{\ell \in \N} \sum_{\vc{x} \in J_3(\ell)} W(\vc{x}) \prod_{i = 1}^k g(\ell,x_i)\bigg| \nonumber \\
&\qquad + \bigg|\sum_{\ell \in \N} \sum_{\vc{x} \in J_5(\ell)} W(\vc{x}) \prod_{i = 1}^k g(\ell,x_i)\bigg|\nonumber \\
&\qquad + \bigg|\sum_{\ell \in \N} \sum_{\vc{x} \in J_6(\ell)} W(\vc{x}) \prod_{i = 1}^k g(\ell,x_i)\bigg| \nonumber \\
&< \frac{\eps}{8} + \frac{3\eps}{8} + \frac{\eps}{8} = \frac{5\eps}{8}. \label{bound_3}
\end{align}
Then using \eqref{bound_1}, \eqref{bound_2}, and \eqref{bound_3} in \eqref{triple_sum}, we find
\eq{
|I(W,f)-I(W,g)| < \frac{\eps}{4} + \frac{\eps}{8} + \frac{5\eps}{8} = \eps,
}
as desired.
\end{proof}

\bibliography{directed_polymers.bib}

\end{document}